\documentclass[11pt]{amsart}
\usepackage[margin=1.5in]{geometry}               
\usepackage[usenames]{color}
\usepackage{enumerate}
\usepackage{amssymb}
\usepackage{amsmath}
\usepackage{amscd}
\usepackage{amsthm}
\usepackage{amsfonts}
\usepackage{graphicx}
\usepackage[all]{xy}
\usepackage{verbatim}
\usepackage{hyperref}
\usepackage{epstopdf}
\usepackage[utf8]{inputenc}
\usepackage{bbm}
\usepackage{esint}
\usepackage{tikz}

\newtheorem{theorem}{Theorem}
\newtheorem{proposition}[theorem]{Proposition}
\newtheorem{lemma}[theorem]{Lemma}
\newtheorem{corollary}[theorem]{Corollary}

\theoremstyle{definition}\newtheorem*{claim}{Claim}

\theoremstyle{definition}\newtheorem{definition}[theorem]{Definition}

\theoremstyle{definition}

\theoremstyle{definition}
\theoremstyle{definition}
\theoremstyle{definition}
\numberwithin{theorem}{section}

\newcommand{\cB}{\mathcal{B}}

\newcommand{\cN}{\mathcal{N}}

\newcommand{\cS}{\mathcal{S}}

\newcommand{\bA}{\mathbb{A}}

\newcommand{\bC}{\mathbb{C}}

\newcommand{\bG}{\mathbb{G}}
\newcommand{\bL}{\mathbb{L}}

\newcommand{\bR}{\mathbb{R}}
\newcommand{\bZ}{\mathbb{Z}}
\newcommand{\bQ}{\mathbb{Q}}
\newcommand{\bF}{\mathbb{F}}

\newcommand{\bN}{\mathbb{N}}

\newcommand{\bH}{\mathbb{H}}
\newcommand{\bT}{\mathbb{T}}
\newcommand{\bS}{\mathbb{S}}

\newcommand{\gog}{\mathfrak{g}}
\newcommand{\goh}{\mathfrak{h}}
\newcommand{\gol}{\mathfrak{l}}

\newcommand{\gor}{\mathfrak{r}}
\newcommand{\gos}{\mathfrak{s}}

\newcommand{\SL}{\operatorname{SL}}

\newcommand{\GL}{\operatorname{GL}}

\newcommand{\Mat}{\operatorname{Mat}}

\newcommand{\lie}{\operatorname{Lie}}
\newcommand\set[1]{\left\{#1\right\}} 
\newcommand\pa[1]{\left[#1\right]}

\newcommand\on[1]{\operatorname{#1}} 
\newcommand\diag[1]{\operatorname{diag}\left(#1\right)}
\newcommand\mat[1]{\pa{\begin{matrix}#1\end{matrix}}} 
\newcommand\smallmat[1]{\pa{\begin{smallmatrix}#1\end{smallmatrix}}}

\newcommand{\acts}{\hspace{-1pt}\mbox{\raisebox{1.3pt}{\text{\huge{.}}}}\hspace{-1pt}}

\newcommand{\la}{\langle}

\newcommand{\ra}{\rangle}

\newcommand{\onto}{\xymatrix{\ar@{>>}[r]&}}
\newcommand{\da}[4]{\xymatrix{#1 \ar@<.5ex>[r]^{#2} \ar@<-.5ex>[r]_{#3} & #4}}

\newif\ifdraft\drafttrue

\newcommand{\SLQ}[1][d]{\operatorname{SL}_{#1}(\bQ_p)}

\newcommand{\SLS}[1][d]{\operatorname{SL}_{#1}(\bR\times\bQ_p)}
\newcommand{\SLSS}[1][d]{\operatorname{SL}_{#1}(\bQ_S)}
\newcommand{\SLZ}[1][d]{\operatorname{SL}_{#1}(\bZ_p)}
\newcommand{\SLZP}[1][d]{\operatorname{SL}_{#1}(\bZ[\tfrac1p])}
\newcommand{\SO}[1][d]{\operatorname{SO}_{#1}}

\newcommand{\Hv}{\bH_{v}}
\newcommand{\HvS}[1][S]{H_{v,{#1}}}
\newcommand{\HV}{\bH_{\Lambda_{v}}}
\newcommand{\HVS}[1][S]{H_{\Lambda_{v},{#1}}}
\newcommand{\Lv}{\bL_{v}}

\newcommand{\LvS}[1][S]{L_{v,{#1}}}

\newcommand{\oLvS}[1][S]{\widetilde{L}_{v,{#1}}}
\newcommand{\oT}{\theta_v}

\newcommand{\lsl}{\mathfrak{sl}_{2}}

\newcommand{\stab}{\operatorname{stab}}
\newcommand{\vol}[1]{\operatorname{vol}\left(#1\right)}

\newcommand{\ad}[1][X]{\operatorname{ad}_{{#1}}}
\newcommand{\Ad}[1][g]{\operatorname{Ad}_{{#1}}}

\newcounter{consta}
\renewcommand{\theconsta}{{\kappa_{\arabic{consta}}}}
\newcounter{constb}[section]

\newcounter{constc}[section]
\renewcommand{\theconstc}{{c_{\arabic{constc}}}}
\newcommand{\consta}{\refstepcounter{consta}\theconsta} 

\newcommand{\constc}{\refstepcounter{constc}\theconstc}

\title{Distribution of Shapes of orthogonal Lattices}
\thanks{The authors acknowledge the support of the SNF Grant  200021-152819. R.\ R.\ also acknowledges the support of the 
ERC Starting Grant DLGAPS 279893}

\author{Manfred Einsiedler}
\address{M.E. \& P.W.\\ D-Math, ETH Z\"urich, R\"amistrasse 101, CH-8092 Z\"urich, Switzerland}
\author{Ren\'e R\"uhr}
\address{R.R.\\ Department of Mathematics, University of Tel Aviv, 69978 Tel-Aviv, Israel}
\author{Philipp Wirth}
\date{\today}

\begin{document}
\maketitle

\setcounter{tocdepth}{1}

\tableofcontents


\section{Introduction}

For an integer $d \geq 3$, let $\bS^{d-1} = \left\{x \in \bR^{d} : \left\|x\right\| = 1\right\} \subset \bR^{d}$ be the euclidean unit sphere of dimension~$d-1$
and denote by $\widehat{\bZ}^{d}$ the set of primitive vectors in $\bZ^{d}$. 
Let us start by recalling Linnik's problem concerning the equidistribution of the finite set
\begin{equation}\label{linnik-eq1}
 \frac{1}{\sqrt{D}}\left(\widehat{\bZ}^{d}\cap\sqrt{D}\bS^{d-1}\right)\subset\bS^{d-1}
\end{equation}
as~$D\to\infty$ (and assuming that this set is nonempty). The hardest case of this problem
concerns the case~$d=3$ and was resolved
by Duke~\cite{Duke88} (building on a breakthrough of Iwaniec~\cite{Iwaniec}). 

Following  Maass \cite{maass1956, maass1959} and W.~Schmidt \cite{WSchmidt-sublattices} (see
also \cite{Marklof-Frobenius} and \cite[Conjecture 1.5]{EMSS})
we are interested in the following refinement of Linnik's problem. 
Fix a positive integer $D$.  For any $v \in \widehat{\bZ}^{d}$ with $\left\|v\right\|^{2} = D$ we 
introduce the orthogonal lattice of $v$,
\[
\Lambda_{v} = v^{\perp} \cap \bZ^{d}
\]
and study the joint equidistribution of the vector~$\frac1{\sqrt{D}}v$ 
belonging to the set in~\eqref{linnik-eq1} and the `shape of the lattice'~$\Lambda_v$. 
More precisely we let $\on{SL}_{d}(\bR)$ act from the right on 
 $\bR^{d}$ (considered as the space of row vectors) and choose a rotation $k_{v} \in \SO(\bR)$ such 
that $\frac{1}{\sqrt{D}}vk_v$ equals the last standard basis vector of~$\bR^d$ 
and hence $\Lambda_{v}k_{v}\subset \bR^{d-1}\times \left\{0\right\}$. 
Moreover, we let
\[
 a_{v} = \on{diag}\Bigl(D^{-{1}/{2(d-1)}},\dots,D^{-{1}/{2(d-1)}},D^{1/2}\Bigr) \in \SL_{d}(\bR)
\]
and note that it rescales the lattice $\Lambda_{v}k_{v}\subset\bR^{d-1}$ of covolume $\sqrt{D}$
by a homothety to covolume $1$ without changing its shape (see Section \ref{geometry}). Denote by $[\Lambda_{v}]$ the right $\on{SO}_{d-1}(\mathbb{R})$-orbit of $\Lambda_{v}k_{v}a_{v}$. We identify the space of unimodular lattices in $\bR^{d-1}$ with $\SL_{d-1}(\bZ)\backslash \SL_{d-1}(\bR)$ so that $\left[\Lambda_{v}\right]$ is an element of
\[
\mathcal{X}_{d-1} = \SL_{d-1}(\bZ) \backslash \SL_{d-1}(\bR) / \SO[d-1](\bR).
\]
Finally we note $k_{v}\on{SO}_{d-1}(\bR)$ is uniquely defined by the above requirement on $k_v$
and hence $[\Lambda_{v}]$ is canonically attached to the vector $v$. We refer to $[\Lambda_{v}]$
as the \emph{shape} of the lattice associated to $v$.

We let $m_{\bS^{d-1}}$ denote the normalised rotation invariant Lebesgue probability measure on the sphere and let
 $m_{\mathcal{X}_{d-1}}$ denote the probability measure on the space of 
 shapes of unimodular lattices induced from Haar measure on $\on{SL}_{d-1}(\mathbb{R})$.

Aka, Shapira and the first name author proved in \cite[Thm.~1.2]{AES} and \cite[Thm.~1.2]{AES3}  the following equidistribution of projected integer points jointly with the shapes of their associated lattices. 

\begin{theorem}[AES]\label{mainresultAES}
Let $d \geq 3$ and for any positive integer $D$, define
\[
\mathcal{Q}_{D} = \left\{\left(\tfrac{v}{\left\|v\right\|},\left[\Lambda_{v}\right]\right)  :  v\in \widehat{\bZ}^{d},\left\|v\right\|^{2} = D\right\}\subset \bS^{d-1}\times\mathcal{X}_{d-1}.
\]
Then the normalized counting measure~$m_{\mathcal{Q}_D}$ on $\mathcal{Q}_{D}$ converges to $m_{\bS^{d-1}}\times m_{\mathcal{X}_{d-1}}$ in the weak$^*$ topology as $D \rightarrow \infty$, provided that the set $\mathcal{Q}_{D}$ is non-empty and that in addition one has the following condition on the number~$D$:
\begin{itemize}
\item If $d=3$, $D$ is square free and there are two distinct fixed odd primes $p, q$ such that $-D$ is a square in $(\mathbb{F}_{p})^{\times}$ and in $(\mathbb{F}_{q})^{\times}$.
\item If $d = 4,5$, there exists a fixed odd prime $p$ such that $p \nmid D$.
\end{itemize}
\end{theorem}

In fact, for~$d\geq 4$ a stronger claim is proven in \cite{AES}
(where instead of the shape of the lattice the `grid' $\on{ASL}_{d-1}(\mathbb{Z})g_{v}k_{v}a_{v}\on{SO}_{d-1}(\mathbb{R})$
associated to the vector~$v$ is considered).
The congruence condition for the lower dimensions is an artefact of the proof and should not
be necessary. The purpose of this paper is to remove this condition\footnote{We note that the case~$d=3$ is fundamentally different and cannot be handled
by the methods of this paper.} in the case 
where $d \in \{4,5\}$. Moreover, in these cases we will prove an effective version of Theorem \ref{mainresultAES}.

\begin{theorem}
\label{mainresult}
Let $d=4$ or $d=5$ and for any positive integer $D$, let $m_{\mathcal{Q}_{D}}$ denote the normalized counting measure on $\mathcal{Q}_{D}$. Then there exists an absolute constant $\consta\label{exp:mainresult} > 0$ such that for any $f\in C_c^\infty(\mathbb{S}^{d-1} \times \mathcal{X}_{d-1})$
\[
\left|m_{\mathcal{Q}_{D}}(f) - m_{\mathbb{S}^{d-1}} \times m_{\mathcal{X}_{d-1}}(f)\right| \ll D^{-\ref{exp:mainresult}}\mathcal{S}_{\infty}(f),
\]
provided that the set $\mathcal{Q}_{D}$ is non-empty.
\end{theorem}

 We note that the equidistribution result follows from this as smooth functions are dense in the space of continuous functions, so that together with Theorem~\ref{mainresultAES}
 we obtain the following

\begin{corollary}\label{maincorollary}
Theorem \ref{mainresultAES} holds without any congruence condition on~$D$ for all~$d\geq 4$.
\end{corollary}

Let us add a few remarks. For two quantities~$A,B$ the notation ``$A \ll B$'' stands for $A\leq cB$, where~$c$ is some absolute constant independent of $D$. We will also write $A\asymp B$
for $A\ll B\ll A$ (where two different implicit absolute constants are allowed).
The space $C_c^\infty(\mathbb{S}^{d-1} \times \mathcal{X}_{d-1})$ denotes as usual smooth\footnote{More precisely
we identify functions on the orbifold~$\mathcal{X}_{d-1}$ with~$\on{SO}_{d-1}(\bR)$-invariant functions 
on~$\on{SL}_{d-1}(\bZ)\backslash\on{SL}_{d-1}(\bR)$ and say that a function
is smooth if it is smooth on the manifold~$\on{SL}_{d-1}(\bZ)\backslash\on{SL}_{d-1}(\bR)$.} functions of compact support.
Also $\mathcal{S}_{\infty}$ denotes the Sobolev norm 
\begin{equation}\label{manifold-sobolev}
 \mathcal{S}_{\infty}(f)=\sum_{\mathcal{D}}\|(1+\on{ht}(\cdot))^{d_2}\mathcal{D}f\|^2_2,
\end{equation}
where the sum is taken over partial derivatives of order less than $d_2\ll1$ with respect to a fixed basis of the tangent space, $\on{ht}(x)$ denotes a height function on
the non-compact space~$\mathcal{X}_{d-1}$ (see Definition \ref{height}) and $\|\cdot\|_2$ denotes the $L^2$ norm on $\mathbb{S}^{d-1} \times \mathcal{X}_{d-1}$ with respect to the natural measures.  For a further discussion on Sobolev norms, we refer to  
Section \ref{sobolev}. 

Finally, we note that by a theorem of Legendre, a positive integer can be written as a sum of three squares if and only if it is not of the form $4^{n}(8k+7)$ for some integers $k$ and $n$; and Lagrange proved that every positive integer can be written as a sum of four squares. This implies that for $d=4$, the set $\mathcal{Q}_{D}$ is non-empty if $D$ is not divisible by $8$; and that it is never empty for $d=5$.
Throughout the paper we will assume that $D$ is chosen such that $\mathcal{Q}_D$ is non-empty.


\subsection{Unipotent Dynamics and the Splitting Condition}

The proof of Theorem~\ref{mainresultAES} in the case of~$d\geq 4$ in \cite{AES} uses a $p$-adic analogue
of the Mozes-Shah theorem~\cite{ms} as provided by Gorodnik and Oh in \cite{Gorodnik-Oh}.
Hence it can be seen as a corollary of Ratner's measure classification theorem for 
unipotent flows on $S$-arithmetic
quotients, see \cite{ratner} and \cite{margulis}. However, for these theorems to be useful
one needs to find unipotent flows related to the equidistribution problem. While it is easy
to relate the problem at hand to the dynamics of a semi-simple subgroup, it is not immediate that this semi-simple subgroup is non-compact and so contains unipotent subgroups -- 
via this `splitting condition' the 
condition that $p$ does not divide $D$ enters into the proof of Theorem~\ref{mainresultAES}. 
We also refer to \cite{Ellenberg-Venkatesh} where the same method
has been applied before and the same splitting condition in low dimensions
appears.

In \cite{EMV} Margulis, Venkatesh,
and the first named author made certain cases of the Mozes-Shah 
theorem on real homogeneous spaces effective. In \cite{EMMV}
Margulis, Mohammadi, Venkatesh, and the first named author used 
similar arguments to prove an effective equidistribution on an adelic
quotient. Relying on Prasad's volume \cite{Prasad}
formula this effective theorem does not require the splitting condition.
However, the main result of \cite{EMMV} is restricted to maximal semi-simple subgroups. For our application
this means that the equidistribution on~$\bS^{d-1}$ or the equidistribution on~$\mathcal{X}_{d-1}$
can be obtained directly from \cite{EMMV} without the congruence condition on~$D$
but not the joint equidistribution
 on~$\bS^{d-1}\times\mathcal{X}_{d-1}$.  Our main argument applies the same technique but
 by limiting to the cases $d=4,5$ we can reduce the input from algebraic geometry and
Bruhat-Tits theory to its minimum and at the same time avoid the use of Prasad's volume formula. 
In fact we will not assume Bruhat-Tits theory and prove what is needed along the way. 
As one motivation of the current paper was also to provide an introduction to the 
adelic equidistribution theorem in \cite{EMMV} we hope that this helps some readers.


\section{Reformulation within Homogeneous Dynamics}
\label{ReformulationTowardsHomogeneousDynamics}

In this section we are going to recall the argument from \cite{AES}
that relates  Theorem~\ref{mainresultAES} to an 
equidistribution result on a $p$-adic cover of a homomgeneous space. By doing so, we introduce some notation that is used throughout the paper. We adopt most of the notation from \cite{AES} but for the fact that lattices appear to the right in loc cit.\ 
and we confirm in this point with \cite{EMMV} by having the lattices on the left. Another change to \cite{AES} concerns   that we only prove equidistribution on the space of shapes of lattices instead of the space of grids
(which leads to some changes in the notation). We will also prove in this section a few preliminary results
and deductions.

We fix throughout the paper some integer $D>1$
and some vector $v \in \widehat{\bZ}^{d}$ of norm $\|v\|=\sqrt{D}$. 
We will refer to $D$ as the discriminant as it equals the discriminant of the
integer quadratic form obtained by restricting $\|\cdot\|^2$ to $\Lambda_v$, see Section \ref{geometry} below.
Depending on~$D$ we will also fix some prime number $p$ as in 
Section \ref{choosing_prime}. 


\subsection{The covolume}\label{geometry}

We start by recalling from \cite{AES} that $\Lambda_v=v^{\perp} \cap \mathbb{Z}^{d}$ is a lattice of 
covolume $\sqrt{D}=\left\|v\right\|$ in $v^{\perp}$. In fact, since $v\in\bZ^d$ is primitive,
 there exists a vector $w\in\bZ^d$ with $\langle v,w\rangle=1$, which implies with $\langle v,\Lambda_v\rangle=0$ that $\bZ^d=\Lambda_v+\bZ w$ and that $w$ 
has distance $D^{-\frac12}$ from $v^\perp$. As the covolume of $\bZ^d$ is $1$ we see that $\Lambda_v$ has covolume $\sqrt{D}$
in the hyperplane~$v^\perp$. In particular, this shows that~$\Lambda_v k_v a_v$
has covolume 1 as claimed in the introduction.


\subsection{Choosing the prime}
\label{choosing_prime}
Throughout the paper $p$ always denotes a prime number, which we will now find
to satisfy  the congruence condition of Theorem~\ref{mainresultAES} in the case of $d=4,5$.
The crucial difference to Theorem~\ref{mainresultAES} in \cite{AES}  is that in Theorem \ref{mainresultAES} the prime is fixed
and $D$ is assumed to satisfy $p\nmid D$ while here we allow $D$ to vary freely (always 
assuming that $\mathcal{Q}_D$ is non-empty)
and choose\footnote{This means that to some extend the ambient space
or at least the dynamics considered varies with $D$, which is the reason why the ineffective measure classification results for unipotent flows and their corollaries are insufficient.} $p$ according to $D$ by using the prime number theorem as follows.

\begin{proposition}\label{prime_existence}
For any $M>0$ there exists $D_M$ such that  for all positive integers $D\geq D_M$ 
there exists a prime $M\leq p \ll \log(D)^{2}$ satisfying the following conditions:
\begin{itemize}
\item $p \equiv 1 \pmod 4$,
\item $p \nmid D$.
\end{itemize}
\end{proposition}
\begin{proof}
Define $\pi_M\left(x;4,1\right)$ to be the number of primes $M\leq p \leq x$ with $p \equiv 1 \pmod 4$. As removing the primes below $M$ is irrelevant for the asymptotics 
the prime number theorem for primes in arithmetic progressions gives
\[
\lim_{x \rightarrow \infty}\frac{\pi_M\left(x;4,1\right)\log(x)}{x} = \frac{1}{2}.
\]
In particular we have
\[
\pi_M(x;4,1) \geq  \frac{x}{3\log(x)}
\]
for all sufficiently large $x$ (depending only on $M$).

 Let $x$ satisfy this estimate and set $k= \pi_M\left(x;4,1\right)$. Suppose that $\left\{p_{ 1},\dots,p_{k}\right\}$ are all the primes satisfying $M\leq p \leq x$ and $p \equiv 1 \pmod 4$ and that all these primes divide $D$. Since these are pairwise different primes, $D$ is also divisible by the product $p_{ 1}\dots p_{k}$. But this implies that
\begin{equation*}\label{eq:to_contradict}
D \geq p_{ 1}\dots p_{k} > k ! =  \pi_M\left(x;4,1\right) !
\end{equation*}
Using the fact that $k! \geq e^{k}$ once $x$ (and hence $k$) is
sufficiently large and also that $k\geq \frac{x}{3\log(x)} \gg \sqrt{x}$
we have  that
$$(\log D)^2>\log{\bigl(\left(\pi_M\left(x;4,1\right) \right)!\bigr)}^2\gg x.$$
Thus if all primes congruent $1 \pmod 4$ between $M$ and $x$ were to divide $D$ and $x$ is sufficiently large, then
$(\log D)^{2} > \ref{const:smallprime2} x$ for some constant $\constc\label{const:smallprime2} > 0$.

We now set $x=\frac{1}{\ref{const:smallprime2}} \log(D)^{2}$, and assume that $D$ is sufficiently big so that $x$ satisfies all of the above estimates. It follows that there exists a prime $M\leq p < x$ which does not divide $D$ and such that $p \equiv 1 \pmod 4$.
\end{proof}


\subsection{Ambient spaces}

For an algebraic group $\mathbb{G}$, we write $G_{\infty} = \mathbb{G}(\mathbb{R})$ 
 and $G_{p} = \mathbb{G}(\mathbb{Q}_{p})$, and for $S = \{p,\infty\}$ we set $G_{S} = G_{\infty} \times G_{p}$. Moreover, let $G_{p}^{+}$ be the finite index subgroup of $G_{p}$ 
 which is generated by unipotent elements (we will see that $\mathbb{G}$ and $p$ have this property) and put $G_{S}^{+} = G_{\infty} \times G_{p}^{+}$. 
 As for the real place, we note that 
 for the algebraic groups we consider (namely $\on{SO}_{d}$ and $\on{SL}_d$) the group of all $\bR$-points
 will be connected in the Hausdorff topology. If a group $H$ embeds in both groups 
 $G_{1}$ and  $G_{2}$, we write $\Delta_{H} = \left\{(h,h):h \in H\right\}$
 for the diagonally embedded subgroup. 
 
 From now on we let either $d=4$ or $d=5$ and introduce the   groups
\[
\mathbb{G}_{1} = \on{SO}_{d},~ \mathbb{G}_{2} = \on{SL}_{d-1} ~ \mbox{and} ~ \mathbb{G}_\text{joint} = \mathbb{G}_{1}\times \mathbb{G}_{2}.
\]
We refer to these groups as the ambient groups as these define the homogeneous spaces on which
we study dynamical and equidistribution properties. Implicit constants are allowed to depend on 
these algebraic groups (but not on the chosen prime $p$).
Let $\Gamma = \mathbb{G}_\text{joint}(\mathbb{Z}[\tfrac{1}{p}])$ be diagonally embedded in $G_{\text{joint},S}$. Recall that by a theorem of Borel and Harish-Chandra (see e.g.~\cite[Sect.~I.3.2]{margulisbook}), $\Gamma$ is a lattice and define the ambient spaces
\[
\mathcal{Y}_{i} = \Gamma G_{i,S}  \mbox{ and } \mathcal{Y}_\text{joint} =
\Gamma G_{\text{joint},S}\cong \mathcal{Y}_{1} \times \mathcal{Y}_{2},
\]
for $i=1,2$ as well as
\[
\mathcal{Y}_{\infty} = \mathbb{G}_{\text{joint}}(\mathbb{Z})\backslash G_{\text{joint},\infty}, ~\mathcal{Y}^+_{i}=\Gamma G^+_{i,S} \mbox{ and } \mathcal{Y}_\text{joint}^+ = \mathcal{Y}^+_{1} \times \mathcal{Y}_{2}
\]
for $i=1,2$. 
Note that $G_{2,p}^+=G_{2,p}$ and so $\mathcal{Y}^+_{2}=\mathcal{Y}_{2}$. 
Since the orbit $\Gamma\bG(\bQ_S)$ is isomorphic to $\bG(\bZ[\tfrac{1}{p}])\backslash\bG(\bQ_S)$, we keep writing $\Gamma$ for $\bG(\bZ[\tfrac{1}{p}])$ for any~$\bG\in\{\bG_1,\bG_2,\bG_{\text{joint}}\}$. For $x\in \Gamma\backslash\bG(\bQ_S)$ and $g\in\bG(\bQ_S)$ we denote the natural action by $g\acts x=xg^{-1}$.


\subsection{The orbits of the stabilizing subgroups}\label{orbitsofsubgroups}

Recall from Section \ref{geometry} that $\Lambda_v= v^\perp\cap\bZ^d=\bZ w_1+\cdots+ \bZ w_{d-1}$
has covolume $\sqrt{D}$ and~$\bZ^d=\bZ w_1+\cdots+\bZ w_{d-1}+\bZ w$. 
We define the matrix $g_v$ with rows $w_1,\ldots,w_{d-1},w$  
such that $\bZ^{d-1}g_{v} = \Lambda_{v}$.
We may also suppose that $\det g_v>0$ so that $g_v\in \SL_d(\bZ)$ 
and note that $g_{v}k_v$ sends $\bR^{d-1}$ to itself. 
By the covolume calculation in Section \ref{geometry}
$g_vk_va_v$ belongs to $\on{ASL}_{d-1}(\bR)$ where
\[
 \on{ASL}_{d-1}=\left\{
 \begin{pmatrix}
	 \theta &0\\ w'&1
	 \end{pmatrix} :
	 \theta\in\on{SL}_{d-1}
 \right\}.
\]
If $\theta_{v}$ denotes the upper left~$d-1$ by~$d-1$ block matrix of $g_{v}k_{v}a_{v}$, then the shape of $\Lambda_{v}$ is given by 
\[
[\Lambda_{v}] = \on{SL}_{d-1}(\mathbb{Z})\theta_{v}\on{SO}_{d-1}(\mathbb{R}).
\]
We also define the stabilizer group
\[
\bH_{v} = \left\{g \in \operatorname{SO}_{d} : vg = v\right\} = \operatorname{Stab}_{\mathbb{G}_{1}}(v)
~ \mbox{and} ~ \bH_{\Lambda_v} = \overline{g_v\bH_{v}g_v^{-1}},
\]
where the latter is the projection of $g_v\bH_{v}g_v^{-1}<\operatorname{ASL}_{d-1}$ to $\operatorname{SL}_{d-1}$. We note that $\bH_v$ is semi-simple and so intersects
the unipotent kernel of this projection trivially, which shows that $\bH_v$ and $\bH_{\Lambda_v}$
are isomorphic algebraic groups.
Finally we define their diagonal embedding and its projection
\[
\widetilde{\mathbb{L}}_{v} = (e,g_{v})\Delta_{\mathbb{H}_{v}}(e,g_{v}^{-1}) \leq \mathbb{G}_{1} \times \operatorname{ASL}_{d-1}
~ \mbox{and} ~ {\mathbb{L}}_{v}<\mathbb{G}_\text{joint}.
\]
 Inside the space $\mathcal{Y}^+_\text{joint}$, we consider the joint orbit
 \[
 \Gamma {L}_{v,S}^{+}\left(k_{v},e,\theta_{v},e\right) = \Gamma \left(k_{v},e,\theta_{v},e\right)\Delta_{\on{SO}_{d-1}(\mathbb{R}) \times \mathbb{H}_{v}^+(\mathbb{Q}_{p})},
 \]
where $e$ is the identity element in the corresponding group. Let
\[
\mu_{v,S} = m_{\Gamma {L}_{v,S}^{+}\left(k_{v},e,\theta_{v},e\right)}
\]
be the Haar measure on this orbit and let $m_{\mathcal{Y}^+_\text{joint}}$ be the Haar measure on $\mathcal{Y}^+_\text{joint}$ both normalized to be probabilty measures. If $\pi_i:\mathcal{Y}_\text{joint}\to\mathcal{Y}_i$ denotes the natural projection 
for~$i=1,2$ then ${(\pi_i)}_*\mu_{v,S}$ is the probability orbit measure on $\Gamma H^{+}_{v,S} (k_v,e)$ if $i=1$ respectively $\Gamma \HVS^+ (\theta_v,e)$ if $i=2$.

As we will show in the bulk of the paper these orbits equidistribute in the corresponding ambient spaces with respect to the Haar measures $m_{\mathcal{Y}^+_{\operatorname{joint}}}$, resp.\ ${\pi_i}_*m_{\mathcal{Y}^+_\text{joint}}=m_{\mathcal{Y}^+_i}$ for $i=1,2$.
For the connection to Theorem~\ref{mainresult} the following version is most useful.

\begin{theorem}
\label{t:FullDynamics}
There exist some absolute constant
$\consta\label{exp:FullDynamicsJoint} > 0$ and $d_2 \geq1$ such that for any large enough~$D$
there exists a prime number $p$ such that for all $v \in \widehat{\mathbb{Z}}^{d}$ with $\left\|v\right\|^{2} = D$ and for any $f\in C_c^\infty(\mathcal{Y}_{\operatorname{joint}})$
\[
\left|\mu^{\on{Full}}_{v,S}(f) - m_{\mathcal{Y}_{\operatorname{joint}}}(f)\right| \ll D^{-\ref{exp:FullDynamicsJoint}}\mathcal{S}_{d_2}(f)
\]
where $ \mu^{\on{Full}}_{v,S} = m_{\Gamma {L}_{v,S}\left(k_{v},e,\theta_{v},e\right)}$.
The analogous statement holds for functions
on $\mathcal{Y}_i$ for $i=1,2$.
\end{theorem}

The notion of smoothness of a function on $\mathcal{Y}_\text{joint}$ and properties of the $S$-adic Sobolev norm $\mathcal{S}_{d_2}$ of degree $d_2$ will be discussed in Section \ref{sobolev}. 
We note however that using e.g.\ the homogeneous space~$\on{SO}_d(\bR)\times\on{SL}_{d-1}(\bZ)\backslash\on{SL}_{d-1}(\bR)$ and a fixed basis of the Lie algebra of $\on{SO}_d\times\on{SL}_{d-1}$
it is easy to define a Sobolev norm $\mathcal{S}_\infty$ on 
$C_c^\infty( \mathbb{S}^{d-1}\times\mathcal{X}_{d-1})$ 
by the formula \eqref{manifold-sobolev} 
(see also the discussion right 
after \eqref{abunchofarrows} below).
The remainder of the section is devoted to the proof
that Theorem \ref{t:FullDynamics} implies
Theorem~\ref{mainresult} (which will follow \cite{AES} closely).


\subsection{Principal genus and Hecke friends}
\label{principalgenus}
First, note that $\on{Stab}_{\mathbb{G}_{1}}(e_{d})(\mathbb{R}) = k_{v}^{-1}\mathbb{H}_{v}(\mathbb{R})k_{v}$ with $k_{v}$ defined as in the introduction. We refer to this group as the standard embedding of $\on{SO}_{d-1}(\mathbb{R})$ in $\on{SO}_{d}(\mathbb{R})$ and call it $H_{\infty}$. We will identify the sphere $\mathbb{S}^{d-1}$ with $G_{1,\infty} / H_{\infty}$ via the right action of $G_{1,\infty}$ on $\mathbb{S}^{d-1}$. We can also embed $H_\infty$ into $G_{2,\infty}$ in an obvious way. We define $\mathbf{S}^{d-1} = \mathbb{G}_{1}(\mathbb{Z}) \backslash \bS^{d-1}$ by identifying points on the sphere on the same $\mathbb{G}_1(\bZ)$-orbit and let $\rho$ be the projection from $\mathcal{Y}_{\infty}\to \mathbf{S}^{d-1}\times\mathcal{X}_{d-1}$ by dividing by $H_\infty\times H_\infty$ from the right.
Let $K = \bG_\text{joint}(\mathbb{Z}_{p})$ and define the clopen orbit $\mathcal{U} = \Gamma (G_{\text{joint},\infty} \times K)\cap \mathcal{Y}_{\text{joint}}$. The projection $\pi$ from $\Gamma (G_{\text{joint},\infty} \times K)$ to $\mathcal{Y}_{\infty}$ is defined by dividing from the right by $\left\{e\right\}\times K$. Finally, there is also a projection to the first factor $\pi_{1} : \mathcal{Y}_{\text{joint}} \rightarrow \mathcal{Y}_{1}$. Summarising some of the notation we have
\begin{equation}\label{abunchofarrows}
\mathcal{Y}_{1}\xleftarrow{\pi_1}
\mathcal{Y}_{\text{joint}}\xleftarrow{\imath} 
\Gamma (G_{\text{joint},\infty} \times K)\cap \mathcal{Y}_{\text{joint}} \xrightarrow{\pi} 
\mathcal{Y}_{\infty}\xrightarrow{\rho} \mathbf{S}^{d-1}\times\mathcal{X}_{d-1}.
\end{equation}
Here~$\imath$ is the inclusion map of the clopen orbit $\mathcal{U}\subset \mathcal{Y}_{\text{joint}}$.
Using this inclusion we may think of every function $f$ on
$\mathbf{S}^{d-1}\times\mathcal{X}_{d-1}$ as a function on~$\mathcal{Y}_{\text{joint}}$: indeed precomposing $f$
with $\rho\circ\pi$ we obtain a function on $\mathcal{U}$, 
which we
may extend to $\mathcal{Y}_{\text{joint}}$ by setting it equal to zero
on the complement $\mathcal{Y}_{\text{joint}}\setminus\mathcal{U}$. The restriction of the Sobolev norm
$\mathcal{S}_{d_2}$
to the functions in $C_c^\infty(\mathcal{Y}_{\text{joint}})$
agrees with the Sobolev 
norm $\mathcal{S}_\infty$.

By Theorem 5.1 in \cite{rapinchuk}, there exists a finite set $M \subset H_{v,p}$ such that
\begin{equation}\label{pr-finite-M}
\bigcup_{h \in \mathbb{H}_{v}^{+}(\mathbb{Q}_{p})}\Gamma h\mathbb{H}_{v}(\mathbb{Z}_{p}) = \bigsqcup_{h \in M}\Gamma h\mathbb{H}_{v}(\mathbb{Z}_{p}).
\end{equation}

 Note that for each $h \in M$, the double coset $\Gamma h\mathbb{H}_{v}(\mathbb{Z}_{p})$ is either contained in $\pi_{1}(\mathcal{U})$, or it is disjoint from $\pi_{1}(\mathcal{U})$. Set
\[
M_{0} = \left\{h \in M : \Gamma h\mathbb{H}_{v}(\mathbb{Z}_{p}) \subset \pi_{1}(\mathcal{U})\right\}.
\]
By definition, $M_{0} \subset \Gamma \mathbb{G}_{1}(\mathbb{Z}_{p})$ and for $h\in M_{0}$ we can write $h = \gamma_{1}(h)c_{1}(h)$ with $\gamma_{1}(h) \in \Gamma$ and $c_{1}(h) \in \mathbb{G}_{1}(\mathbb{Z}_{p})$. Since $\mathbb{G}_{2}$ is simply connected, a similar statement is true for every element of $
G_{2,p}$. In fact, for any $h \in M_{0}$ we can project $g_{v}hg_{v}^{-1}$
to~$\on{SL}_{d-1}(\bQ_p)$ and then write the image as $\gamma_{2}(h)c_{2}(h)$ with $\gamma_{2}(h) \in \Gamma$ and $c_{2}(h) \in \mathbb{G}_{2}(\mathbb{Z}_{p})$.

We define  the shorthand $\widehat{\mathbf{Z}}^{d} = \mathbb{G}_{1}(\mathbb{Z}) \backslash \widehat{\bZ}^{d}$ and write $\mathbf{v}$ for the $\mathbb{G}_{1}(\mathbb{Z})$-orbit of a vector $v \in \widehat{\mathbb{Z}}^{d}$. Notice that $[\Lambda_{v}]$ only depends on $\mathbf{v}$ and so we also denote it by $[\Lambda_{\mathbf{v}}]$. Clearly, the projection $\widehat{\bZ}^{d}\rightarrow \bS^{d-1},~ v \mapsto \tfrac{v}{\left\|v\right\|}$
descends to a projection $\widehat{\mathbf{Z}}^{d} \rightarrow \mathbf{S}^{d-1},~\mathbf{v} \mapsto \tfrac{\mathbf{v}}{\left\|\mathbf{v}\right\|}$. Therefore, the double coset
\[
\mathbb{G}_\text{joint}(\mathbb{Z})(k_{v},\theta_{v})H_{\infty} \times H_{\infty}
\]
represents the pair
\[
\left(\tfrac{\mathbf{v}}{\left\|\mathbf{v}\right\|},\left[\Lambda_{\mathbf{v}}\right]\right) \in \mathbf{S}^{d-1} \times\mathcal{X}_{d-1}.
\]

We define a relation $\sim$ on $\left\{v \in \widehat{\bZ}^{d} : \left\|v\right\|^{2} = D\right\}$ in the following way: $v \sim w$ (is a Hecke friend) if and only if there exist $\gamma \in \bG_1(\bZ[\frac{1}{p}])$ and $b \in \bG_1(\bZ_{p})$ such that $w = v\gamma = vb$ (and so $b\gamma^{-1} \in H_{v,p}$).
\begin{lemma}[{\cite[Section 5]{AES}}]
\label{equivalencerelation}
The relation $\sim$ is an equivalence relation and it descends to an equivalence relation on $\left\{\mathbf{v} \in \widehat{\mathbf{Z}}^{d} : \left\|\mathbf{v}\right\|^{2} = D\right\}$.
\end{lemma}
For $v \in \widehat{\mathbb{Z}}^{d}$ we define
\[
P_{v} = \left\{\mathbf{w} : \mathbf{w} \sim \mathbf{v}\right\}, \  R_{v} = 
\left\{\left(\frac{\mathbf{w}}{\left\|\mathbf{w}\right\|},\left[\Lambda_{\mathbf{w}}\right]\right) : \mathbf{w} \in P_{v}\right\}
\  \mbox{and}\  \mathbf{Q}_D=\bigcup_{\|v\|=D} R_v.
\]
\begin{proposition}[{\cite[Proposition 6.2]{AES}}]
\label{SupportOfNu}
For $h \in M_{0}$, define $\varphi(h) = \mathbb{G}_{1}(\mathbb{Z})\gamma_{1}(h)k_{v}H_{\infty}$. Then $\varphi$ is a bijection from $M_{0}$ to $\left\{\mathbb{G}_{1}(\mathbb{Z})k_{u}H_{\infty} : \mathbf{u} \in P_{v}\right\}$. For any $h\in M_0$
we have that $\varphi(h)$ corresponds to $\mathbf{u} = \mathbf{v}\gamma_{1}(h)^{-1}$
and we may identify the double coset 
$ \mathbb{G}_{2}(\mathbb{Z})\gamma_{2}(h)\theta_{v}H_{\infty}$ 
with the shape $[\Lambda_{\mathbf{u}}]$. 
\end{proposition}


\subsection{Proof of Theorem \ref{mainresult}}
\label{proofofmainresult}
We first introduce the relevant  probability measures and discuss their relation to each other. 
For a measure $\nu$ we denote by $\nu|_A$ the normalized restriction 
defined by $\nu|_A(B)=\frac{\nu(A\cap B)}{\nu(B)}$ for any measurable $B$.
Let $m_{\mathbf{S}^{d-1}\times \mathcal{X}_{d-1}}$ be the natural probability measure on $\mathbf{S}^{d-1}\times \mathcal{X}_{d-1}$. Further, let $m_{\mathbf{Q}_D}$ denote
the pushforward of the normalized counting measure $m_{\mathcal{Q}_D}$ on~$\mathcal{Q}_D$ obtained by taking the quotient by $\bG_1(\bZ)$, which makes $m_{\mathbf{Q}_D}$ a weighted counting measure on $\mathbf{Q}_D$. Moreover, note that $(\rho\circ\pi)_*m_{\mathcal{Y}_\text{joint}}|_\mathcal{U}=m_{\mathbf{S}^{d-1}\times \mathcal{X}_{d-1}}$ and define 
$$\mu_v=(\rho\circ\pi)_*\mu_{v,S}^{\on{Full}}|_\mathcal{U}\quad\mbox{and} \quad\nu_v=m_{\mathbf{Q}_D}|_{R_v}$$
which are both measures on $\mathbf{S}^{d-1}\times \mathcal{X}_{d-1}$. By Proposition~6.1 in \cite{AES}, $\pi_*\mu_{v,S}^{\on{Full}}|_\mathcal{U}$ is a probability measure on
$$\bigsqcup_{h \in M_{0}}\mathbb{G}_{\text{joint}}(\mathbb{Z})(\gamma_{1}(h)k_{v},\gamma_{2}(h)\theta_{v})H_{\infty} \times H_{\infty}$$
so that by Proposition~\ref{SupportOfNu} both $\mu_v$ and $\nu_v$ have support in $R_v$.

We say that two probability measures $\mu$ and $\nu$ on $\mathcal{Y}_\text{joint}$ are $D^{-\kappa}$ close if they satisfy $|\mu(f)-\nu(f)|\ll D^{-\kappa}\cS_{d_2}(f)$ for all $f\in C_c^\infty(\mathcal{Y}_\text{joint})$. This notion also has a natural extension to probability measures on $\mathbf{S}^{d-1}\times \mathcal{X}_{d-1}$ and other related spaces. 
We note that the Sobolev norm $\cS_{d_2}$ (as defined in Section~\ref{sobolevchapter}) restricted to function in $C_c^\infty(\mathbf{S}^{d-1}\times \mathcal{X}_{d-1})$ agrees
with $\mathcal{S}_{\infty}$, which is a  real Sobolev norm modified by the weight function $(1+\on{ht}(x))^{d_2}$ 
 to ensure that the Sobolev embedding theorem holds
 in a convenient form on the non-compact space $\mathcal{X}_{d-1}$.

Theorem~\ref{t:FullDynamics} implies that $\mu_v$ is $D^{-\ref{exp:FullDynamicsJoint}}$ close to $m_{\mathbf{S}^{d-1}\times \mathcal{X}_{d-1}}$ and we will show in Lemma \ref{muandnu} below that $\mu_v$ is $D^{-\ref{exp:muandnu}}$ close to $\nu_v$ for some absolute constant $\ref{exp:muandnu}>0$. Since this is true for any equivalence class $R_v$ it follows also that $m_{\mathbf{Q}_D}$ is $D^{-\min(\ref{exp:FullDynamicsJoint},\ref{exp:muandnu})}$ close to $m_{\mathbf{S}^{d-1}\times \mathcal{X}_{d-1}}$. Finally, consider the average $\widetilde{f}=\tfrac{1}{\bG_1(\bZ)}\sum_{\gamma\in\bG_1(\bZ)}\gamma\acts f$ for $f\in C_c^\infty(\mathbb{S}^{d-1}\times \mathcal{X}_{d-1})$ and note that both $m_{\bS^{d-1}}\times m_{\mathcal{X}_{d-1}}$ and $m_{\mathcal{Q}_D}$ are $\bG_1(\bZ)$-invariant. Theorem~\ref{mainresult} therefore follows from 
\[
 |m_{\mathcal{Q}_v}(f)-(m_{\bS^{d-1}}\times m_{\mathcal{X}_{d-1}})(f)|=|m_{\mathbf{Q}_v}(\widetilde{f})-m_{\mathbf{S}^{d-1}}\times m_{\mathcal{X}_{d-1}}(\widetilde{f})|
 \]
and the triangle inequality for $\cS_\infty$ by setting $\ref{exp:mainresult}=\min(\ref{exp:FullDynamicsJoint},\ref{exp:muandnu})$.


\subsection{Relating $\mu_v$ and $\nu_v$}
It remains to verify that $\mu_v$ and $\nu_v$ are $D^{-\ref{exp:muandnu}}$-close.
For $\mathbf{u} \in \mathbf{S}^{d-1}$ let $S(\mathbf{u}) = \left|\operatorname{Stab}_{\mathbb{G}_{1}(\mathbb{Z})}(u)\right|$ for some $u\in\mathbf{u}$ and define $E = \widetilde{E}\times \mathcal{X}_{d-1}$, where
\[
\widetilde{E} = \left\{\mathbf{u} \in \mathbf{S}^{d-1} : S(\mathbf{u}) > 1\right\}.
\]
The following lemma shows that the weights of the measures $\mu_{v}$ and $\nu_{v}$ are constant on the complement of $E$ and uniformly bounded on $E$.
\begin{lemma}[Lemma 5.3 \cite{AES}]
\label{BoundedWeights}
We define $M_{v} = \max_{x \in R_{v}}\mu_{v}(x)$, $N_{v} = \max_{x \in R_{v}}\nu_{v}(x)$ and $a = \left|\mathbb{G}_{1}(\mathbb{Z})\right|$. For every $x \in R_{v}$, we have
\[
\frac{M_{v}}{a} \leq \mu_{v}(x) \leq M_{v} \mbox{ and } \frac{N_{v}}{a} \leq \nu_{v}(x) \leq N_{v}.
\]
Furthermore, equality holds on the right hand side of both inequalities when $x \in R_{v} \setminus E$.
\end{lemma}
We need to replace \cite[Lemma 6.4]{AES} with the an effective version of the statement that $E$ is a null set.
\begin{lemma} 
\label{SmallE}
There exists $\consta\label{exp:SmallE}\label{exp:muandnu} > 0$ such that $$|E\cap R_v|\ll D^{-\ref{exp:SmallE}}|R_v|.$$
\end{lemma}

\begin{proof}
As $\mathbb{G}_{1}(\mathbb{Z})$ consists up to signs of permutations, any fixed point lies in a hyperplane of the form $\{w\in\bR^d: w_i=\pm w_j\}$ for some $1\leq i\neq j\leq d$. Let $F$ denote the ($\mathbb{G}_{1}(\mathbb{Z})$-invariant) union of such planes intersected with $\mathbf{S}^{d-1}$ so that $\widetilde{E}\subset F$. For any $\varepsilon > 0$ there exists a function $f\in C^\infty(\mathbf{S}^{d-1})$ with $\mathbbm{1}_{F} \leq f$ and $m_{\mathbf{S}^{d-1}}(f) \ll\varepsilon$ such that $\mathcal{S}_{\infty}(f)\ll \varepsilon^{-d_2}$. 
Indeed, we may write
\[
F=\bigcup_{\ell=1}^{n} g_\ell\{\|x\|=1: x_d=0\}
\]
for some finite $\mathbb{G}_{1}(\mathbb{Z})$-invariant
list $\{g_{1},\dots ,g_{n}\} \subset \on{SO}_{d}(\mathbb{R})$. Fix some nonnegative function $\chi\in C_c^\infty((-1,1))$ with $\chi(0)=1$ and define
\[
J_\varepsilon(y)= \begin{cases}
\chi(\varepsilon^{-1}y) &\mbox{if } \left|y\right| < \varepsilon\\
0 &\mbox{otherwise}
\end{cases}
\]
and note that $J_{\varepsilon}(x_{d}) \geq \mathbbm{1}_{\left\{x_d=0\right\}}$. Then $f(x)=\sum_{\ell=1}^{n} g_\ell\acts J_{\varepsilon}(x)$ satisfies the requirements.

Precomposing $f$ with the projection from $\mathbf{S}^{d-1}\times\mathcal{X}_{d-1}$ to $\mathbf{S}^{d-1}$
and the projections $\rho\circ\pi$ from \eqref{abunchofarrows}
we may identify $f$ with a smooth function on $\mathcal{Y}_{\text{joint}}$. 
Applying Theorem \ref{t:FullDynamics} we obtain
$$
\mu_{v}( {E}) \leq \mu_{v}(f) \ll \left|\mu_{v}(f) - m_{\mathbf{S}^{d-1}}(f)\right|+\varepsilon \ll D^{-\ref{exp:FullDynamicsJoint}}\mathcal{S}_{\infty}(f) + \varepsilon \ll D^{-\ref{exp:FullDynamicsJoint}}\varepsilon^{-d_{2}} + \varepsilon.$$
Choosing $\varepsilon=D^{-\frac{\ref{exp:FullDynamicsJoint}}{2d_{2}}}$ we get $\mu_{v}({E}) \ll D^{-\ref{exp:SmallE}}$
for $\ref{exp:SmallE}=\frac{\ref{exp:FullDynamicsJoint}}{2d_{2}}$. Using Lemma \ref{BoundedWeights}, we see that
\[
\frac{1}{a}\frac{\left|E \cap R_{v}\right|}{\left|R_{v}\right|} = \frac{\frac{M_{v}}{a}\left|{E} \cap R_{v}\right|}{M_{v}\left|R_{v}\right|} \leq \frac{\mu_{v}({E}\cap R_{v})}{\mu_{v}(R_{v})} = \mu_{v}({E}) \ll D^{-\ref{exp:SmallE}}.
\]
\end{proof}
Combining both lemmata will give the remaining step.
\begin{lemma}
\label{muandnu}
For any $f\in C_c^\infty(\mathbf{S}^{d-1}\times \mathcal{X}_{d-1})$ we have
$$|\mu_v(f)-\nu_v(f)|\ll D^{-\ref{exp:muandnu}}\cS_{\infty}(f).$$
\end{lemma}
\begin{proof}
We start by controlling $M=M_{v} = \max_{x \in R_{v}}\mu_{v}(x)$ with respect to $|R_v|$. 
Applying Lemma \ref{BoundedWeights} we see that
\[
 M|R_v\setminus E| \leq  \mu_v(1)=1,
\]
 which implies that $M\leq(|R_v|-|E\cap R_v|)^{-1}$.
 By Lemma \ref{SmallE} we have $|E\cap R_v|\leq c D^{-\ref{exp:SmallE}}|R_v|$
 for some $c>0$. 
 Note that there exists some constant $c'$ with $(1-cD^{-\ref{exp:SmallE}})^{-1}\leq1+c'D^{-\ref{exp:SmallE}}$
for all sufficiently large $D$. Therefore, we obtain the upper bound
in 
\[
\tfrac1{|R_v|}\leq M\leq \tfrac1{|R_v|}(1+c'D^{-\ref{exp:SmallE}}),
\] 
where the lower bound follows by using the definition of $M$ in Lemma \ref{BoundedWeights}. 

Let $\lambda_v$ denote the normalized counting measure on $R_v$. Then $$|\lambda_v(f)-\mu_v(f)|\leq \left|\sum_{x\not\in E} f(x)(\tfrac1{|R_v|}-M)+\sum_{x\in E}f(x) (\tfrac1{|R_v|}-\mu_n(x))\right|$$
$$\ll \|f\|_\infty\frac{|R_v\cap E^c|}{|R_v|}D^{-\ref{exp:SmallE}}+\|f\|_\infty\frac{|R_v\cap E|}{|R_v|}\ll\cS_{\infty}(f)D^{- \ref{exp:SmallE}}$$
having used Lemma \ref{SmallE} and Lemma \ref{BoundedWeights} once more and the Sobolev embedding theorem (that is, property (S1) of Section \ref{sobolev}) for $\cS_{\infty}$. 

Since Lemma \ref{BoundedWeights} holds for both measures, the same calculation holds with $\mu_v$ replaced by $\nu_v$ so that the lemma follows.
\end{proof}

As explained in Section \ref{proofofmainresult}
this finishes the proof of Theorem \ref{mainresult}
assuming Theorem \ref{t:FullDynamics}.


\section{Further Setup and Equidistribution on a Single Factor}
\label{single_proof}

We recall that we choose $p$ throughout the paper 
depending on $D$
as in Lemma~\ref{choosing_prime} and that $p$ is implicitly
appearing in the definition of our ambient space $\mathcal{Y}_{\text{joint}}$. 

We will start to discuss the dynamical argument in this section. 
For this argument it is far better to work with the orbits
of the subgroups $L_{v,S}^+$ and the corresponding measure $\mu_{v,S}$
on the ambient 
space $\mathcal{Y}^+_{\text{joint}}$ (or the corresponding orbits and measures on the factors $\mathcal{Y}_i^+$ for $i=1,2$). 
In other words we will give a dynamical proof of the following result.

\begin{theorem}
\label{maindynamicalresultsingle}\label{maindynamicalresult}
There exists absolute constants 
$\consta\label{exp:maindynamicalresultsingle}, \consta\label{exp:maindynamicalresult} > 0$ and $d_2,d_2'\geq1$ such that for any $v \in \widehat{\mathbb{Z}}^{d}$ with $\left\|v\right\|^{2} = D$  
and for any $f\in C_c^\infty(\mathcal{Y}^+_i)$
\[
\left|{\pi_i}_*\mu_{v,S}(f) - m_{\mathcal{Y}^+_i}(f)\right| \ll D^{-\ref{exp:maindynamicalresultsingle}}\mathcal{S}_{d_2}(f)
\]
for $i=1,2$ and for any $f\in C_c^\infty(\mathcal{Y}^+_{\operatorname{joint}})$
\[
\left|\mu_{v,S}(f) - m_{\mathcal{Y}^+_{\operatorname{joint}}}(f)\right| \ll D^{-\ref{exp:maindynamicalresult}}\mathcal{S}_{d_2'}(f).
\]
\end{theorem}

Theorem~\ref{maindynamicalresultsingle} implies equidistribution of the full orbit as in Theorem~\ref{t:FullDynamics} after discussing
the properties of the Sobolev norm in Section \ref{sobolevchapter}. 
The remainder of the paper will then be devoted to proving Theorem \ref{maindynamicalresultsingle}.

To prove equidistribution of the orbit
$\Gamma\LvS^+(k_v,e,\oT,e)$ (whose normalized Haar measure
is given by $\mu_{v,S}$)
in the joint space $\mathcal{Y}^+_\text{joint}$ our first step will be to show equidistribution on the factor spaces. In this section, we reduce the first statement of Theorem \ref{maindynamicalresultsingle} to a purely dynamical result (Proposition \ref{almostinvariance}), whose proof
will be completed in Section \ref{dynamics}. As most of the steps will also be used for the joint equidistribution in Section \ref{proofofequidistribution} we will formulate these steps in the necessary generality.

\subsection{Height and invariant metric}
For a matrix (or product of matrices) $g=(g_\infty,g_p)$ in the space $\Mat_{n,m}(\bQ_S)=\Mat_{n,m}(\bR)\times\Mat_{n,m}(\bQ_p)$ we define the height by
$$\|g\|_S=\|g_\infty\|_\infty\|g_p\|_p,$$
where both norms are chosen to be the maximum over the real respectively $p$-adic absolute values among the coefficients of $g_\infty$ respectively $g_p$. Recall that $\|\cdot\|_p$ is bi-$\GL_d(\bZ_p)$-invariant when defined on $\Mat_{d,d}(\bQ_p)$. 

The Lie group $\on{SL}_d(\bR)$ carries a left-invariant metric $d_\infty$ induced from a left-invariant Riemannian metric. We let $B_{r}^{d_\infty}$ denote the ball\footnote{We note that in Section~\ref{classification} we are also going to use $B_{\ell}^{\on{SL}_d(\bQ_p)}$, the ball of radius $p^{\ell}$ with respect to the matrix norm $\|\cdot\|_p$.} of radius $r$ with respect to $d_\infty$.  

We may define a left-invariant metric $d_p$ on $\on{SL}_d(\bQ_p)$ by taking the metric induced by $\|\cdot\|_p$ on $K = \on{SL}_{d}(\mathbb{Z}_{p})$ and declare it to have distance $2$ between different $K$-cosets $g_1K\neq g_2K$. Let $d_S$ denote the resulting product metric on $\on{SL}_d(\bQ_S)$.

For a subset $L\subset \SLS$ we denote the conjugation with $g$ by $L^{g}=g^{-1}Lg$. If $g_s\in \on{SL}_d(\bQ_s)$ for $s\in S$ then $L^{g_s}$ will denote the conjugation with $g_s$ embedded in $\SLS=\SLSS$ and we agree on the analogous convention if translating an orbit in $\SLZP\backslash \SLSS$, taking intersections or doing similiar operations.

\subsection{Height}
Fix a group $\bG\in\{\bG_1, \bG_2, \bG_{\on{joint}}\}$
and denote its Lie algebra by $\gog$. We define $\gog_\bZ=\gog\cap\Mat_d(\bZ)$ and note that $\bZ[\tfrac{1}{p}]\otimes_\bZ\gog_\bZ=\gog_{\bZ[1/p]}=\gog\cap\Mat_{d,d}(\bZ[\tfrac1p])$ is a discrete subgroup of $\gog(\bQ_S)\cong\bQ_S^{\dim(\bG)}$ that is invariant under the adjoint action of $\bG(\bZ[\tfrac{1}{p}])$. This becomes important in the following definition of measuring the complexity of a point in $\Gamma\backslash \bG(\bQ_S)$.
\begin{definition}
\label{height} The \textit{height} of a point $x\in \Gamma\backslash \bG(\bQ_S)$ is
$$\on{ht}(x)=\sup{\left\{\|\on{Ad}(g^{-1})w\|_{S}^{-1}: x=\Gamma g, w\in \gog_{\bZ[1/p]}\right\}}.$$
\end{definition}
By the invariance of $\gog_{\bZ[1/p]}$, the height is independent of the chosen representative $x=\Gamma g$. This notion is only relevant for $\mathcal{Y}_2$, where $\gog=\on{Lie}(\bG_2)=\mathfrak{sl}_{d-1}$ and we may take $\gog_\bZ$ to consist of traceless matrices with integer coefficients. Indeed, $\mathcal{Y}_1$ is compact (because $\bG_{1}$ is anisotropic at $\infty$, \cite[Thm.~I.3.2.4]{margulisbook}). For the same reason, also the orbit $\Gamma {L}_{v,S}^{+}\left(k_{v},e,\theta_{v},e\right)$ (recall that~$\mathbb{L}$ is 
obtained from $\widetilde{\mathbb{L}}_{v} = (e,g_{v})\Delta_{\mathbb{H}_{v}}(e,g_{v}^{-1})$
by projection from~$\bG_1\times\on{ASL}_{d-1}$ to~$\bG_1\times\on{SL}_{d-1}$) is compact (which is the reason for \eqref{pr-finite-M}).

By the generalized Mahler's compactness criterion (\cite[Thm.~7.10]{kleinbocktomanov}) the set $$\Sigma_{\bG_{\on{joint}}}(R)=\{x\in\mathcal{Y}_{\on{joint}}^+: \on{ht}(x)\leq R\}$$ is compact and we wish to choose $R$ large enough in the sense that it covers a large part of the support of $\mu_{v,S}=m_{\Gamma {L}_{v,S}^{+}\left(k_{v},e,\theta_{v},e\right)}$.
This is the context of the following theorem, which relies on the non-divergence results of Margulis and Dani (\cite{Dani:1981ik}, Kleinbock-Margulis~\cite{kleinbockmargulis}) or rather its $S$-adic generalization due to Kleinbock-Tomanov~\cite{kleinbocktomanov}.
\begin{lemma}[Non-Divergence]
\label{nondivergence} There exists absolute constants
$\consta\label{exp:nondivp},\consta\label{exp:nondivR}>0$ such that for every $v \in \widehat{\mathbb{Z}}^{d}$ with $\left\|v\right\|^{2} = D$,
$$\mu_{v,S}(\mathcal{Y}_{\on{joint}}^+\setminus \Sigma_{\bG_{\on{joint}}}(R))\ll p^{\ref{exp:nondivp}}R^{-\ref{exp:nondivR}}.$$
\end{lemma}
As remarked before, $\Sigma_{\bG_{\on{joint}}}(R)=\mathcal{Y}_1^+\times\Sigma_{\bG_2}(R)$ for $R$ large enough, so that $\mu_{v,S}(\Sigma_{\bG_{\on{joint}}}(R))={\pi_2}_*\mu_{v,S}(\Sigma_{\bG_2}(R))$ which reduces it to the maximal case and we may cite \cite[Lemma 7.2]{EMMV} for the above formulation.
Making the same choice as in \cite{EMMV} we put $$X_\text{cpt}=\Sigma_{\bG_{\on{joint}}}\left(p^{(\ref{exp:nondivp}+20)/\ref{exp:nondivR}}\right).$$
This gives $\mu_{v,S}(X_\text{cpt})>1-2^{-20}$ if $p$ is sufficiently large (which we may assume by Proposition~\ref{prime_existence}) to take care of the implicit (and absolute) constant in Lemma~\ref{nondivergence}.

\subsection{$S$-adic Sobolev norms}
\label{sobolevchapter}
Let $\bG<\on{SL}_d$ be a semisimple $\bQ$-group. The space of smooth functions $C_c^\infty(X)$ on $X=\Gamma\backslash\bG(\bQ_S)^+$ consists of compactly supported functions that are invariant under $$K[m]=\{g\in\bG(\bZ_p):\|g-e\|_p\leq p^{-m}\}$$ for some $m$ and are smooth at the real place.  
The latter requirement means that for $f\in C_c^\infty(X)$, for any monomial $\mathcal{D}$ in $\on{dim}(\bG)$ variables, and for any basis $X_i$ of $\on{Lie}(\bG(\bR))$, $\mathcal{D}(X_1,\dots,X_{\on{dim}(\bG)})f$ exists. 
We will use the following $S$-adic Sobolev norms $\cS_{d}$ of degree $d$ on $C_c^\infty(X)$, a variant of this already having been introduced in~\cite{venkatesh}:
\begin{equation*}
\cS_{d}(f)^2=\sum_{m\geq0}\left(p^{md}\sum_\mathcal{D}\left\|\on{pr}[m](1+\on{ht}(x))^{d}\mathcal{D}f\right\|^2_{L^2_{m_X}}\right)
\end{equation*}
The inner sum runs over all monomials $\mathcal{D}$ in the elements of a fixed basis of $\on{Lie}(\bG(\bR))$ of degree less than $d$. The operator $\on{pr}[m]$ is defined to be  the difference $\on{Av}_{m}-\on{Av}_{(m-1)}$ where $\on{Av}_{m}$ denotes average operator over $K[m]$ for $m\geq0$ and $\on{Av}_{(-1)}=0$. We will think of $\on{pr}[m]$ as the projection operator to the "space of functions of pure level $m$". Let us summarize the properties given in~\cite[Section 7.4]{EMMV}.

\begin{proposition}[Properties of Sobolev Norms]
\label{sobolev}
The following properties hold:
\begin{itemize}
\item[(S1)]\textit{($\infty$-Norm)} There exists $d_0\geq 1$ depending on $\on{dim}(\bG)$ only such that for all $d' \geq d_0$ we have 
\[
\|f\|_\infty\ll_{d'}\cS_{d'}(f).
\]
\item[(S2)]\textit{(Trace)} For every $d' \geq d_0$ there exist integers $d_2>d_1>d'$ and an orthonormal basis $\{e_k\}$ of the completion $C_c^\infty(X)$ with respect to $\cS_{d_2}$ which is orthogonal with respect to $\cS_{d_1}$ so that $$\sum_k\cS_{{d_1}}(e_k)^2<\infty\text{ and }\sum_k\frac{\cS_{d'}(e_k)^2}{\cS_{{d_1}}(e_k)^2}<\infty.$$
\item[(S3)]\textit{(Translation)} For any $g\in\bG(\bQ_S)$ and~$d'\geq1$ we have $$\cS_{d'}(g\acts f)\ll_{d'} \|g\|_S^{4{d'}}\cS_{d'}(f).$$ If $g\in K[0]$ then $$\cS_{d'}(g\acts f)=\cS_{d'}(f).$$
\item[(S4)]\textit{(Lipschitz)} If $g\in K[m]$ and~$d'\geq d_0$ then $$\|g\acts f-f\|_\infty\ll p^{-m}\cS_{d'}(f).$$
\item[(S5)]\textit{(Product)} If $f_{1}, f_{2} \in C_{c}^{\infty}(X)$ and $d'' = d' + d_{0} + 1$, then
\[
\cS_{d'}(f_{1}f_{2}) \ll_{d'} \cS_{d''}(f_{1})\cS_{d''}(f_{2}).
\]
\end{itemize}
\end{proposition}

The last property is not formulated in \cite{EMMV} and is proven below.
\begin{proof}[Proof of (S5)]
Let $d' \geq 0$ and let $\mathcal{D}_{0}$ be a monomial of degree at most $d'$. Then by the proof of the Sobolev inequality (S1) in ~\cite[Section A.5]{EMMV},
\begin{equation}
\label{sobolevS5estimate}
\left|(1+\on{ht}(x))^{d'}\mathcal{D}_{0}\on{pr}[k]f(x)\right|^{2} \ll p^{-k(d'+1)}\mathcal{S}_{d_0 + d' + 1}(f)^{2}
\end{equation}
for all $k \geq 0$ and $x \in X$. We are going to use the decompositions
\[
f_{1} = \sum_{k \geq 0}\on{pr}[k]f_{1} \mbox{ and } f_{2} = \sum_{\ell \geq 0}\on{pr}[\ell]f_{2}
\]
for $f_{1}$ and $f_{2}$. Moreover, note that $(\on{pr}[k]f_{1})(\on{pr}[\ell]f_{2})$ is a function of level at most $\on{max}(k,\ell)$, and that $(\on{pr}[k]f_{1})(\on{pr}[\ell]f_{2})$ has pure level exactly $\on{max}(k,\ell)$ if $k \neq \ell$. More formally, $\on{pr}[m]((\on{pr}[k]f_{1})(\on{pr}[\ell]f_{2}))$ vanishes if $k,\ell < m$, or if $k \neq \ell$ and $\on{max}(k,l) \neq m$. This implies that
\begin{align*}
\on{pr}&[m](f_{1}f_{2}) = \on{pr}[m]\sum_{k,\ell \geq 0}(\on{pr}[k]f_{1})(\on{pr}[\ell]f_{2})\\
&= \on{pr}[m]\sum_{k \geq m}(\on{pr}[k]f_{1})(\on{pr}[k]f_{2}) + (\on{pr}[m]f_{2})\sum_{k=0}^{m-1}\on{pr}[k]f_{1} + (\on{pr}[m]f_{1})\sum_{\ell=0}^{m-1}\on{pr}[\ell]f_{2}.
\end{align*}
Together with the definition of the Sobolev norm, the Leibniz rule, and the estimate~$1+\on{ht}(x)\ll\on{ht}(x)$ for all~$x\in X$ this gives
\begin{align*}
\mathcal{S}_{d'}(f_{1}f_{2})^{2} &\ll \sum_{m \geq 0}p^{md'}\sum_{\mathcal{D}_{1}, \mathcal{D}_{2}}\left\|\on{pr}[m](1+\on{ht}(x))^{d'}\mathcal{D}_{1}f_{1}\mathcal{D}_{2}f_{2}\right\|_{L^{2}_{m_{X}}}^{2}\\
& \ll \sum_{m \geq 0}p^{md'}\sum_{\mathcal{D}_{1}, \mathcal{D}_{2}} \Big\| \on{pr}[m]\sum_{k \geq m}\big(\on{pr}[k]\on{ht}(x)^{d'}\mathcal{D}_{1}f_{1}\big)\big(\on{pr}[k]\on{ht}(x)^{d'}\mathcal{D}_{2}f_{2}\big)\\
& \hspace{24mm} + \big(\on{pr}[m]\on{ht}(x)^{d'}\mathcal{D}_{2}f_{2}\big) \sum_{0\leq k < m}\on{pr}[k]\on{ht}(x)^{d'}\mathcal{D}_{1}f_{1}\\
& \hspace{24mm} + \big(\on{pr}[m]\on{ht}(x)^{d'}\mathcal{D}_{1}f_{1}\big) \sum_{0\leq k < m}\on{pr}[k]\on{ht}(x)^{d'}\mathcal{D}_{2}f_{2}\Big\|_{L^{2}_{m_{X}}}^{2},
\end{align*}
where the inner sum runs over all monomials $\mathcal{D}_{1}, \mathcal{D}_{2}$ of degree at most $d'$. Let $d'' = d_{0} + d' + 1$ and use (\ref{sobolevS5estimate}) for each of the six projections to obtain
\begin{align*}
\mathcal{S}_{d'}&(f_{1}f_{2})^{2} \ll \sum_{m \geq 0} p^{md'} \sum_{\mathcal{D}_{1},\mathcal{D}_{2}}\Big(\sum_{k \geq m}p^{-\frac12k(d' + 1)}\cS_{d''}(f_{1})p^{-\frac12k(d' + 1)}\cS_{d''}(f_{2}) \\
& \hspace{40mm} + p^{-\frac12m(d' + 1)}\cS_{d''}(f_{2})\sum_{0\leq k < m} p^{-\frac12k(d' + 1)}\cS_{d''}(f_{1}) \\
& \hspace{40mm} + p^{-\frac12m(d' + 1)}\cS_{d''}(f_{1})\sum_{0\leq k < m} p^{-\frac12k(d' + 1)}\cS_{d''}(f_{2})\Big)^{2} \\
& \ll \sum_{m \geq 0}p^{md'} \Big(p^{-m(d' + 1)} + p^{-\frac12m(d'+1)}\Big)^{2}\mathcal{S}_{d''}(f_{1})^{2}\mathcal{S}_{d''}(f_{2})^{2}\\
& \ll \sum_{m \geq 0}p^{-m}\mathcal{S}_{d''}(f_{1})^{2}\mathcal{S}_{d''}(f_{2})^{2} \ll \mathcal{S}_{d''}(f_{1})^{2}\mathcal{S}_{d''}(f_{2})^{2}.
\end{align*}

\end{proof}

\subsection{Proof of Equidistribtion of Full Orbit}

\begin{proof}[Proof of Theorem \ref{t:FullDynamics}]
Let us start by recalling that the groups $G_{\text{joint},S}^+$ and ${L}_{v,S}^{+}$ are normal in $G_{\text{joint},S}$ resp.\ ${L}_{v,S}$ 
(see the argument of Lemma~\ref{plusisplus}) and are of index $4$ (see 
\cite[Lemma~3.6]{AES} for an argument using the spinor norm resp.\ 
our concrete discussions in Lemma~\ref{indexsplit} and Lemma~\ref{lemmagivingsl2sl2}). 
Let $\ell^v_j=(e,h^v_j,e,(h^v_j)^{g_v})$ and $(e,g_j,e,e)$ for $j=1,\ldots,4$ 
denote coset representatives in ${L}_{v,S}$ resp.\ $G_{\text{joint},S}$.
By the argument at the end of \cite[Section~3.5]{AES} 
we may even suppose that $h^v_j\in g_j G_{1,p}^+$ for $j=1,\ldots,4$.
Hence we may define the orbit measures  $\mu^j_{v,S}$ and $m_{\mathcal{Y}^j_\text{joint}}$ corresponding to
\[
 \Gamma {L}_{v,S}^{+}\left(k_{v},h^v_j,\theta_{v},(h^v_j)^{g_v}\right)\subset
\mathcal{Y}^j_\text{joint} =
\Gamma (e,g_j,e,e)G_{\text{joint},S}^+.
\]
Normality immediately implies that $\frac14\sum \mu^j_{v,S}=\mu^{\on{Full}}_{v,S}$ 
and the analoguous statement for the ambient Haar measure (\cite[Lemma~3.8]{AES}).

Theorem~\ref{t:FullDynamics} will therefore follow if we can get the statement of Theorem~\ref{maindynamicalresult} but with $\mu_{v,S}$ replaced by $\mu_{v,S}^j$ and $m_{\mathcal{Y}^+_\text{joint}}$ by $m_{\mathcal{Y}^j_\text{joint}}$ for every $j=1,\ldots,4$.
Let $f\in C_c^\infty(\mathcal{Y}_{\on{joint}})$, and decompose it into $f=\sum f_j$ where $f_j$ is the restriction of $f$ to $\mathcal{Y}_{\on{joint}}^j$. 
By translating $f_j$ with $\ell^v_j$, we get functions $\tilde{f}_j\in C_c^\infty(\mathcal{Y}^+_{\on{joint}})$.

Note also that (by normality) $\mu^j_{v,S}$ is the push forward measure of $\mu_{v,S}$ by $\ell^v_j$, and $m_{\mathcal{Y}^j_\text{joint}}$ is the push forward measure of $m_{\mathcal{Y}^+_\text{joint}}$ by $g_j$. 
In fact, we might as well push $m_{\mathcal{Y}^+_\text{joint}}$ by $\ell^v_j$ to get $m_{\mathcal{Y}^j_\text{joint}}$.
This gives
\[
\left|\mu^j_{v,S}(f_j) - m_{\mathcal{Y}^j_{\operatorname{joint}}}(f_j)\right| \ll D^{-\ref{exp:maindynamicalresult}}\mathcal{S}_{d_2}(f_j^{\ell^v_j}).
\]
Finally, the representatives can be choosen to satisfy $\|h^v_j\|_p\leq p$ (with two elements of norm $1$ corresponding to the square and non-square representatives of $\bF^\times_p$, and the other two of norm equal to $p$), see the proof of Lemma~\ref{indexsplit} for the rank one case and Lemma~\ref{lemmagivingsl2sl2} for the rank two case. Thus, Theorem~\ref{t:FullDynamics} follows by using (S3) of the Sobolev properties and the bound on $p$ in 
Lemma~\ref{choosing_prime} (and setting e.g.\ $\ref{exp:FullDynamicsJoint}=\ref{exp:maindynamicalresult}/2$ to absorb the $(\log D)^2$-term). 
\end{proof}

The remainder of the paper is devoted to proving Theorem~\ref{maindynamicalresult}. 

\subsection{The principal $\on{SL}_2$}
For the dynamical argument we will use a unipotent flow in $\on{SL}_2(\bQ_p)$ sitting inside each simple factor of $\oLvS[p]$, which we will define now. As will be shown in Corollary \ref{classificationofH} there exists $h_v\in\on{GL}_d(\bZ_p)$ such that conjugation by $h_v$ sends $\bH_v(\bQ_p)<\on{SL}_{d}(\bQ_p)$ to\footnote{Even though
the notion of signature of a quadratic form
is meaningless over~$\bQ_p$, we use the standard notation (slightly decorated)
as this makes the definitions easy to remember.} $\on{SO}(2,1)(\bQ_p)<\on{SL}_{d}(\bQ_p)$ if $d=4$ and to either $\on{SO}(2,2)(\bQ_p)<\on{SL}_{d}(\bQ_p)$ or $\on{SO}_\eta(3,1)(\bQ_p)<\on{SL}_{d}(\bQ_p)$ if $d=5$. 
 
 In the following we will say that a homomorphism $\phi:\bH \rightarrow\bG $
 between two algebraic groups over $\bQ_p$ (each endowed with a concrete
 realization as a matrix group) is defined over $\bZ_p$ if $\phi^{-1}(\bG(\bZ_p))=\bH(\bZ_p)$. 

There exists a surjective algebraic homomorphism defined over $\bZ_p$
$$\on{SL}_2 \to \on{SO}(2,1) $$
corresponding to the adjoint representation of $\on{SL}_2$ (which  
 naturally extends to one into $\on{SO}_\eta(3,1) $). 
In the split case $\on{SO}(2,2)$ for $d=5$ we consider the representation
$$\on{SL}_2 \times \on{SL}_2 \to \on{SO}(2,2) $$
as described in Section \ref{kak-split} that maps $\on{SL}_2(\bQ_p)\times \on{SL}_2(\bQ_p)$
onto $\on{SO}(2,2)(\bQ_p)^{+}$. In this case 
we consider the image of
$\on{SL}_2$ embedded diagonally in  $\on{SL}_2\times \on{SL}_2$. 
Again, this is defined over $\bZ_p$ and since also $g_v\in\on{GL}_d(\bZ)$ 
 we may summarise the above by the next lemma.

\begin{lemma}[Principal $\on{SL}_2$]\label{principalsl2}
For $\bH\in\{\Hv,\HV, \Lv\}$ there exists a homomorphism defined over $\bZ_p$
$$\on{SL}_2(\bQ_p)\to H_p^+$$
which projects non-trivially to the isotropic almost direct factors of $H_p^+$ over $\bQ_p$.
\end{lemma}
Actually, we will give the $\lsl$-triples in $\goh$ associated to this principal $\on{SL}_2$ concretely in Section \ref{liealgebras}. We denote by $\{u_t\}$ the image of the upper unipotents in $\on{SL}_2(\bQ_p)$ under the homomorphism in Lemma \ref{principalsl2}.

\subsection{Spectral Gap}
We begin with the following which is a classical consequence from the strong approximation property of a semisimple $\bQ$-group that is isotropic over $\bQ_p$ (see \cite[Chapters II.6, II.7]{margulisbook}).
\begin{proposition}
\label{ergodic}
For $i=1,2$, $\bG^+_{i,p}$ acts ergodically on $L^2_{m_{\mathcal{Y}_i^+}}$. For $\bH\in\{\Hv,\HV,\Lv\}$ and 
the corresponding $\mu\in\{{\pi_1}_*\mu_{v,S},{\pi_2}_*\mu_{v,S},\mu_{v,S}\}$, $H_p^+$ acts ergodically on $L^2_{\mu}$.
\end{proposition}
The next property of the Sobolev norm is the deepest input of this paper, namely 
we need a form of property $(\tau)$. This is the following result about the spectral isolation of the regular representation of congruence quotients
and is also the reason for working with the subgroups $H_{v,S}^+, H_{\Lambda_v,S}^+,L_{v,S}^+$ and the homogeneous spaces $\mathcal{Y}_1^+,\mathcal{Y}_{\text{joint}}^+$. See \cite[Theorem 4.1 and Equation (4.1)]{EMMV} and the ambient section for the history of this theorem.
\begin{theorem}
\label{propertytau}
Let $\bH'$ be a simply connected algebraic semisimple $\bQ$-group that is isotropic over $\bQ_p$. Let $\bH$ be an algebraic $\bQ$-group such that there is an isogeny $\bH'\to\bH$ defined over $\bQ_p$. If $K$ is a good maximal subgroup of $\bH'(\bQ_p)$ then there exists some $\consta\label{exp:decay}>0$ which only depends on $\dim\bH$
such that for any $f_1,f_2\in L^2_{0}(\bH(\bZ[\tfrac1p])\backslash\bH(\bQ_S)^+)$ and $g_p\in \bH'(\bQ_p)$ we have
$$\left|\la g_p\acts f_1,f_2\right\ra|\leq\dim{(Kf_1)}^{\tfrac12}\dim{(Kf_2)}^{\tfrac12}\|f_1\|_2\|f_2\|_2\Xi_{\bH'(\bQ_p)}(g_p)^{\ref{exp:decay}},$$
where $\Xi_{\bH'(\bQ_p)}$ is the Harish-Chandra spherical function of $\bH'(\bQ_p)$.
\end{theorem}
The prototypical example of a good maximal subgroup is $\on{SL}_2(\bZ_p)<\on{SL}_2(\bQ_p)$, see e.g.\ \cite[Chapter 2.1]{Oh} and Appendix \ref{Good_Subgroups}. The Harish-Chandra function $\Xi_{\bH'(\bQ_p)}(g_p)$ can be bounded from above in terms of $\|g_p\|_p^{-\kappa}$ for some  $\kappa>0$ which only depends on $\dim\bH$
(see for instance \cite[Thm.~1.11]{adelicmixing} and the references therein).

We want to apply this to $\bH\in\{\Hv,\HV,\Lv\}$
where $\bH'$ will be the simply connected cover
(studied in Section \ref{simplycoverOne} and \ref{simplycoverTwo}) of one of the model groups $\on{SO}(2,1)(\bQ_{p})$, $\on{SO}_{\eta}(3,1)(\bQ_{p})$ or $\on{SO}(2,2)(\bQ_{p})$ classified in Proposition \ref{SOclassification}.
Moreover,  the isomorphism is in fact defined over $\bZ_p$ and is given in Corollary \ref{classificationofH}. A proof that for these groups and the ambient groups $\bG_1$ and $\bG_2$, their $\bZ_p$-points define good maximal compact subgroups is provided in Appendix \ref{Good_Subgroups}.

Since $H_p^+$ acts ergodically on $\Gamma H_S^+g$, where $g\in\{k_v,\oT,(k_v,\oT)\}$
and since $\on{SL}_2(\bQ_p)$ projects non-trivially onto each simple factor of $H_p^+$, 
the one-parameter subgroup $\{u_t:t\in\bQ_p\}$ of the principal $\on{SL}_2$ also acts ergodically by the Mautner phenomenom (\cite[Proposition II.3.3]{margulisbook}). Property $(\tau)$ of $\bH$ implies that the $\on{SL}_2(\bQ_p)$-action is also $\frac1m$-tempered for some absolute $m\geq 1$. Using that one deduces the following (see \cite[Appendix A.8]{EMMV}):
\begin{itemize}
\item[(S5)]\textit{(Decay of Matrix Coefficients)} There exists $\consta\label{exp:decay2}>0$ such that for all $d'\geq d_0$
$$\left|\langle u_t\acts f_1,f_2\rangle_{L^2_\mu}-\mu(f_1)\mu(\overline{f}_2)\right|\ll(1+|t|_p)^{-\ref{exp:decay2}}\cS_{d'}(f_1)\cS_{d'}(f_2).$$
\end{itemize}

We consider now one of the factors, say $\mathcal{Y}_i^+$ for~$i\in\{1,2\}$. In the following we will use the Hecke operator $\bT_t=\on{Av}_L\star\;\delta_{u(t)}\star \on{Av}_L$ on $L^2_{m_{\mathcal{Y}^+_i}}$ introduced in \cite{EMMV}. As above the operator $\on{Av}_L$ denotes convolution with the characteristic function of $\bG_i^+(\bQ_p)\cap K[L]$ and $\delta_{u(t)}$ is the action of $u_t$ from the principal $\on{SL}_2$ (in $\HvS[p]$ respectively $\HVS[p]$).
Using the Mautner phenomenon once more, $u_t$ acts also ergodically on $L^2_{m_{\mathcal{Y}^+_i}}$ and this representation is $\frac1m$-tempered.
\begin{itemize}
\item[(S6a)]\textit{(Convolution on the ambient space)} $$|\bT_t(f)(x)-m_{\mathcal{Y}^+_i}(f)|\ll p^{d_2L}\on{ht}(x)^{d_2}\|\bT_t\|_2\cS_{d_2}(f).$$
\item[(S6b)] There exists $\consta\label{exp:heckeA11}>0$ such that $\|\bT_t\|_2\ll|t|_p^{-\ref{exp:heckeA11}}p^{2d_2L}$ where $\|\bT_t\|_2$ denotes the operator norm of $\bT_t$ on $L^2_{m_{\mathcal{Y}^+_i},0}$.
\end{itemize}

We refer to \cite[Appendix A]{EMMV} for a proof of (S1) to (S6).

\subsection{Almost invariance}
We recall some more terminology from \cite{EMV}.
Recall that $g\acts x=xg^{-1}$ and let $\mu^g$ denote the push forward with respect to the map $x\mapsto g\acts x$.
\begin{definition}[Almost invariant measures]
\label{invmeas1}
The measure $\mu$ on $\mathcal{Y}$ is called $\varepsilon$-almost invariant w.r.t.~$\mathcal{S}_{d'}$ under
\begin{itemize} 
\item $g\in\bG{(\bQ_p)}$ if $|\mu^{g}(f)-\mu(f)|\leq \varepsilon \cS_{d'}(f)$ for all $f\in C^\infty_c(\mathcal{Y})$,
\item  a subgroup $L<K$ if it is $\varepsilon$-almost invariant under all $g\in L$,
\end{itemize}
\end{definition}
The main ingredient for the proof of the first half of Theorem \ref{maindynamicalresult} is the following dynamical result:

\begin{proposition}
\label{almostinvariance}
There exists $\consta\label{exp:addinvD}>0$, $d_2 > 0$ such that ${\pi_i}_*\mu_{v,S}$ is $D^{-\ref{exp:addinvD}}$-almost invariant w.r.t.~$\mathcal{S}_{d_2}$ under $\bG_i{(\bQ_p)}^+\cap K[1]$ for $i=1,2$.
\end{proposition}

As mentioned before since $\pi_i(\Gamma\LvS^+(k_v,e,\oT,e))$ are MASH sets in the sense of \cite{EMMV}, this and the results in Section \ref{dynamics} are already implied by \cite[Theorem 1.5]{EMMV}. However, we provide the argument as the case at  hand is quite a bit easier and as  the framework will also be needed in Section \ref{proofofequidistribution}.

\subsection{Proof of Theorem \ref{maindynamicalresultsingle} - Single Factor}
\label{convolution}
We now upgrade almost invariance of ${\pi_i}_*\mu_{v,S}$ under $\bG_i^+(\bQ_p)\cap K[1]$ produced by Proposition \ref{almostinvariance} to saying that $\mu$ must be close to the Haar measure. This is identically 
to \cite[Section 7.9]{EMMV} (and as such a variant of \cite[Proposition 15.1]{EMV}). We will be using the Hecke operator $\bT_t$ introduced below Theorem  \ref{propertytau}, satisfying the two properties (S6a) and (S6b) which we combine to say that
\begin{itemize}
\item[(S6)] $|\bT_t(f)(x)-m_{\mathcal{Y}^+_i}(f)|\ll \on{ht}(x)^{d_2}|t|_p^{-\ref{exp:heckeA11}}p^{3d_2L}\cS_{d_2}(f)$.
\end{itemize}
We will now prove the first half of Theorem \ref{maindynamicalresultsingle} assuming that there exists $\ref{exp:maindynamicalresultsingle}>0$ and $d_2>0$ such that  $\mu={\pi_i}_*\mu_{v,S}$  
\begin{itemize}
\item is $D^{-\ref{exp:addinvD}}$-almost invariance under $\bG_i^+(\bQ_p)\cap K[1]$ with respect to the Sobolev norm $\cS_{d_2}$ and
\item satisfies the non-divergence estimate $\mu(X\setminus\Sigma_{G_{i,S}}(R))\ll p^{\ref{exp:nondivp}}R^{-\ref{exp:nondivR}}$ (as in Lemma \ref{nondivergence}). 
\end{itemize}

\begin{proof}[Proof of Theorem \ref{maindynamicalresultsingle} - Single Factor]
Let $\pi^+$ denote integration with respect to $m_{\mathcal{Y}^+_i}$ and $\mu={\pi_i}_*\mu_{v,S}$. Then 
\begin{equation}
\label{splitting}
\int f\on{d}\!\mu-\pi^+(f)=\left(\int f\on{d}\!\mu-\int \bT_t(f)\on{d}\!\mu\right)+\int\left(\bT_t(f-\pi^+(f))\right)\on{d}\!\mu
\end{equation}
and we treat the two terms separately. As we will see the first term will be small because $\mu$ is almost invariant, 
and for the latter we deduce an estimate from the fact (S6) concerning the operator $\bT_t$. 

We split $\mathcal{Y}$ into $\Sigma_{\bG}(R)$ and its complement. 
For all $x\in \Sigma_{\bG}(R)$ we have
$$|\bT_t(f-\pi^+(f))(x)|\ll R^{d_2}|t|_p^{-\ref{exp:heckeA11}}p^{3d_2L}\cS_{d_2}(f)$$
by property (S6).  
Using the non-divergence estimate and that $\bT_t$ does not increase the supremums norm  we obtain
$$\Bigl|\int\left(\bT_t(f-\pi^+(f))\right)\on{d}\!\mu\Bigr|\ll \left(p^{3d_2L}R^{d_2}|t|_p^{-\ref{exp:heckeA11}}+p^{\ref{exp:nondivp}}R^{-\ref{exp:nondivR}}\right)\cS_{d_2}(f).$$

Note that the first term on the right hand side of \eqref{splitting}  equals 
$$\int f\on{d}\!\mu-\int \on{Av}_1\star f \on{d}\!\mu+\int\on{Av}_1\star f\on{d}\!\mu-
\int \on{Av}_1\star(\delta_{u(t)}\star \on{Av}_1\star f)\on{d}\!\mu.
$$
Since $\delta_{u(t)}\star \on{Av}_1$ \emph{does} change Sobolev norms (controlled by (S3) 
and $\|u_t\|_p\ll|t|_p^d$) we find by almost invariance of $\mu$ under $K[1]$ and invariance of $\mu$ under $u(t)$
that
\begin{eqnarray*}
	 \Bigl|\int f\on{d}\!\mu-\int \bT_t(f)\on{d}\!\mu\Bigr|&\ll& D^{-\ref{exp:addinvD}}(\cS_{d_2}(f)+\cS_{d_2}(\delta_{u(t)}\star \on{Av}_1\star f))\\
	&\ll& D^{-\ref{exp:addinvD}}(1+|t|_p^{4d_2d})\cS_{d_2}(f).
\end{eqnarray*}
The expression in \eqref{splitting} we can therefore estimate by
$$\Bigl|\int f\on{d}\!\mu-\pi^+(f)\Bigr|\ll_\varepsilon \left(D^{-\ref{exp:addinvD}+\beta4d_2d}+D^{(3d_2L+\ref{exp:heckeA11})\varepsilon+d_2\alpha-\ref{exp:heckeA11}\beta}+D^{\ref{exp:nondivp}\varepsilon-\alpha\ref{exp:nondivR}}\right)\cS_{d_2}(f),$$
where we have set $R=D^{\alpha}$, $|t|_p\in [p^{-1}D^{\beta}, D^\beta]$, and $\varepsilon>0$ is used to bound 
$p=O_\varepsilon(D^\varepsilon)$ in terms of $D$. We now choose in turn $\beta>0$ such
that the first exponent of $D$ is negative, $\alpha>0$ such that $ d_2\alpha-\ref{exp:heckeA11}\beta<0$, and finally $\varepsilon>0$ such that the second and
the third exponent of $D$ are negative. Fixing one such definition
of $\alpha,\beta,\varepsilon$
in terms of $\ref{exp:nondivp}, \ref{exp:nondivR}, \ref{exp:heckeA11},\ref{exp:addinvD},d,$ and $d_2$,
we obtain an upper bound of the form $D^{-\ref{exp:maindynamicalresultsingle}} \cS_{d_2}(f)$ as required.
\end{proof}


\section{Quadratic Forms, Discriminants, and Orthogonal Groups}
\label{discriminants}\label{classification}

As the main dynamical argument will happen on the $p$-adic factor using the principal $\on{SL}_2$, we need to show
the existence of the principal $\on{SL}_2$ and in particular
that the corresponding groups are non-compact. We will also analyze the volume growth within $\SO[Q](\bQ_{p})$ by studying the transitive action on a regular tree $H/K$ or on a product of such trees. As all of these facts are well known, we postpone some parts of the argument to Appendix \ref{Buildings}.

\subsection{Non-compactness of the orthogonal group}
For the convenience of the reader, we provide a proof of the following well known statement: Whenever a quadratic form in at least three variables is anisotropic over $\bQ_{p}$ for some prime number $p \neq 2$, then its discriminant is divisible by $p$.  Recall that the discriminant of a quadratic form is defined to be the determinant of the corresponding symmetric matrix.

We also note that the orthogonal group is never compact if the quadratic form has at least five variables, because every such quadratic form is isotropic (see for example \cite[Thm. IV.2.2.6]{serrebook}). This
is the reason why Theorem \ref{mainresultAES} has no congruence condition for $d\geq 6$
and we can restrict ourself to the cases $d=4,5$. 

\begin{proposition}\label{diagonalQF}
Let $p \neq 2$ be prime and let $Q$ be a quadratic form  over $\bQ_{p}$ in $n$ variables and let $A=\left(a_{ij}\right) \in \Mat_{n}\left(\bQ_{p}\right)$ be the symmetric matrix corresponding to $Q$. Then $Q$ is $\bZ_{p}$-equivalent to a diagonal form
\[
c_{1}x_{1}^{2} + c_{2}x_{2}^{2} + \ldots + c_{n}x_{n}^{2},
\]
i.e.\ after a $\GL_d(\bZ_p)$-coordinate change $Q$ has the above form.
Moreover, we have that $\max_{k}\left|c_{k}\right|_{p} = \max_{i,j}\left|a_{ij}\right|_{p}$.
\end{proposition}

For an element $g \in \operatorname{Mat}_{n}(\mathbb{Q}_{p})$
recall that $\|g\|_p=\max_{i,j}|g_{i,j}|_{p}$, which satisfies
\[
\left\|gh\right\|_p\leq\left\|g\right\|_p \left\|h\right\|_p
\]
for any two $g,h\in\on{Mat}_{n.n}(\bQ_p)$. In particular $g \in\on{Mat}_d(\bQ_p)\mapsto \left\|g\right\|_p$ is bi-invariant under the compact subgroup $\on{GL}_n(\bZ_p)$. We prove Proposition \ref{diagonalQF} in Appendix~\ref{proofoffirstinB}. 

Notice that a coordinate change corresponding to a matrix in $\on{GL}_n(\bZ_p)$ (as in the proof of Proposition \ref{diagonalQF}) does not change the valuation of the discriminant of $Q$ with respect to $p$.
In fact by definition the discriminant changes by the square of the determinant of the coordinate
change matrix.

By construction of $p$ we have $p\equiv 1 ~ \operatorname{mod} ~ 4$ so that by Hensel's Lemma  $-1$ is a square in $\bZ_p$ and $2$ is invertible in $\bZ_{p}$. We write $Q_{1} \sim Q_{2}$ for two quadratic forms $Q_{1}$, $Q_{2}$ if they differ by a change of basis over $\bZ_{p}$, or equivalently if their corresponding matrices are in the same $\on{GL}_{n}(\bZ_{p})$-orbit. If
$\tilde Q_2'=aQ_2$ for some $a\in\bZ_p^\times$, then the special orthogonal
group for $Q_2$ and $Q_2'$ are identical and we also say that the special orthogonal groups for $Q_1$ and $Q_2$ are $\mathbb{Z}_{p}$-conjugate to each other.

We denote the special orthogonal group of a quadratic form $Q$ by $\SO[Q]$, and in the special case of $Q=\sum_{i=1}^{n} x_{i}^{2}$, write $\on{SO}_n$.
\begin{proposition}
\label{SOclassification}
Let $Q(x,y,z)$ be a ternary quadratic form over $\mathbb{Z}_{p}$ with $p \nmid \operatorname{disc}(Q)$. Then $\operatorname{SO}_{Q}(\mathbb{Q}_{p})$ is $\mathbb{Z}_{p}$-conjugate to $\SO[2xy+z^{2}](\bQ_{p})$ and so is isotropic.

If $Q(x,y,z,w)$ is a quaternary quadratic form over $\mathbb{Z}_{p}$ with $p \nmid \operatorname{disc}(Q)$, then either $\operatorname{SO}_{Q}(\mathbb{Q}_{p})$ is $\mathbb{Z}_{p}$-conjugate to $\SO[2xy+z^{2}+\eta w^{2}](\bQ_{p})$ for a non-square $\eta\in\bZ_{p}^{\times}$ or to $\SO[2xy+2zw](\bQ_{p})$. In both cases, $\operatorname{SO}_{Q}(\mathbb{Q}_{p})$ is once more isotropic.
\end{proposition}

\begin{proof}
We only prove the quaternary case as the same calculation will work for a form in three variables. The proof relies on the following easy observations:
\begin{itemize}
\item If $a\in\left(\bZ_{p}^{\times}\right)^{2}$ then $ax^{2}\sim x^{2}$.
\item $x^{2}+y^{2}\sim x^{2}-y^{2}\sim 2xy$ (since $-1 \in (\mathbb{Z}_{p}^{\times})^{2}$).
\item $xy\sim cxy$ for any $c\in\bZ_p^\times$.
\item If $a,b$ are both non-squares in $\bZ_{p}^{\times}$ then $\frac a b$ is a square (by Hensel's Lemma and the structure of $\bF_p^\times$, see e.g. \cite[II.3.3]{serrebook}).
\end{itemize}
Applying Proposition \ref{diagonalQF}, we may assume that $Q$ is of the form $d_{1}x_{1}^{2} + d_{2}x_{2}^{2} + d_{3}x_{3}^{2} + d_{4}x_{4}^{2}$ with $d_{1},d_{2},d_{3},d_{4} \in \mathbb{Z}_{p}$. Since $p\nmid \operatorname{disc}(Q) = d_{1}d_{2}d_{3}d_{4}$, we even have $d_{1},d_{2},d_{3},d_{4} \in \mathbb{Z}_{p}^{\times}$. We prove the proposition by going through a case by case study of how many of the coefficients of $Q$ are non-squares.

If all coefficients are squares, then $d_1x_1^{2}+d_2x_2^{2}+{d_3}x_3^{2}+{d_4}x_4^{2}\sim x_1^{2}+x_2^{2}+x_3^{2}+x_4^{2}\sim 2x_1x_2+2x_3x_4$.

If only (say) ${d_4}$ is a non-square, then $d_1x_1^{2}+d_2x_2^{2}+{d_3}x_3^{2}+{d_4}x_4^{2}\sim x_1^{2}+x_2^{2}+x_3^{2}+{d_4}x_4^{2}\sim 2x_1x_2+x_3^{2}+{d_4}x_4^{2}$.

If ${d_3}$ and ${d_4}$ are non-squares, then $d_1x_1^{2}+d_2x_2^{2}+{d_3}x_3^{2}+{d_4}x_4^{2}\sim 2x_1x_2+ {d_3}(x_3^{2}+\frac {d_4} {d_3} x_4^{2})\sim 2x_1x_2+{d_3}2x_3x_4\sim 2x_1x_2+ 2x_3x_4.$

The remaining cases can be reduced to the above by multiplying the quadratic form with a non-square in~$\bZ^\times$.
\end{proof}

Let us denote the above three model groups by $H=\on{SO}(2,1)(\bQ_{p})$, $H=\on{SO}_{\eta}(3,1)(\bQ_{p})$ and $H=\on{SO}(2,2)(\bQ_{p})$. We will also think of $H$ as a subgroup of $\on{SL}_{4}(\mathbb{Q}_{p})$  respectively of $\on{SL}_{5}(\mathbb{Q}_{p})$ 
by identifying $H$ with 
\[
\left\{\begin{pmatrix}h&0\\0&1\end{pmatrix} ~:~ h \in H\right\} \leq \on{SL}_{d}(\mathbb{Q}_{p}),~ d=4 \mbox{ or } 5.
\]

The previous propositions culminate in the following uniform description of the stabilizer groups $H_{v,p}$ defined in Section \ref{ReformulationTowardsHomogeneousDynamics}. It will allow us to restrict ourselves to the study of the model groups since the isomorphism to $H_{v,p}$ is via $\on{GL}_d(\bZ_p)$ and so also isometric with respect to the $p$-adic maximum norm.
\begin{corollary}
\label{classificationofH}
Let $d=4,5$. For any vector $v \in \mathbb{Z}^{d}$ and 
our choice of $p$ (with $p\nmid D=\left\|v\right\|^{2}$) there exists an $h_{v}\in \on{GL}_{d}(\bZ_{p})$ such that
\[
H_{v,p}=h_{v}Hh_{v}^{-1},
\]
where $H$ is one of our three model groups seen as a subgroup of $\on{SL}_{d}(\mathbb{Q}_{p})$.
\end{corollary}
\begin{proof}
Denote by $Q_v$ the restriction of the quadratic form $\sum_{i=1}^{d}x_{i}^{2}$ to the lattice $\Lambda_v\subset v^{\perp}$ and let  be a $\bZ$-basis of $\Lambda_v$. 
Since the covolume of $\Lambda_v$ can be calculated as the square root of the determinant of the matrix consisting of all inner products $\langle w_i,w_j\rangle$
we see from Section \ref{geometry} that the discriminant of $Q_v$ is $D$. Also note that $H_{v,p}$ is the orthogonal group
of $Q_v$. Using the matrix consisting of the integer rows $w_1,\ldots,w_{d-1},v$ 
(and determinant $D\in \bZ_p^\times$) we may conjugate $H_{v,p}$ into a block matrix form (of
the same form as $H$ as a subgroup of $\on{SL}_d(\bQ_p)$). 
Applying the previous proposition and using $p\nmid D$ again, we see that $H_{v,p}$ is $\mathbb{Z}_{p}$-conjugate to one of the model groups.
\end{proof}

\subsection{The norm balls}
For any subgroup $H<\on{GL}_2(\bQ_p)$ we define the following balls and spheres
\[
B_\ell=B_{\ell}^H=\{h\in H: \|h\|_p\leq p^{\ell}\}\quad\text{ and }\quad \partial B_{\ell}=\{h\in H: \|h\|_p= p^{\ell}\}
\]
 for any $\ell \geq 0$.
For $H=\SO[Q](\bQ_{p})$ we define the `standard' compact subgroup of $H$
by
\[
 K=\SO[Q](\bZ_{p})=B_0^H.
 \]
For the three cases of our model groups we also define the
`standard' diagonal subgroup $A=A_+A_+^{-1} \leq H$, where
\begin{align*}
A_{+} &=  \left\{\operatorname{diag}(p^{-m},p^{m},1) : m \in \mathbb{Z}_{\geq 0}\right\} \leq \operatorname{SO}(2,1)(\mathbb{Q}_{p}),\\
 A_{+} &= \left\{\operatorname{diag}(p^{-m},p^{m},1,1) : m \in \mathbb{Z}_{\geq 0}\right\} \leq \operatorname{SO}_\eta(3,1)(\mathbb{Q}_{p}),\\
 A_{+} &= \left\{\operatorname{diag}(p^{-m},p^{m},p^{-n},p^{n}) : m\geq n \in \mathbb{Z}_{\geq 0}\right\} \leq \operatorname{SO}(2,2)(\mathbb{Q}_{p})
\end{align*}
is the positive Weyl chamber in $A$. 

In Appendix \ref{secB2-cartan} we will state and prove the Cartan decomposition $H=KA_{+}K$ for the three model groups. Let us note here a few immediate corollaries of this decomposition.
For instance since $\|a\|_p=\|a^{-1}\|_p$ for all $a\in A$ it follows that
\[
 \|h\|_p=\|h^{-1}\|_p\mbox{ for all }h\in H.
\]
The Cartan decomposition of $H$ and the bi-invariance of the norm under $K$
also gives immediately the shape of $\partial B_{\ell}$ as we now explain.

 For the rank one groups $\on{SO}(2,1)(\bQ_{p})$ and $\on{SO}_{\eta}(3,1)(\bQ_{p})$ we define $a_{p}=\diag{p^{-1},p,1}$ and $a_{p}=\diag{p^{-1},p,1,1}$ respectively. Then
\[
\partial B_{\ell}=Ka_{p}^{\ell}K
\]
for any $\ell\geq 0$.

For $\on{SO}(2,2)(\bQ_{p})$ we have a second parameter and set $a_{p}=\diag{p^{-1},p,1,1}$ and $b_{p}=\diag{1,1,p^{-1},p}$. Then,
\[
\partial B_{\ell}=\bigsqcup_{0\leq j \leq \ell}Ka_{p}^{\ell}b_{p}^{j}K\cup\bigsqcup_{0\leq i \leq \ell}Ka_{p}^{i}b_{p}^{\ell}K.
\]
Since $\omega = \left(\begin{smallmatrix}&&1&0\\&&0&1\\1&0&&\\0&1&&\end{smallmatrix}\right) \in K$ and $\omega^{-1} a_{p} \omega = b_{p}$, this is the same as
\[
\partial B_{\ell}=\bigsqcup_{0\leq j \leq \ell}Ka_{p}^{\ell}b_{p}^{j}K.
\]

\subsection{The geometric structure of the rank one groups}
\label{quasisplitgeometry}

With the observations above and by counting the number of left cosets of $K$ in the level sets $Ka_{p}^{\ell}K$ (which is done in Appendix \ref{Buildings}), we are now able to calculate the volume of the norm balls in the rank one cases. For this we always normalize the Haar measure $m$ of $H$ such that $m(K)=1$.

\begin{proposition}[Volume of norm balls]
\label{rankonevolume}
The Haar measure of $B^{H}_{\ell}$ is equal to $1+\frac{p+1}{p-1}(p^{\ell}-1)$ for $\on{SO}(2,1)(\bQ_{p})$ and equal to $1+\frac{p^{2}+1}{p^{2}-1}(p^{2\ell}-1)$ for $\on{SO}_{\eta}(3,1)(\bQ_{p})$.
\end{proposition}
\begin{proof}
For $\on{SO}(2,1)(\bQ_{p})$, we decompose
\[
m(B^{H}_{\ell})=\sum_{k=0}^\ell m(\partial B^{H}_{k})=\sum_{k=0}^\ell m(Ka_p^kK).
\]
By the normalization assumption $m(Ka_p^0K)=m(K)=1$ and therefore we have for the first sphere $m(Ka_pK)=p+1$ (see Lemma \ref{KaKindex} and Lemma \ref{cosetdecomposition}). On the other hand, by applying Lemma \ref{higherlevelcosets}, we also see that $m(Ka_p^kK)=pm(Ka_p^{k-1}K)$ and consequentially
\[
m(B^{H}_{\ell})=1+(p+1)\sum_{k=0}^{\ell-1}p^k=1+\tfrac{p+1}{p-1}(p^\ell-1).
\]
The case $\on{SO}_{\eta}(3,1)(\bQ_{p})$ follows upon replacing $p$ by its square.
\end{proof}

Let $H=\on{SO}_\eta(3,1)(\bQ_{p})$ ($H=\on{SO}(2,1)(\bQ_{p})$). 
We define a metric on $H / K$ as follows:
\[
d(gK,hK) = \log_{p}\left\|g^{-1}h\right\|_p~\mbox{ for } g,h \in H.
\]
If we define an incidence relation by $gK\sim hK$ if $d(gK,hK)=1$ then $H/K$ is a $p^{2}+1$-regular  ($p+1$-regular) tree on which $H$ acts transitively and neighbour preserving. A more detailled discussion of the tree structure of $H/K$ is given in Appendix \ref{app-tree}. We are going to use this geometric action later in the paper to find lattice points with a certain property. 

\begin{lemma}
\label{geodesics}
	The subset $AK=\{a_p^nK: n\in\bZ\}$
	describes a geodesic inside the tree through the point $K$.  Translating by an element $g\in H$ moves $AK$ to another geodesic $gAK$. More specifically, if $h=gamg^{-1}\in H$ for 
	$a\in A$, $m\in M=K\cap C_H(A)$, and $g\in H$, where $C_H(A)$ denotes the commutator group of $A$ inside $H$, then $h$ preserves the geodesic $F=gAK$. The element $g$ can be chosen to satisfy $d(gK,K)=d(F,K).$
\end{lemma}
\begin{proof}
 Since $d(a_p^mK,a_p^nK)=\vert m-n\vert$ for all $m,n\in\bZ$
 we see that $AK$ is indeed a geodesic in $H/K$. Since the action by $H$
 is isometric, the same holds for $gAK$ for any $g\in H$. 

	Let us suppose now $h=gamg^{-1}$ and $a=a_p^n$ for some $n$. Then $hgAK=gamAK=gAK$ because $m\in C_H(A)\cap K$. Finally, if $g$ conjugates $h$ to $a\in A$ then any other element of $gA$ does as well and we may replace $g$ with $ga_p^n$ where $n$ satisfies $d(ga_p^nK,K)=d(F,K)$.
\end{proof}

Now recall that the adjoint representation
of $\SL_2$ acts on $\mathfrak{sl}_2$ (consisting of matrices with zero trace), which we
may equip with the determinant quadratic form. This defines a map from $\SL_2(\bQ_p)$
into $\SO[](2,1)(\bQ_p)$, which can be used to prove the following result (we omit the details).  

\begin{proposition}
\label{pserretree}
$\SLQ[2]$ acts transitively on vertices of even distance on a $p+1$-regular tree and $\SLZ[2]$ stabilizes a vertex. For $\ell\geq1$ we have
\begin{align*}
	m_{\SLQ[2]}\left(\SLZ[2]a_p^\ell\SLZ[2]\right)&=(p+1)p^{2\ell-1}\\
	m_{\SLQ[2]}\left(B^{\SLQ[2]}_\ell\right)&=1+\tfrac{p}{p-1}(p^{2\ell}-1),
\end{align*}
 where $a_p=\on{diag}(p,p^{-1})$ and $m_{\SLQ[2]}$ denotes the Haar measure on ${\SLQ[2]}$ normalized such that $m_{\SLQ[2]}(\SLZ[2])=1$. \end{proposition}

\subsection{The simply connected cover in the rank one cases}
\label{simplycoverOne}
As defined in Section \ref{ReformulationTowardsHomogeneousDynamics}, $H^+$ is the subgroup of $H$ generated by its unipotent elements. An alternative description can be given for $\SO[](2,1)$ by considering the adjoint representation $\SL_2\to\on{SO}(2,1)$. Here and also in the other cases below we have that the index
of $H^+$ in $H$ equals $[\bQ_p^\times:(\bQ_p^\times)^2]=4$.

\begin{lemma}
\label{indexd4}
The group $\on{SO}(2,1)(\bQ_p)^+$  is a normal subgroup of index $4$ in $\on{SO}(2,1)(\bQ_p)$, namely the image of $\on{SL}_2(\bQ_p)$
under the adjoint representation.
\end{lemma}

We skip the proof (which is similar to the following).
In the case $\SO[\eta](3,1)$ we may use the sporadic isogeny $\psi:\SL_2(\bQ_p(\sqrt{\eta}))\to \on{SO}_\eta(3,1)(\bQ_{p})$ (both considered as $\bQ_p$-groups) given by
$$
\psi(g)x=gxg^* \mbox{ acting on } V=\bigl\{x\in\Mat_{2,2}(\bQ_p(\sqrt{\eta})):x=x^*\bigr\},
$$
where for $x=\left[\begin{smallmatrix}x_1&x_2\\ x_3&x_4\end{smallmatrix}\right]$ we denote $x^*=\left[\begin{smallmatrix}\overline{x}_1&\overline{x}_2\\ \overline{x}_3&\overline{x}_4\end{smallmatrix}\right]^T$ and $\overline{x}_i$ is the Galois conjugate of $x_i$ in $\bQ_p(\sqrt{\eta})$. We note that $\psi$ preserves  $Q(x)=\det(x)$.

We choose the basis of $V$ consisting of
\[
e_{1}=\smallmat{2&0\\0&0},\; e_{2}=\smallmat{0&0\\0&1},\; e_{3}=\smallmat{0&\varepsilon \\\varepsilon &0},\; e_{4}=\smallmat{0&\sqrt{\eta}\\-\sqrt{\eta}&0},
\]
 where $\varepsilon\in\bQ_p $ is a square root of $-1$. Then $Q(x,y,z,w)=2xy+z^2+\eta w^2$ agrees with the defining form of $\on{SO}_\eta(3,1)$. If we introduce the notation $\Re{(x)}=\frac{x+\overline{x}}{2}$, $\Im(x)=\frac{x-\overline{x}}{2\sqrt{\eta}}$ and $|x|^2=x\overline{x}$ for $x\in\bQ_p(\sqrt{\eta})$ then for $g=\smallmat{a&b\\c&d}\in \SL_2(\bQ_p(\sqrt{\eta}))$ we have for instance
 
\[
 \psi(g)e_1=ge_1g^*=\smallmat{2|a|^2&2a\overline{c}\\2\overline{a}c&2|c|^2}=|a|^2e_1+2|c|^2e_2+\tfrac{2}{\varepsilon}\Re{(a\overline{c})}e_3+2\Im(a\overline{c})e_4 
\]

 We will prove now that $\psi$ is $2$ to $1$ and the image of $\psi$ agrees with $\on{SO}_\eta(3,1)(\bQ_{p})^{+}$.

\begin{lemma}
\label{plusisplus}
The image of $\SL_2(\bQ_p(\sqrt{\eta}))$ under $\psi$ is a normal subgroup that agrees with $\on{SO}_\eta(3,1)(\bQ_{p})^{+}$.
\end{lemma}
\begin{proof}
	We first note that~$\SL_2(\bQ_p(\sqrt{\eta}))$ and so also its image $\psi\left(\SL_2(\bQ_p(\sqrt{\eta}))\right)\subset \on{SO}_\eta(3,1)(\bQ_{p})$ is generated by unipotent matrices.
	Moreover, since $\psi$ has finite kernel, the image has the same Lie algebra as $\on{SO}_\eta(3,1)(\bQ_{p})$. 
	Now note that every unipotent element of $\on{SO}_\eta(3,1)(\bQ_{p})$ belongs to a unique one-parameter unipotent subgroup
	which is also uniquely 
	determined by a single nilpotent element of the Lie algebra. Together we see that~$\psi\left(\SL_2(\bQ_p(\sqrt{\eta}))\right)$
	is the subgroup generated by all unipotent matrices of $\on{SO}_\eta(3,1)(\bQ_{p})$, which is normal because the generating set is invariant under conjugation.
\end{proof}

\begin{lemma}
\label{indexsplit}
The kernel of $\psi$ is $\{\pm e\}$ and its image $\on{SO}_\eta(3,1)(\bQ_{p})^{+}$ has index $4$ in $\on{SO}_\eta(3,1)(\bQ_{p})$.
\end{lemma}
\begin{proof}
We will simplify an arbitrary element of $\on{SO}_\eta(3,1)(\bQ_{p})$ using the procedure of Proposition \ref{pRankOneKAK} as much as possible and will see that the isotropic part and the anisotropic part each have two cosets under the image  of diagonal matrices under $\psi$.
For $g=\smallmat{a&b\\c&d}\in \SL_2(\bQ_p(\sqrt{\eta}))$, $\psi(g)$ takes the form
\[
\begin{bmatrix}
|a|^2	&	\tfrac12|b|^2	&	\varepsilon  \Re(a\overline{b})	&	\eta\Im(a\overline{b})	\\

2|c|^2	&	|d|^2	&	2\varepsilon  \Re(c\overline{d})	&	2\eta\Im(c\overline{d})	\\

-2\varepsilon \Re(a\overline{c})	&	-\varepsilon \Re(b\overline{d}) &	\Re(a\overline{d}+b\overline{c})	&	-\varepsilon \eta\Im(a\overline{d}-b\overline{c})	\\

2\Im(a\overline{c})	&	\Im(b\overline{d})	&	\varepsilon \Im(a\overline{d}+b\overline{c}) &	\Re(a\overline{d}-b\overline{c})
\end{bmatrix}.
\]
With this one also checks that the unipotents introduced in the beginning of Proposition~\ref{pRankOneKAK} satisfy
\begin{align*}
u_{3}\left(t\right) &=\psi\left(\smallmat{1	&	0	\\	-\tfrac{\varepsilon t}{2}	&	1}\right),& 
&v_{3}\left(t\right) =\psi\left(\smallmat{1	&	-\varepsilon t	\\	0	&	1}\right),\\
u_{4}\left(t\right) &= \psi\left(\smallmat{1	&	0	\\	\tfrac{\varepsilon\sqrt{\eta}t}{2}	&	1}\right),&
&v_{4}\left(t\right) = \psi\left(\smallmat{1	&	-\varepsilon\sqrt{\eta}t	\\	0	&	1}\right).
\end{align*}
and the element $\psi\left(\smallmat{0	&	-1	\\	1	&	0}\right)=\smallmat{0&\tfrac12&0&0\\2&0&0&0\\0&0&1&0\\0&0&0&-1}$ may replace the role of $\omega$ in that proposition to deduce that after multiplying an arbitrary $g\in\on{SO}_\eta(3,1)(\bQ_{p})$ with elements of $\on{SO}_\eta(3,1)(\bQ_{p})^{+}$ on the left and right, we can assume it to be of the form
\[
\smallmat{A&0&0&0\\0&D&0&0\\0&0&X&Z\\0&0&Y&W}.
\]
We note that in the upper diagonal block we must have $D=A^{-1}$ and so the lower block matrix defines an element in $\on{SO}_{z^2+\eta w^2}(\bZ_{p})$. A calculation reveals that the lower block is of the form 
 $\smallmat{\Re(s) 	&	-\varepsilon \eta \Im(s)	\\	
 \varepsilon \Im(s)	&	\Re(s)}$ for some $s\in\bQ_p(\sqrt{\eta})$ with $|s|^2=1$.
We also see that this matrix is diagonal if and only if $s=\pm1$.

Let us study the image of $\psi$ now. 
If $\psi(g)$ is of the above block form, then in particular $b=c=0$ and $g$ is diagonal. 

In fact, we also see that $\psi(g)$ itself is diagonal if and only if $g$ is diagonal and either $g$ or $\sqrt{\eta}g$ has entries in $\bQ_p$. The image 
\[
	\left\{\psi(\on{diag}(a,a^{-1})): a\in\sqrt{\eta}^\mu\bQ_p^\times,\;\mu\in\{0,1\}\right\}
\]
of these matrices agrees with 
\[
	\left\{\on{diag}(a,a^{-1},(-1)^\mu,(-1)^\mu): a\in\bQ_p^\times, \log_p|a|_p\in2\bZ,\mu\in\{0,1\}\right\},
\]
which is of index two in the group of all diagonal matrices of $\on{SO}_\eta(3,1)(\bQ_{p})$.

On the other hand, for $t\overline{t}=1$, the image of $\on{diag}(t,\overline{t})$ is
\[
	\smallmat{ 1&0&0&0 \\ 0&1&0&0 \\
		0&0&\Re(t^2)	&	-\varepsilon \eta \Im(t^2)	\\	
		0&0&\varepsilon \Im(t^2)	&	\Re(t^2)
	}
\]
and thus forms (when restricted to the lower block) an index two subgroup of $\on{SO}_{z^2+\eta w^2}(\bZ_{p})$ because $[\bF_p^{\times}:(\bF_p^{\times})^2]=2$.
Together, this shows that $\on{SO}_\eta(3,1)(\bQ_{p})^+$ has index $4$ in $\on{SO}_\eta(3,1)(\bQ_{p})$.

Finally, we can deduce from the above also that $\psi(g)=e$ if and only if $g=\pm e$.
\end{proof}

It will also be necessary to understand subgroups of $\on{SO}_\eta(3,1)(\bQ_{p})^+$ that are locally isomorphic to $\on{SO}(2,1)(\bQ_{p})$.

\begin{lemma}
\label{oneconjugateclass} Any algebraic subgroup $H$ of $\on{SO}_\eta(3,1)(\bQ_{p})$ over~$\bQ_p$
that is locally isomorphic to 
$\on{SO}(2,1)(\bQ_{p})$ is conjugate to any other such group in $\on{SO}_\eta(3,1)(\bQ_{p})$.
\end{lemma}
\begin{proof}
The Lie algebra of $H$ is $\lsl$, and the preimage $L$ of $H$ in the simply-connected cover $\SL_2(\bQ_p(\sqrt{\eta}))$ must   be   isomorphic to $\SL_2(\bQ_p)$.
Fix some~$\bQ_p$-split torus in~$L$. After conjugation, we may assume that this torus coincides with the $\bQ_p
$-split part of the diagonal subgroup of $\SL_2(\bQ_p(\sqrt{\eta}))$ and so the Borel subgroup of~$L$ is a subgroup of the upper diagonal matrices in $\SL_2(\bQ_p(\sqrt{\eta}))$. The unipotent radical $N$ of the Borel subgroup of $L$ then forms a one-dimensional subgroup in the two-dimensional (over $\bQ_p$) subgroup of the upper unipotent subgroup $U$ in $\SL_2(\bQ_p(\sqrt{\eta}))$. Thus there exists $\beta$ such that $N_\beta=\{\smallmat{1&\beta x\\0&1}:\;x\in\bQ_p\}$. We may assume without
loss of generality that $\vert\beta\vert_p=1$. Also note that the torus and the subgroup $N_\beta$ uniquely determine the subgroup $L$ (e.g. by the Jacobson-Morozov theorem).

Let us now consider a second subgroup~$H'$ with cover~$L'$. As above we may conjugate~$L'$
and arrive at another subgroup $N_{\beta'}$ with~$\vert\beta'\vert_p=1$. Notice that conjugating by a diagonal element $\diag{\alpha,\alpha^{-1}}$ in $\SL_2(\bQ_p(\sqrt{\eta}))$ commutes with the~$\bQ_p$-split torus considered above and normalizes $U$. In particular, the conjugation class of the group associated to $\beta$ are the groups associated to $\alpha^2\beta$, for $\alpha$ of norm one. We see that there are two conjugation classes depending on whether $\beta$ is a square in $\bZ_p(\sqrt{\eta})^\times$ or not.

On the other hand, $N_\beta$ where $\beta=\beta_1+\sqrt{\eta}\beta_2$ is mapped under the isogeny $\psi$ to
\[
\psi(N_\beta)=\left\{\begin{bmatrix}
1&\frac{x^2}{2}& \varepsilon x\beta_1& -\eta x\beta_2\\
0&1&0&0 \\
0&-\varepsilon x \beta_1& 1&0 \\
0&\beta_2 x & 0 & 1
\end{bmatrix}\mid x\in \bQ_p\right\}.
\]
We now conjugate the elements of this subgroup by the block matrix consisting of the identity in the upper left block and the lower right block $\smallmat{\Re(s) 	&	-\varepsilon \eta \Im(s)	\\	
 \varepsilon \Im(s)	&	\Re(s)}$  where $s\overline{s}=1$
(which does not necessarily belong to $\on{SO}_\eta(3,1)(\bQ_{p})^+$). 
Conjugating $\psi(N_\beta)$ by this element gives, after a short calculation, $\psi(N_{s\beta})$. 

 We conclude that the two different conjugates classes merge when allowing conjugation by elements of $\on{SO}_\eta(3,1)(\bQ_{p})$ (instead of just
the elements of the index 4 subgroup $\on{SO}_\eta(3,1)(\bQ_{p})^+$).
\end{proof}

\subsection{The group $\operatorname{SO}(2,2)(\bQ_{p})$}
\label{kak-split}
\label{simplycoverTwo}
In this subsection we will consider $\on{SO}(2,2)$ but note that for some of the arguments it will be more convenient to use 
a different but $\bZ_p$-equivalent
quadratic form.
In fact let us define $\on{SO}(2,2)(\bQ_{p})$ by  
using the quadratic form 
\[
Q = 2 \det \mbox{ on } V=\Mat_{2,2}(\bQ_{p}).
\] 
We will also use the standard basis
\[
e_{1}=\smallmat{1&0\\0&0},\; e_{2}=\smallmat{0&1\\0&0},\; e_{3}=\smallmat{0&0\\1&0},\; e_{4}=\smallmat{0&0\\0&1},
\]
and note that this means that instead of looking at the quadratic form $2xy+2zw$ we consider now the 
quadratic form $2xw-2yz\sim 2xy+2zw$.  

We consider the action of $(g,h)\in\SLQ[2]\times\SLQ[2]$ on $v\in V$ given by $$\psi_{g,h}:v\mapsto gvh^{-1}.$$
Clearly, $\psi_{g,h}$ defines an element in $H=\SO[2\det](\bQ_{p}) \simeq \on{SO}(2,2)(\bQ_{p})$ and the kernel of the map $(g,h)\mapsto\psi_{g,h}$ is $\{\pm(e,e)\}$. Fix the unipotents
$$u_{1}(t)=\smallmat{1 &  &  &  \\t & 1 &  &  \\ &  & 1 &  \\ &  & t & 1},\; u_{2}(t)=\smallmat{1 &  &  &  \\ & 1 &  &  \\t &  & 1 &  \\ &  t&  & 1},\; v_{1}(t)=u_{1}(t)^{T},\; v_{2}(t)=u_{2}(t)^{T}$$ which are the images of $(e,\smallmat{1&-t\\&1})$, $(\smallmat{1&\\t&1},e)$ and their transposes. Denote by $\on{SO}(2,2)(\bQ_{p})^{+}$ the image of $\psi$ and similiarly the image
of $\on{SL}_2(\bZ_p)\times\on{SL}_2(\bZ_p)$ by $\on{SO}(2,2)(\bZ_{p})^{+}$ (it is easy
to see that the $\bZ_p$-points are mapped to $\bZ_p$-points, see also Lemma \ref{prevstructure}). Clearly, the argument of Lemma \ref{plusisplus} applies again and it agrees with our usual definition of $\on{SO}(2,2)(\bQ_{p})^{+}$.

\begin{lemma}
\label{lemmagivingsl2sl2}
The subgroup $\on{SO}(2,2)(\bQ_{p})^{+}$ has index $4$ in $\on{SO}(2,2)(\bQ_{p})$.
\end{lemma}
\begin{proof}
We will use the above unipotent matrices as row (column) operations by multiplying with them on the left (right). Let $g\in H$. Multiplying with $v_{1}(1)$, $v_{2}(1)$ or $v_{2}(1)v_{1}(1)$ on the left as needed, we may assume that the upper left entry does not vanish. We now can multiply with suitable $u_{1}(t_{1})$, $u_{2}(t_{2})$ from the left and $v_{1}(t_{1})$, $v_{2}(t_{2})$ from the right to get a matrix whose first row  and first column is of the form $(*,0,0,*)$ respectively $(*,0,0,*)^{T}$. Since the first entry is nonzero, but $Q(e_{1})=0$, invariance of the quadratic form forces the last entry  of the first column to vanish. Since the symmetric matrix corresponding to $Q$ is its own inverse, $H=H^{T}$ and we may argue similarly for the first row vector.

Applying the matrix to $e_{2}$ and $e_{3}$ (satisfying $Q(e_2)=Q(e_3)=0$) we see that the matrix is now of the form $$\smallmat{* & 0 & 0 & 0 \\0 & * & 0 & * \\0 &  0& * & * \\0 & * & * & *}\text{ or }\smallmat{* & 0 & 0 & 0 \\0 & 0 & * & * \\0 & * & 0 & * \\0 & * & * & *}.$$ Continue by applying the matrix to $e_{1}+e_{2}$ and $e_{1}+e_{3}$ (once more with $Q(e_1+e_2)=Q(e_1+e_3)=0$) to see that the last column of either matrix must vanish but for the very last entry, doing the same for its transpose gives us the possible matrices
$$\smallmat{* & 0 & 0 & 0 \\ 0& * & 0 & 0 \\ 0& 0 & * & 0 \\ 0& 0 & 0 & *}\text{ or }\smallmat{* & 0 & 0 & 0 \\ 0& 0 & * & 0 \\ 0& * & 0 & 0 \\ 0&  0& 0 & *}.$$ 

Apply the matrix to $e_{1}+e_{4}$ and $e_{2}+e_{3}$ (with $Q(e_1+e_4)=-Q(e_2+e_3)=1$) to see that the $(1,1)$ and $(4,4)$ entries, respectively the other pair, are inverses to each other. The latter matrix then has determinant $-1$ and thus is not an element of $H$, so that the matrix must be of the form $\diag{s,t,t^{-1},s^{-1}}$.
The matrix $\psi_{h_{a},h_{b}}$ is diagonal if and only if $h_{a}$ and $h_{b}$ are diagonal, and if $h_{a}=\diag{a,a^{-1}}$ and $h_{b}=\diag{b,b^{-1}}$ then $\psi_{h_{a},h_{b}}=\diag{ab,ab^{-1},a^{-1}b,a^{-1}b^{-1}}$. There exists $a,b\in\bQ_{p}^{\times}$ such that $\psi_{h_{a},h_{b}}=\diag{s,t,t^{-1},s^{-1}}$ if and only if $st$ is a square. As $\left|\bQ_{p}^{\times}/(\bQ_{p}^{\times})^{2}\right|=4$ for $p>2$ we get that the index of the image subgroup is 4.
\end{proof}

In view of this it suffices to work with $\on{SO}(2,2)(\bQ_{p})^{+}$ from now on. Before we calculate the volume, let us make the following remark.

\begin{lemma}
\label{prevstructure}
$\psi$ preverses the Cartan decomposition in the following sense. Let $m,n\geq 0$, then the image of
$$\SLZ[2]\times \SLZ[2] \bigl\{(a_p^m, a_p^n),(a_p^n,a_p^m)\bigr\}\SLZ[2]\times\SLZ[2]$$ is $$Ka_p^{m+n} b_p^{|m-n|}K\cap \on{SO}(2,2)(\bQ_{p})^{+}.$$
\end{lemma} 
\begin{proof}
	It is easy to see that $\psi_{g,h}\in K$ if $g,h\in\SLZ[2]$. 
	Because of this we next take the image of $(a_p^m,a_p^n)$. 
Calculating the matrix representation of $\psi_{a_p^m,a_p^n}$ we see that it corresponds to the tensor product $a_p^m\otimes a_p^n$ with eigenvalues $p^{\pm m\pm n}$. This shows that the image of $\SLZ[2]a_p^m\SLZ[2]\times \SLZ[2]a_p^n\SLZ[2]$ is contained in $Ka_p^{m+n} b_p^{|m-n|}K $.
Since this holds for every $m,n\geq 0$
and since 
the element $a\in A_+$ in the Cartan decomposition in Proposition \ref{KAKrank2}
is unique, the lemma follows.
\end{proof}

\begin{lemma}\label{ranktwovolume}
If we normalise the Haar measure of $\on{SO}(2,2)(\bQ_{p})^{+}$ such that $m(\on{SO}(2,2)(\bZ_{p})^{+})=1$ then
\[
m\left(\partial B^{\on{SO}(2,2)(\bQ_{p})^{+}}_{\ell}\right)=
(2(p+1)p+(p+1)^2(\ell-1))p^{2\ell-2} \mbox{ for }\ell \geq 1.
\]
In particular, for any $\varepsilon > 0$ we have $p^{2\ell}\leq m\left(B^{\on{SO}(2,2)(\bQ_{p})^{+}}_{\ell}\right)\ll_{\varepsilon} p^{(2+\varepsilon)\ell}$.
\end{lemma}
\begin{proof}
Let $m_{\on{SL}}$ be the Haar measure on $\on{SL}_{2}(\mathbb{Q}_{p})$, which we normalise so that $m_{\on{SL}}(\on{SL}_{2}(\mathbb{Z}_{p})) = 1$. Since the kernel of $\psi$ is contained in $\on{SL}_{2}(\mathbb{Z}_{p})\times \on{SL}_{2}(\mathbb{Z}_{p})$, a disjoint union of $\on{SL}_{2}(\mathbb{Z}_{p})\times \on{SL}_{2}(\mathbb{Z}_{p})$ cosets remain disjoint in the image. Thus by Lemma \ref{prevstructure}, $\partial B^{\on{SO}(2,2)(\bQ_{p})^{+}}_{\ell}$ is the image of
\begin{align*}
\bigcup_{m+n=\ell} & \SLZ[2]\times \SLZ[2] \bigl\{(a_p^m, a_p^n),(a_p^n,a_p^m)\bigr\}\SLZ[2]\times\SLZ[2]\\
= \bigsqcup_{m=0}^{\ell} & \SLZ[2]\times \SLZ[2] (a_p^m, a_p^{(\ell-m)})\SLZ[2]\times\SLZ[2].
\end{align*}
Using the coset decomposition for the summands, we see by Proposition \ref{pserretree} that for $\ell \geq 2$,
\begin{align*}
m\left(\partial B^{\on{SO}(2,2)(\bQ_{p})^{+}}_{\ell}\right) &=
(p+1)p^{2\ell-1}+\sum_{m=1}^{\ell-1}(p+1)^2p^{2m-1}p^{2(\ell-m)-1}+(p+1)p^{2\ell-1}\\
&=\bigl((p+1)p+(p+1)^2(\ell-1)+(p+1)p\bigr)p^{2\ell-2}\\
&\leq\bigl(4p^2+4p^2(\ell-1))p^{2\ell-2}=4\ell p^{2\ell}\ll p^{(2+\varepsilon)\ell}
\end{align*}
since~$p+1\leq 2p$ and~$\ell\ll_\varepsilon 2^{\ell\varepsilon}\leq p^{\ell\varepsilon}$.
The lemma follows from this easily.
\end{proof}


\section{Notation and Tools from Homogeneous Dynamics}
\label{sec-notation}

\subsection{Injectivity radius and small neighborhoods}
\label{smallneighborhood}

Depending on the point $x\in\Gamma\backslash\bG(\bQ_S)$ the map $g\mapsto g\acts x$ is injective on the ball $B_{r}^{d_\infty }\times \bG(\bZ_p)$ for sufficiently small $r$. The supremum over such $r$ is called \textit{injectivity radius} at $x$ and can be bounded from below in terms of the height, see e.g.\ \cite[Equation (7.3)]{EMMV})
\begin{lemma}[Relationship of injectivity radius and height]
\label{injrad}
There exists $\consta\label{expinht:injrad}>0$ such that for all $x \in\Gamma\backslash \bG(\bQ_S)$ the map $g\mapsto g\acts x$ is injective on $$\left\{g=(g_{\infty},g_{p})\in\bG(\bQ_S): d_\infty(g_\infty,e)\ll \on{ht}(x)^{-\ref{expinht:injrad}}, g_p\in\bG(\bZ_p)\right\}.$$
\end{lemma}

In particular it follows that every point~$x\in X_\text{cpt}$ has injectivity radius at least $p^{-\ref{exp:uniforminjrad}}$ for some constant $\consta\label{exp:uniforminjrad}>0$.
For the following we fix the neighborhood $$\Omega_S=\Omega_\infty\times \bG_{\on{joint}}(\bZ_p)^+$$ in $\bG_{\on{joint}}(\bQ_S)^+$ where $\Omega_\infty$ is an open set such that $\Omega_S\ni g\mapsto xg$ is injective for all $x\in X_\text{cpt}$. Furthermore, we also want to assume that
 $\Omega_\infty\Omega_\infty^{-1}\Omega_\infty\Omega_\infty^{-1}\times \bG_{\on{joint}}(\bZ_p)^+$ is injective in that sense.
By the above this holds if $$\Omega_\infty= \{g\in G_{{\on{joint}},\infty}: d_\infty(g,e)\leq p^{-\ref{exp:uniforminjrad}}/4\}.$$ 
Assuming that the Riemannian metric on~$\mathfrak{sl}_{d-1}(\bR)$ is invariant
under the adjoint action of~$\SO[d-1](\bR)$ we obtain that $\Omega_\infty$ 
is invariant under conjugation 
by all elements of the compact subgroup $\SO[d-1](\bR)$. 
We note that~$\Omega_\infty$ depends on~$p$ and so also on~$D$.

\subsection{Normalization of measure on ambient space}
\label{ambientspace}
The natural measures $m_{\mathcal{Y}^+_i}$ are taken to be probability measures and the Haar measures $m_{G_{i,S}^+}$ are normalized to be compatible with respect to the projections $G_{i,S}^+\to\mathcal{Y}^+_i$. 
The analoguous normalization of $m_{G_{{\on{joint}},S}^+}$ gives rise to $m_{G_{{\on{joint}},S}^+}=m_{G_{1,S}^+}\times m_{G_{2,S}}$. Since $\mathcal{Y}_2=\Gamma G_{2,\infty}\times\bG_2(\bZ_p)$, the set $F\times\bG_2(\bZ_p)$ is a fundamental domain for $\mathcal{Y}_2$ if $F$ is a fundamental domain for $\bG_2(\bZ)\backslash\bG_2(\bR)$. Indeed, any $(g_\infty,g_p)\in G_{2,S}$ decomposes into $$(g_\infty,g_p)=(g_\infty,\gamma h)=\gamma(\gamma^{-1}g_\infty,h)=\gamma({\gamma'}f,h)=\gamma{\gamma'}(f,{h'})$$ where $h,{h'}\in \bG_2(\bZ_p)$, $f\in F$, $\gamma\in\bG_2(\bZ[\tfrac1p])$ and ${\gamma'}\in\bG_2(\bZ)$. Thus $m_{G_{2,S}}$ is normalized such that $m_{G_{2,S}}(F\times\bG_2(\bZ_p))=1$ and $m_{G_{2,\infty}}(F)=m_{G_{2,p}}(\bG_2(\bZ_p))=1$ would give natural choices on how do that.

\subsection{Definition of Volume}
\label{orbitmeasures}
Next we deal with orbits of closed unimodular subgroups $H<G_{{\on{joint}},S}^+$ to which we attach a \textit{volume} as done in \cite[Section 2.3]{EMMV}. Consider a finite volume orbit $\Gamma g_1Hg_2$ in $\mathcal{Y}^+_{\on{joint}}=\Gamma\backslash G^+_{\on{joint},S}$.
The set $\Gamma g_1Hg_2$ is an orbit of the group $H^{g_2}$ and we may rewrite this orbit as $\Gamma g_1g_2H^{g_2}$. It suffices therefore consider the case $\Gamma g H$, where we call~$H$ the acting subgroup. For $g \in {G}^{+}_{\on{joint},S}$, the orbit $\Gamma gH$ is naturally identified with
$$X_g=\on{Stab}_{H}(\Gamma g)\backslash H=(H\cap \Gamma^{g})\backslash H$$ and is equipped with an $H$-invariant probability measure $m_{X_g}$. We may assume that the \textit{orbit measure} $m_{\Gamma gH}$ of the orbit $\Gamma gH$ is the push-forward of $m_{X_g}$ under the isomorphism $X_g\ni \on{Stab}_{H}(\Gamma g)h\mapsto \Gamma gh$. The Haar measure $m_{H}$ of $H$ is now normalized to be compatible with $m_{X_g}$ under the natural projection. We define the \textit{volume} of $\Gamma gH$ to be
$$
\on{vol}(\Gamma gH)=(m_{H}(\Theta))^{-1},
$$
where~$\Theta=\Theta_\infty\times \bG_{{\on{joint}}}(\bZ_p)$
and $\Theta_\infty\subset G_{{\on{joint}},\infty}^+$ is a fixed 
precompact open neighbourhood 
of the identity element. Clearly, the volume notion $V=V_\Theta$ depends on the 
choice of $\Theta_\infty$ but for any other fixed choice $\Theta'_\infty$ we 
have $V_\Theta\ll V_{\Theta'}\ll V_\Theta$ by precompactness 
(see \cite[Section 2.3]{EMMV}). We will assume that $\Theta_\infty$ 
is invariant under conjugation by $H_\infty$ and such that~$\Theta$
intersects trivially with $\Gamma$ and this is still true for $\Theta_{2}=\Theta\Theta^{-1}$. 
We note that the volume of the ambient space $\mathcal{Y}_2$
is independent of~$p$ (since~$\on{SL}_{d-1}$ is simply connected), 
and that the volume of $\mathcal{Y}_1^+$ (and of $\mathcal{Y}_{\on{joint}}^+$) are bounded from above and below by some constants independent of $p$ since the corresponding adelic orbit $\bG_1(\bQ)\backslash\bG_1(\bA)$ is compact and so is a finite union of $\bG_1(\bR\times\prod_{p'}\bZ_{p'})$-orbits. 

In the context of the orbit~$\Gamma\bH_v(\bQ_S)(k_v,e)$ and~$\Gamma {L}_{v,S}^{+}\left(k_{v},e,\theta_{v},e\right)$
the acting group is $\SO[d-1](\bR)\times\bH_v(\bQ_p)$ resp.\
a diagonally embedded copy of~$\SO[d-1](\bR)\times\bH_v(\bQ_p)$. 
Since $\Theta_\infty\cap\SO[d-1](\bR)$ can be covered by at most $\ll p^{\on{dim}(\SO[d-1])\ref{exp:uniforminjrad}}$ many translates of $\Omega_\infty\cap\SO[d-1](\bR)$ (and contains at least $\gg p^{\on{dim}(\SO[d-1])\ref{exp:uniforminjrad}}$ many disjoint translates), $V_{\Omega}\asymp p^{\ref{exp:massofomega}}V_{\Theta}$ where $p^{-\consta\label{exp:massofomega}}=p^{-\on{dim}(\SO[d-1])\ref{exp:uniforminjrad}}$ is (up to a scalar multiple) the Haar measure of $\Omega_\infty\cap\SO[d-1](\bR)$ with respect to the Haar measure on~$\SO[d-1](\bR)$.

Let us note the following lemma which will be helpful.

\begin{lemma}\label{comparingvolumes}
 Let~$\Gamma'<\Gamma$ be a finite index subgroup of a lattice $\Gamma<G_S$. Let $H<G_S$ be a subgroup 
that has a finite volume orbit~$\Gamma g H$ and suppose that $H'\lhd H$
has finite index. Then the ratio of the volume of $\Gamma gH$ and the volume of $\Gamma'gH'$
is bounded from above and below by some constants that depend only on the set $\Theta$ used in the definition of the volumes, the index $[\Gamma:\Gamma']$, and $[H:H']$.
If $\Theta$ is invariant under conjugation by a subgroup $K\subset G_S$, then the volume of $\Gamma g H$ equals the volume of $\Gamma g H k=\Gamma g k H^k$.
\end{lemma}

\begin{proof}
	Let us first compare the volumes of $\Gamma g H$ and $\Gamma gH'$.
	For this notice that the Haar measure of $H'$ can be obtained by restricting the Haar measure of $H$ to $H'$.
	More precisely, we normalize the Haar measure $m_H$ on $H$ to be compatible with the orbit measure on $\Gamma gH$
	and the Haar measure $m_{H'}$ on $H'$ with the probability orbit measure on $\Gamma gH'$.
	If now $\Gamma gH=\Gamma gH'\sqcup\Gamma g h_2H'\sqcup\cdots\sqcup \Gamma g h_\ell H'$
	is the decomposition of the $H$-orbit into disjoint $H'$-orbits, then $\ell\leq[H:H']$
	and multiplying $\Gamma gH'$ on the right by $h_i$
	gives $\Gamma gh_iH'$ since $H'\lhd H$. This already gives $m_{H'}=\ell m_H|_{H'}$ and
	\[
	 m_{H'}(\Theta)=\ell m_H(\Theta\cap H')\leq [H:H']m_H(\Theta),
	\]
	which gives the first inequality between the volumes by taking inverses. 
	
	For the converse we let $h_1=e,h_2,\ldots,h_{[H:H']}\in H$
	be a complete set of representatives of the equivalence modulo $H'$ and obtain
	\[
	 \Theta\cap H=\Theta\cap (H'\sqcup h_2H'\sqcup\ldots\sqcup h_{[H:H']}H').
	\]
	For a given $h_i$ the intersection $\Theta\cap (h_iH')$ could of course be empty. However, 
	if it is not empty then there exists some $h'\in H'$ with $h_ih'\in\Theta\cap (h_iH')$
	and so also 
	\[
	 m_H(\Theta\cap(h_iH'))=m_H((h_ih')^{-1}(\Theta\cap(h_iH')))\leq m_H((\Theta^{-1}\Theta)\cap H').
	\]
	This now gives
	\[
	 m_H(\Theta)\leq [H:H']m_{H'}(\Theta^{-1}\Theta).
	\]
	However, as the ratio of two notions of volume defined using $\Theta$ resp. $\Theta^{-1}\Theta$
	can be bounded from above and below by constants the second inequality between the volumes follows.
	
	 Switching from $\Gamma$ to $\Gamma'$ we note that $[\Gamma^g\cap H:\Gamma'{}^g\cap H]\leq[\Gamma:\Gamma']$
	is the factor by which the normalization of the Haar measure on~$H$ changes if we study 
	the orbit~$\Gamma 'gH$ instead of $\Gamma gH$. Together with the above, this gives the first claim in the lemma.

	Assume now that $\Theta$ is invariant under $K$ conjugation then $\Theta\cap H^k=(\Theta\cap H)^k$. If $m_H$ is the compatible Haar measure for $\Gamma gH$ then $k_*m_H$ is the compatible Haar measure of $H^k$ for $\Gamma gkH^k$ where $k_*$ denotes the push forward under the conjugation map. Measuring the volume gives
	\[
	m_{H^k}(\Theta)=k_*m_{H}(\Theta)=m_{H}(k\Theta k^{-1})=m_H(\Theta).
	\]
\end{proof}

\subsection{Lie algebras - First encounter}
\label{liefacts}
We recall notation and facts presented in \cite[Section 6.5]{EMMV}. We let $K=\bG_{\on{joint}}(\bZ_p)^+$ and $$K[m]=\{k\in K: \|k-e\|_p\leq p^{-m}\} \mbox{ and }\Omega[m]=\Omega_\infty\times K[m].$$
Denote by $\goh_1=\lie{\HvS[p]}$ respectively $\goh_2=\lie{\HVS[p]}$ the Lie algebras obtained when projecting to the first respectively second factor. By semi-simplicity there exists invariant complements $\gor_1$ and $\gor_2$ of $\goh_1$ respectively $\goh_2$ such that $$\gog_1=\goh_1\oplus \gor_1\mbox{ and }\gog_2=\goh_2\oplus \gor_2.$$ We say that $\gor_i$ is \textit{undistorted} if $$\gog_i[m]=\goh_i[m]\oplus \gor_i[m]$$ for all $m$. Here, $V[m]=\{X\in V(\bQ_p):\|X\|_p\leq p^{-m}\}$ for any subspace $V\subset\gog$. We will be able to find undistorted complements in Section \ref{liealgebras} using that $\HvS[p]$ and $\HVS[p]$ are~$\bZ_p$-conjugates to the model groups.

The exponential map is an isometry on $\gog[1]$ if $p>2$, and in fact maps sub algebras of
the form~$\gog[m]$ to subgroups. In fact,
$$
K[m]=\exp{\gog[m]}\text{ for all } m\geq 1.
$$ 
This also implies that $\exp(\cdot)$ is measure preserving up to a scalar. From the implicit function theorem one also obtains the following decomposition lemma.
\begin{lemma}[Decomposition, {\cite[Lemma 6.5]{EMMV}}]
\label{decomposition} We have
$$(\HvS[p]^+\cap K[m])\exp{\gor_1[m]}\times (\HVS[p]^+\cap K[m])\exp{\gor_2[m]}=K[m]\mbox{ for all }m>0.$$
\end{lemma}
See Lemma \ref{hensel} for a related statement.

We let $H<G_S=G_{i,S}^+$ denote the acting group $\HvS^{+,k_v}$ resp.~ $\HVS^{\oT}$ for the orbit $\Gamma k_v\HvS^+$ respectively $\Gamma \oT\HVS$ in $\Gamma G_S$ and note that $H_\infty=\on{SO}_{d-1}$. We cite the following lemma, which can easily be deduced from the property that the exponential map is measure preserving.
\begin{lemma}[Adjustment Lemma {\cite[Lemma 6.6]{EMMV}}]
\label{adjustmentlemma} Given two subsets $A_1,A_2\subset K[1]\cap H_p$ of relative measure $>\tfrac12$ and $g\in K[m]$, there exist $\alpha_i\in A_i$ so that $\alpha_1^{-1}g\alpha_2=\exp{r}$ for some $r\in\gor[m]$. 
\end{lemma}

\subsection{Pigeon Hole Principle}
The following will give points that do not lie on the same local orbit, i.e.\ we obtain two nearby points $x$, $y=xg$ on the same $H$-orbit with smallest displacement $g\not\in H$. For a set $\cN\subset G_S$, define the doubled sets $\cN_2=\cN\cN^{-1}$ and $\cN_4=\cN_2\cN_2$. We let $\mu$ denote $\mu_{v,S}$ on $\mathcal{Y}^+_{\on{joint}}$ or ${\pi_i}_*\mu_{v,S}$ on $\mathcal{Y}^+_i$
and let $V$ denote its associated volume defined using $\Theta$.

\begin{lemma}[Pigeon Hole Principle, {\cite[Lemma 7.6]{EMMV}}]
\label{transversalpoints}
Suppose that $E\subset \Gamma\backslash G_S$ is a measurable set with $\mu(E)>\frac34$.
Let $\cN\subset G_S$ be open and assume that $\cN_4\subset \Omega[1]$ and $m_{G_S}(\cN)>2V^{-1}$.
Then there exist $x,y\in E$, so that $y=xg\in x\cN_4$ and $g \in \cN_4\setminus H$.
\end{lemma}
\begin{proof}
We follow \cite{EMMV} verbatim. Let $\{x_i:1\leq i\leq I\}$ be a maximal set of points in $X_\text{cpt}$ such that $x_i\cN$ are disjoint. Then $I\leq m_{G_S}(\cN)^{-1}$ and $X_\text{cpt}\subset \bigcup_ix_i\cN_2$ and therefore there is some $i_0$ so that 
$$\mu(x_{i_0}\cN_2\cap E)\geq\frac{1}{2I}.$$
Let $y_1\in x_{i_0}\cN_2\cap E$ then any $y_2\in x_{i_0}\cN_2\cap E$ is of the form $y_1g$ where $g\in\cN_4$. Suppose contrary to the lemma that this implies $g\in H$. That is, $y_2$ always is on the same local $H$-orbit of $y_1$ in the sense that $y_2\in x_{i_0}\cN_2\cap E\subset y_1(H\cap\cN_4)$. Thus
\begin{multline*}
m_{G_S}(\cN)\leq\frac{1}{I}\leq2\mu(x_{i_0}\cN_2\cap E)\leq2\mu(y_1(H\cap\cN_4))\\
\leq2m_H(H\cap \Omega[1])\leq2m_{H}(\Theta)=2V^{-1}
\end{multline*}
contradicting the assumption $m_{G_S}(\cN)>2V^{-1}$.
\end{proof}

\subsection{Stabilizer Lemma}
The following results incorporate \cite[Lemma 2.2]{EMMV}, which identifies the stabilizer group $\stab{\mu}$ of the orbit measure $\mu$ of $H$ as a subset of the normalizer $N_{G_S}(H)$ of $H$ in $G_S$. We will show
in our case that small elements of $N_{G_S}(H)$ necessarly lie in $H$ and consequentially we
will be able to bypass the need of \cite[Section 5.12]{EMMV} (and the work of Borel and Prasad \cite{Borel-Prasad})
which shows that the orbit associated to $\stab{\mu}$ also has large volume.
\begin{lemma}
\label{centralizer}
The normalizer $N_{\on{SO}_d(\bR)}(\on{SO}_{d-1}(\bR))$
consists of all~$g\in\on{SO}_d(\bR)$ that satisfy~$e_{d}g=\pm e_d$. 
Moreover, $N_{\on{SL}_{d-1}(\bR)}(\on{SO}_{d-1}(\bR))=\on{SO}_{d-1}(\bR)$.
\end{lemma}
\begin{proof}
	Let~$n\in N_{\on{SO}_d(\bR)}(\on{SO}_{d-1}(\bR))$. If~$e_d n\neq \pm e_d$,
	then there exists some~$k\in \on{SO}_{d-1}(\bR)$ such that~$e_dnk\neq e_dn$.
	However, this implies~$e_dnkn^{-1}\neq e_d$ and equivalently~$nkn^{-1}\notin \on{SO}_{d-1}(\bR)$ which contradicts the definition of~$n$. 

For the second case, we want to use the real Cartan decomposition $H_\infty A_\infty H_\infty$ of $\SL_{d-1}(\bR)$. 
From it, we deduce immediately that if $g=k_1ak_2$ normalizes $H_\infty$ then $a=\diag{a_1,\dots,a_{d-1}}$ must
normalize~$H_\infty$ too. Using e.g.\ the Lie algebra elements corresponding
to rotations in planes we see that~$a$ must satisfy $a_i/a_j=1$ for all~$1\leq i<j\leq {d-1}$. 
This forces $a=e$ if we insist 
that $A$ only contains positive diagonal matrices as we may.
\end{proof}

\begin{corollary}[No purely real transversal displacement]
\label{stabilizer}
Assume that two $u_t$-generic points $x,y\in\Gamma g_0H$ (for $g_0\in\{k_v,\oT\}$ respectively) satisfy $xg=y$ with  $g=g_\infty\in B^{d_\infty}_{1/2}$, then we have $g\in H$.
\end{corollary}
\begin{proof}
The assumptions imply that $g\in\stab{\mu}$. Indeed, since $g$ commutes with $u_t$ then for any continuous $f\in C_c(\Gamma\backslash G_S)$
\begin{multline*}
\left|\mu(f)-\mu^g(f)\right|=\left|\lim_{\ell\to\infty}\fint_{B_\ell} f(xu_t)dt-\lim_{\ell\to\infty}\fint_{B_\ell} f(xu_tg)dt\right| \\
=\left|\lim_{\ell\to\infty}\fint_{B_\ell} f(xu_t)dt-\lim_{\ell\to\infty}\fint_{B_\ell} f(yu_t)dt\right|=\left|\mu(f)-\mu(f)\right|=0,
\end{multline*}
where $\fint_{B_\ell} dt$ denotes the normalized integral over the ball of radius $p^\ell$ in $\bQ_p$ with respect to the Haar measure determined by $|\cdot|_p$.

However, $g\in \stab{\mu}$ implies $\Gamma g_0Hg=\Gamma g_0H$ or equivalently $\Gamma^{g_0}Hg= \Gamma^{g_0}H$ so that also $g\in \Gamma^{g_0}H$. Let the corresponding decomposition be $g=\gamma h_1$ with $h_1\in H$ and $\gamma\in \Gamma^{g_0}$ then $\gamma H\gamma^{-1}=g Hg^{-1}\subset\Gamma^{g_0}Hg^{-1}=\Gamma^{g_0}H$. The connected Hausdorff component of the identity in the set  $\Gamma^{g_0}H\subset G_S$  is $H_\infty$ and consequentially $\gamma H _\infty\gamma^{-1}\subset H _\infty$. By definition
we obtain $g\in \left(N(H_\infty)\cap\bG(\bZ[\tfrac1p])^{g_0}\right)H$. 
Therefore $g\in \left(N_{G_\infty}(H_\infty)\cap\bG(\bZ[\tfrac1p])^{g_0}\right)H_\infty$, and we can apply Lemma~\ref{centralizer}. This concludes the discussion already for $G_S=G_{2,S}$ because in that case $N(H_\infty)=H_\infty$. For the first factor, we need the assumption that $g$ is small: If $g$ belongs to the normaliser but not to~$H_\infty$,
then the $(d,d)$-coefficient of $g$ is $-1$ and thus $g$ cannot lie in $B^{d_\infty}_{1/2}$.
\end{proof}

\subsection{Torsion free lattice}
\label{lattices}
To avoid technical complications in the next section we show that there is a torsion free finite index subgroup $$\Gamma_4=\Gamma\cap(\Gamma'_4\times\Gamma'_4)\mbox{ with } \Gamma'_4=\on{ker}\left(\SLZP\to\on{SL}_d(\bZ[\tfrac{1}{p}]/4\bZ[\tfrac{1}{p}])\right).$$
\begin{lemma}\label{torsionfree}
$\Gamma'_4$ is a torsion-free subgroup of index bounded by $4^{d^2}$ in $\SLZP$.
\end{lemma}
\begin{proof}
We embed $\SLZP$ diagonally in $\on{SL}_d(\bQ_p)\times \on{SL}_d(\bZ_2)$. Then $\Gamma'_4=\Gamma\cap \SLQ\times U_4$ where $U_4=\{g\in\on{SL}_d(\bZ_2):g-e\in \on{Mat_d}(4\bZ_2)\}$ is a subgroup of $\on{SL}_d(\bZ_2)$ of index $\leq 4^{d^2}$ without torsion since $U_4$ is clopen and does not contain finite subgroups. The latter follows  since the logarithm map is well-defined on $U_4$ and the Lie algebra of~$\on{SL}_d(\bQ_2)$ has no torsion (see also Section~\ref{singlefactorgenericpoints}).
\end{proof}


\section{Volume and Discriminant}
\label{sec:discvol}
The aim of this section is to relate a geometric invariant (the volume) and an arithmetic  invariant (the discriminant $D$) of the orbit $\Gamma\LvS^+ (k_v,e,\theta_v,e)$. We will show that 
\begin{equation}
\label{eq:voldisc}
D^*\ll\on{vol}\left(\Gamma\LvS^+ (k_v,e,\theta_v,e)\right)\ll D^*,
\end{equation}
where $D^*$ denotes a power of $D$ with some absolute exponent (and we allow different powers on the left and the right).
We will deduce this from the corresponding statement of $\Gamma \HvS^+ k_v$. For $\Gamma \HVS^+ \theta_v$ we will be content with a mere lower bound since the upper bound will only be needed for the joint orbit to deduce equidistribution from the single orbit cases. The methods used here are outlined for real quotients in \cite[Sect.~17]{EMV}
and replaces for our special case 
the use of Prasad's volume formula \cite{Prasad} in \cite{EMMV}.

For the lower bound we begin by noting that there exists a lattice element bounded in norm in terms of the volume by simply comparing the growth of a ball to the size of the fundamental domain. We will take advantage of the concrete definition of the group $\Hv=\on{Stab}_{\on{SO}_d}(v)$, so that these lattice elements must satisfy the ($S$-adic) integer equation $\gamma v=v$. Finding sufficiently many different such $\gamma$, $v$ is the unique common eigenvector and gives restriction on the size of $v$ (and so on $D$) from above. To find such elements, we harvest our preparations from Section \ref{discriminants}: $\Hv(\bZ_p)$ is open and therefore $\Hv(\bQ_p)/\Hv(\bZ_p)$ is a discrete set endowed (in the rank one cases) with the structure of a regular tree (Section \ref{quasisplitgeometry}). The \textit{volume} of the quotient $\Hv(\bZ[\tfrac1p])\backslash \Hv(\bQ_p)/\Hv(\bZ_p)$ - a finite graph - agrees with the number of its vertices. 

\subsection{Existence of small lattice elements in $\HvS$}
\label{sec:latticepoints}
We will exploit the geometric structure we obtained for the model groups to which any $\HvS[p]$ is $\bZ_p$-conjugated to (say by an element $h_v$) by Corollary \ref{classificationofH}. The important property is that $\bZ_p$-points are mapped to $\bZ_p$ under this conjugation, so that we also obtain a decomposition of $\HvS[p]$ into $K=\Hv(\bZ_p)$-cosets. It is further mapping
the norm balls in the one group to the same norm balls in the other group, implying that 
also the tree structure introduced in Section \ref{quasisplitgeometry}
is invariant under the conjugation. Geodesics as described in Lemma \ref{geodesics}, will correspond to the image of the Cartan group $A$ (of the model group) conjugated by $h_v$.

To establish lower bounds for the volume of the orbit $\Gamma \HvS k_v$, we start by establishing the existence of sufficiently many small lattice elements. Recall that by Lemma \ref{torsionfree}, $\Gamma_{4}$ has finite index in $\Gamma$ and is torsion free, and we will formulate the following for $\Gamma_4$ until we eventually bound $V_4=\on{vol}\left(\Gamma_{4} H_{v,S}k_{v}\right) \asymp V=\on{vol}\left(\Gamma H_{v,S}k_{v}\right)$ (see Lemma \ref{comparingvolumes}).

\begin{proposition}
\label{latticepoints}
\begin{enumerate}
\item If $H_{v,p} = \on{SO}(2,1)(\mathbb{Q}_{p})^{h_{v}}$ then there exist $\gamma_{1},\gamma_{2}\in \Gamma_4\cap \HvS$ such that their common eigenspace of eigenvalue $1$ is spanned by $v$,
\[
\|\gamma_i\|_\infty\ll1 \mbox{ and } \|\gamma_{i}\|_p\leq p^4 V_4^2 \mbox{ for } i=1,2.
\]
\item If $\HvS[p]=\on{SO}_\eta(3,1)(\bQ_p)^{h_v}$ then there exist $\gamma_i\in \Gamma_4\cap \HvS$ for $i=1,2,3$ whose common eigenspace for the eigenvalue $1$ is spanned by $v$,
\[
\|\gamma_i\|_\infty\ll1 \mbox{ and } \|\gamma_{i}\|_p\leq p^4V_{4}^2 \mbox{ for } i=1,2,3.
\]
\item If $\HvS[p]^+=\on{SO}(2,2)(\bQ_p)^{+,h_v}$ then there exists a lattice element $\gamma\in\Gamma_4\cap\HvS$ that only fixes the span of $v$ with
\[
\|\gamma\|_\infty\ll1 \mbox{ and } \|\gamma\|_p\ll p^2V_{4}.
\]
\end{enumerate}
\end{proposition}

Notice that in the split case it is more convenient to work with the subgroup $\HvS[p]^+$ and we will see that it suffices to find a single element  $\gamma=\gamma_1=\gamma_2=\gamma_3$.

\subsubsection{Proof of Proposition~\ref{latticepoints}, part (1)}
We give a detailed proof of the first part of Proposition~\ref{latticepoints}, and shall regard the method of finding lattice points in tree-like graphs in a quantitatively fashion to be at our disposal thereafter. The Haar measure associated to the orbit $$\Gamma_4 \HvS k_v=\Gamma_4 k_v\HvS^{k_v}=\Gamma_4 k_v (H_\infty\times H_{v,p})$$ is $m_{\HvS^{k_v}}=m_{H_\infty}\times m_{\HvS[p]}$ satisfying $m_{H_\infty}(\Theta_\infty)m_{\HvS[p]}(K)=V_4^{-1}$. By Corollary~\ref{rankonevolume},
$$m_{\HvS^{k_v}}(\Theta_\infty\times B_{\ell}^{H_{p}})=\left(1+(p^{\ell}-1)\frac{p+1}{p-1}\right)V_4^{-1}>p^{\ell}V_4^{-1}.$$ We  choose 
the integer $\ell\geq 1$
minimal such that $p^\ell V_4^{-1}>1$. This implies that $\Theta_\infty\times B_{\ell}^{H_{p}} $ cannot be an injective set
when projected to $\Gamma_4   k_v(H_\infty \times H_{v,p})$, that is, there exists $\gamma\in\Gamma_4\setminus\{e\}$, $h_1,h_2\in \Theta_\infty\times B_{\ell}^{H_{p}} $ such that $\gamma k_vh_1=k_vh_2$. 
 Thus we found $e\neq\gamma\in (\Theta_\infty)_2\times B_{2\ell}^{\HvS[p]}\cap (\HvS\cap\Gamma_4)$ for $\ell\leq \log_{p}(V_4)+1$,
and the latter gives
\[
\Vert\gamma\Vert_p\leq p^2 V_4^2.
\]

Since $\Hv(\bR)\cong\SO[3](\bR)$ is compact, we also see that $\gamma$ must be diagonalizable
over the algebraic closure of $\bQ$. Moreover, recall that any special rotation in $\bR^3$ has an axis of rotation, i.e. an eigenvector for eigenvalue one. Hence, it follows that the eigenvalues of $\gamma$ are $\lambda,\lambda^{-1},1$ for some algebraic number $\lambda$. Since $\gamma\in\Gamma_4$ and $\Gamma_4$ is torsion-free and discrete in $G_S$, it also follows that $\vert\lambda\vert_p\neq 1$. However, this implies that $\lambda$ and $\lambda^{-1}$ cannot be Galois conjugated over $\bQ_p$, which shows that $\lambda\in\bQ_p$.
Since a maximal $\bQ_p$-split torus in a simple algebraic group over $\bQ_p$ is unique up to conjugation, there exists some $g\in H_{v,p}$ that conjugates $\gamma$ to the element $h_v^{-1}\on{diag}(\lambda,\lambda^{-1},1)h_v$. 
(This shows unfortunately also that $\gamma$ by itself does not yet satisfy the statement of the proposition
since we consider $\bH_v\cong\on{SO}(2,1)$ as a subgroup of $\on{SL}_{4}$ and so $\gamma$ has a two-dimensional eigenspace for
the eigenvalue $1$.)

We set $\gamma_1=\gamma$ and wish to apply the same argument to find a different lattice element satisfying
almost the same estimate. For this we define the Dirichlet set
\[
F_{1}=\set{h\in \HvS^{+,k_v}\cap \Theta: d(hK,K)<d(hK,\gamma_{1}^{\ell}K)\text{ for all }\ell\in\bZ\setminus\set{0}}
\]
for the group $\langle\gamma_{1}\rangle $ (and the origin $K$).
Recall that $\gamma_1$ is acting on the $p+1$-regular tree by a translation along a certain geodesic within the tree.
There are two cases to consider. 

It could be that the geodesic  goes through $K$, in which case $F_1$ might be as small as $K$ and $p-1$ rooted trees
branching out at $K$ (with the remaining two branches leading to $\gamma_1 K$ and $\gamma_1^{-1}K$) but it has to contain
at least these $p-1$ rooted trees.

It could be that the geodesic does not go through $K$, in which case one of the branches starting at $K$ leads to that geodesic (and some 
of the points of this branch may belong to $F_1$)
and the remaining $p$ branches give $p$ rooted trees that completely belong to $F_1$. 

Hence in the worst case $F_1$ contains the set $F_1'$ consisting of $K$ and of $p-1$ out of the $p+1$ branches
out of $K$. This shows that
we have
\[
m_{\HvS^{k_v}}\left(\Bigl(\Theta_\infty\times B_{\ell+1}^{H_{p}}\Bigr)\cap F_1'\right)=\left(1+(p-1)p^{\ell}\right)V_4^{-1}>p^{\ell}V_4^{-1}>1
\]
and so that $\Bigl(\Theta_\infty\times B_{\ell+1}^{H_{p}}\Bigr)\cap F_1'$ cannot be an injective set
for the projection  to $\Gamma_4   k_v(H_\infty \times H_p)$. Hence
there exists $\gamma_2\in\Gamma_4\setminus\{e\}$, $h_1'\neq h_2'\in (\Theta_\infty\times B_{\ell+1}^{H_{p}})\cap F_1'$ 
	such that $\gamma_2 k_vh_1'=k_vh_2'$. It follows just as before that $\Vert\gamma_2\Vert_2\leq p^4V_4^2$ and 
	that $\gamma_2$ is diagonalizable over $\bQ_p$. 
	
Suppose that $\gamma_1$ and $\gamma_2$ have the same common eigenvector $w\in v^\perp$ for the eigenvalue $1$. 
This shows that $\gamma_2$ maps $\{v,w\}^\perp$ into itself. This forces $\gamma_1$ and $\gamma_2$ to commute
which in turn implies that they act by translation along the same geodesic on the $p+1$-regular tree.
However, this is impossible as any translation along this geodesic maps $F_1'$ to a disjoint set
and $\gamma_2$ maps an element of $F_1'$ back to $F_1'$. This concludes the proof of the first part of Proposition~\ref{latticepoints}. \qed

\subsubsection{Proof of Proposition~\ref{latticepoints}, part (2)}

Let $K$ denote the standard compact subgroup of $\HvS[p]$. Let $\mathbb{L}<\bH_v$ be a subgroup 
	that is $\bZ_p$-conjugated to $\on{SO}(2,1)<\on{SO}_\eta(3,1)$. In particular we may
	consider the $p+1$-regular tree $L_p/(L_p\cap K)$ inside the $p^2+1$-regular tree $\HvS[p]/K$. 
	Applying the same argument as in the proof of the first part of the proposition
	for the set $(B_\ell^{\HvS[p]}\cap L_p)K$
	we find an element $\gamma_1\in\Gamma\setminus\{e\}$ with $\Vert\gamma_1\Vert_p\leq p^2 V_4^2$. 
	As in that proof $\gamma_1$ must be diagonalizable over the algebraic closure of $\bQ_p$
	with at least one eigenvalues of absolute value bigger than one and one with absolute value
	smaller than one. 	We may assume that $v_1,v_2,v_3,v_4$ are the eigenvectors
	with the eigenvalues $\lambda_1,\lambda_2,\lambda_3,\lambda_4$ respectively. If all eigenvalues are different from $1$,
	then we simply set $\gamma_1=\gamma_2=\gamma_3$. So assume now that $\lambda_4=1$,
	which implies that we may choose $v_4\in\bQ_p^4$ and that $\gamma_1$ lies in a subgroup that is conjugated to
	$L_p$ (see Lemma \ref{oneconjugateclass}). 
	
	Hence we may apply the second part of the argument above
	to find $\gamma_2$ with the estimate $\Vert\gamma_2\Vert_p\leq p^4 V_4^2$. If $\gamma_1$ and $\gamma_2$ do not
	have a common eigenvector for eigenvalue $1$ 
	(other than $v$) we set $\gamma_3=\gamma_2$. Hence we are reduced to the
	case where $\gamma_1,\gamma_2$ belong to a subgroup conjugated to $L_p$. In this case
	there exists an embedded $p+1$-regular subtree inside our $p^2+1$-regular tree so that both $\gamma_1$
	and $\gamma_2$ preserve that subtree, act via certain isometries on the subtree, and all remaining vertices
	move isometrically along. The following argument for finding $\gamma_3$ is quite similar to the	second step finding $\gamma_2$.  In fact, we consider two cases and define immediately a subset $F_2'$
	of an appropriate Dirichlet set. 

The subtree could contain $K$ and in that case we define $F_2'$
to consist of $K$ and of the $(p^2+1)-(p+1)=p^2-p$ branches at $K$ that do not belong to the subtree.
Note that  any $\gamma\in\langle\gamma_1,\gamma_2\rangle$ preserves the subtree. Moreover any
nontrivial element of $\Gamma$  that preserves the subtree must be moving $K$
to a different vertex in that subtree, which implies $\gamma F_2'\cap F_2'=\emptyset$. 

Or the subtree could be disjoint to $K$ and in that case we define $F_2'$ to consist of $K$ and the $p^2$-many branches
at $K$ that point away from the subtree. As before  $\gamma F_2'\cap F_2'=\emptyset$
for any nontrivial $\gamma\in\Gamma$ that preserves the subtree. 

We now calculate
\begin{multline*}
 m_{\HvS^{k_v}}\left(\Bigl(\Theta_\infty\times B_{\ell+1}^{H_{p}}\Bigr)\cap F_2'\right)\geq
 \left(1+ (p^2-p) \frac{p^{\ell+1} -1}{p^2-1}\right)V_4^{-1}\\
=\left(1+ p\frac{p^{\ell+1} -1}{p+1}\right)V_4^{-1}>\frac12p^{\ell+1} V_4^{-1}>p^{\ell} V_4^{-1} >1.
\end{multline*}
This implies the existence of a $\gamma_3$ that does not preserves the subtree and
satisfies the estimate
\[
\Vert\gamma_3\Vert_p\leq p^4 V_4^2. 
\]
The three elements $\gamma_1,\gamma_2,\gamma_3$ cannot have any  common eigenvector
for eigenvalue $1$ other than $v$ since the common eigenspace for $\gamma_1,\gamma_2$ was $2$-dimensional
and if $\gamma_3$ has the same eigenspace it would belong to the same conjugate of $L_p$
and so preserve the same subtree. \qed

\subsubsection{Proof of Proposition \ref{latticepoints}, part (3)}
In Section \ref{kak-split} we constructed a surjective group homomorphism $\psi$ from $\on{SL}_2(\bQ_p)\times\on{SL}_2(\bQ_p)$ to $\on{SO}(2,2)(\bQ_p)^{+}$ preserving the Cartan decomposition by Lemma \ref{prevstructure}. In particular, we can identify the quotient $$\SLQ[2]/\SLZ[2]\times\SLQ[2]/\SLZ[2]$$ with $\on{SO}(2,2)(\bQ_p)^+/\on{SO}(2,2)(\bZ_p)^+$. 
This  gives $\HvS[p]^+/K$ the structure of a product of two graphs as in Proposition \ref{pserretree}. Also by Lemma \ref{comparingvolumes} we may switch from studying $\Gamma_4H_{v,S}k_v$ to studying $\Gamma_4 H_{v,S}^+k_v$ and know that the volume
changes at most by an absolutely bounded multiplicative factor. Below we let $V_+$
denote the volume of $\Gamma_4 H_{v,S}^+k_v$.

We now construct a single lattice element belonging
to the image of $\SLQ[2] \times \SLZ[2]$. In fact we consider
the image $B$ of $B_\ell^{\SLQ[2]}\times\SLZ[2]$ which satisfies
\[
 m_{\HvS^{+,k_v}}(\Theta_\infty\times B)=\left(1+\tfrac{p}{p-1}(p^{2\ell}-1)\right)V_+^{-1}
\geq p^{2\ell}V_+^{-1}.
\]
We now choose $\ell$ such that the latter is $>1$ which implies the existence of
some  $\gamma\in\Gamma_4\cap (B^2)$
that is in the image of $(g_1,g_2)\in \SLQ[2] \times \SLZ[2]$ with $\Vert g_1\Vert_p\leq p^{2\ell}$. 
Recall that the Jordan-Chevalley decomposition of
elements in algebraic groups is uniquely determined
and well behaved under algebraic homomorphisms. 
Hence, if we consider the Jordan-Chevalley decomposition of $g_1$ and $g_2$
and recall that the kernel of $\psi$ is finite, then we see that the image of $(g_1,g_2)$ cannot be diagonalizable unless $g_1$ and $g_2$ itself are diagonalizable.
By construction, the eigenvalues of $g_1$ are $\lambda_1,\lambda_1^{-1}$ 
with $1\leq|\lambda_1|_p\leq p^{2\ell}$ and the eigenvalues of $g_2$ are $\lambda_2,\lambda_2^{-1}$
with $|\lambda_2|_p=1$. 
However, by the properties of the homomorphism $\psi$ in Section \ref{kak-split} the eigenvalues of $\gamma$ are $\lambda_1^{\pm1}\lambda_2^{\pm1}$. By the same argument as in the proof of the first part of the proposition
one of the eigenvalues must be of norm $\neq 1$, which implies $|\lambda_1|_p>1$ (and a fortiori that $g_1$ is diagonalizable over $\bQ_p$)
and so that none of the eigenvalues of $\gamma$ is equal to $1$.
Finally $\Vert\gamma\Vert_p\leq\Vert g_1\Vert_p\leq p^{2\ell}\leq p^2V_+\ll p^2V_4$. \qed

\subsection{Lower bound for $\vol{\Gamma \HvS^+ k_v}$}
\label{volVSdisc}
We can now establish the first relation between $\|v\|^2=D$ and the volume of the orbit $\Gamma \HvS^+ k_v$. By Lemma \ref{comparingvolumes}, the volume $V$ of $\Gamma \HvS^+ k_v$ is bounded by above and below in terms of the volume of $\Gamma_4 \HvS k_v$ and we therefore may bound the lattice elements $\gamma_i$ from the previous section with $\|\gamma_i\|_S\ll p^4V^2$.
\begin{proposition}
\label{discvol1}
There exists $\consta\label{exp:discvol1} > 0$ such that
\[
D^{\ref{exp:discvol1}} \ll \vol{\Gamma \HvS^+ k_v}.
\]
\end{proposition}
\begin{proof}
Using Proposition~\ref{latticepoints} we find $d-2$ elements $\gamma_{1}, \dots, \gamma_{d-2} \in \Gamma\cap\SLQ$ satisfying $\left\|\gamma_{i}\right\|_{S} \ll p^4 V^2$ for $i=1,\dots ,d-2$
such that the common eigenspace for eigenvalue $1$ is spanned by $v$ only.  Consider now the following system of $d(d-2)$ linear equations
\(
(\gamma_{i} - e)v'=0
\)
for an undetermined vector $v'$. 
As $v$ is (up to scalar multiplication) the only solution by construction, the system has rank $d-1$. Pick $d-1$ linear independent rows $r_j$ of this system. 
Now form the square matrix $R$ consisting of these rows and another row with indeterminant entries $t_1,\ldots,t_d$. Taking the determinant of $R$ and expanding it in terms of the coefficients of the last row we get a nontrivial linear expression $a_1t_2+\cdots a_d t_d$, where
the coefficients $a_j$ are minors of $R$. Note that the determinant of $R$ vanishes if we set the last row equal to one of the vectors $r_j$. Hence $v'=(a_1,\ldots,a_d)^T$ is a solution to the above linear equations and hence must be equal to $v$ up to a scalar.

As $\Vert\gamma_i\Vert_\infty\leq 1$ for all $i$,
it is easy to see that $\Vert v'\Vert\ll 1$. For the $p$-adic norm we have $\Vert\gamma_i\Vert_p\ll p^4 V_4^2$ and so we obtain from the above $\Vert v'\Vert_p \ll (p^4 V_4^2)^{d-1}$. 
For all other primes $v'$ is integral. Hence 
the primitive vector $v$ is up to sign the vector $\Vert v'\Vert_p v'$
and has norm $\ll (p^4 V_4^2)^{d-1}$.  
Using $\Vert v\Vert^2=D$ and $|p|\ll_\varepsilon D^\varepsilon$ by Proposition \ref{prime_existence}
the proposition follows.
\end{proof}

\subsection{Upper bound for $\vol{\Gamma \HvS^+ k_v}$}
The following uses that integers are one apart. 
\begin{proposition}
\label{integersareoneapart}
There exists $\consta\label{exp:integersareoneapart}>0$ such that
$$\vol{\Gamma \HvS k_v}\ll D^{\ref{exp:integersareoneapart}}.$$
\end{proposition}

\begin{proof}
	Since the notion of volume changes at most a bounded amount if we change the precompact open neighborhood $\Theta$ that is used to define $\vol{\Gamma \HvS k_v}$, we may as well assume that $\Theta=G_{1,\infty}\times \bG_1(\bZ_p)$. Using in addition that $\Gamma_4$ has finite index in $\Gamma$ (as in Lemma \ref{comparingvolumes}) we see that $\vol{\Gamma \HvS k_v}\asymp N$, where $N$ is the number of disjoint $H_\infty\times\bH_v(\bZ_p)$-orbits in $\Gamma \HvS k_v$.
	
	By the discussion in Section \ref{principalgenus} (especially \eqref{pr-finite-M} and Proposition \ref{SupportOfNu}) the number of disjoint $H_\infty\times\bH_v(\bZ_p)$-orbits in $\Gamma G_{1,\infty}\times \bG_1(\bZ_p)$ is bounded by 
	the number of integer points on the sphere
	of radius $\sqrt{D}$, i.e. by $\ll D^{\frac{d-1}2}$. 
	
	The number of disjoint $H_\infty\times\bH_v(\bZ_p)$-orbits in $\Gamma g_p G_{1,\infty}\times \bG_1(\bZ_p)$
	for different $g_p\in\bG_1(\bQ_p)$ can be handled in the same way as these correspond to integer points on certain ellipsoids. Since the adelic quotient $\bG_1(\bQ)\backslash \bG_1(\mathbb{A})$ is compact,
	it is a finite union of orbits of $\bG_1(\bR\times\prod_{p'}\bZ_{p'})$. Hence there are only finitely
	many ellipsoids to consider, which gives the lemma with $\ref{exp:integersareoneapart}=\frac{d-1}2$.
\end{proof}

\subsection{Upgrading to $\Gamma L_{v,S}^+ (k_v,e,\theta_v,e)$}

\begin{proposition}
\label{discvoljoint}
\label{integersareoneapart2}
The joint orbit measure satisfies
$$D^{\ref{exp:discvol1}} \ll \on{vol}\left(\Gamma L_{v,S}^+ (k_v,e,\oT,e)\right)\ll D^{\ref{exp:integersareoneapart}}.$$
\end{proposition}
\begin{proof}
We claim that $\vol{\Gamma\widetilde{L}_{v,S}^+(k_v,e,g_vk_va_v,e)}=\vol{\Gamma\HvS^+(k_v,e)}$.  
For this first notice that the image of the Haar measure $m_{\HvS^{+,k_v}}$ on $\HvS^{+,k_v}=H_\infty\times H_{v,p}^+$ under the push forward of $(h_\infty,h_p)\mapsto (h_\infty, h_p, h_\infty,h_p^{g_v^{-1}})$ defines a Haar measure on 
the acting group for the orbit $\Gamma\widetilde{L}_{v,S}^+(k_v,e,g_vk_va_v,e)$ (see Section \ref{orbitsofsubgroups}
for the notation $\widetilde{L}_{v}$). Moreover,
this map is consistent with the map 
\begin{multline*}
	 \Gamma h(k_v,e)\in\Gamma \HvS^+(k_v,e)\mapsto\\\Gamma(h(k_v,e),(g_v,g_v) h (k_va_v,g_v^{-1}))\in \Gamma\widetilde{L}_{v,S}^+(k_v,e,g_vk_va_v,e),
\end{multline*}
which induces the normalized Haar measure on the second orbit from the normalized Haar measure of the first. 
Finally we may assume that the set $\Theta\subset G_{1,S}\times\on{SL}_d(\bR\times\bQ_p)$ has the form $\Theta_{1,\infty}\times\bG_1(\bZ_p)\times\Theta_{2,\infty}\times\on{SL}_d(\bZ_p)$ and satisfies $\Theta_{1,\infty}\cap H_\infty=\Theta_{2,\infty}\cap H_\infty$.  Together
with the definition of the volume, this gives the claim.

On the other hand it is clear that projecting
the orbit $\Gamma\widetilde{L}_{v,S}^+(k_v,e,g_vk_va_v,e)$ within the quotient corresponding to $\on{ASL}_{d-1}$ to the orbit $\Gamma {L}_{v,S}^+(k_v,e,\theta_v,e)\subset\mathcal{Y}_2$, volume can only decrease. Hence,
we obtain
\[
 \vol{\Gamma\widetilde{L}_{v,S}^+(k_v,e,\oT',e)}\geq \on{vol}\left(\Gamma L_{v,S}^+ (k_v,e,\oT,e)\right) \geq \vol{\Gamma\HvS^+(k_v,e)}, 
\]
where the last inequality
follows by the same argument using the projection to~$\mathcal{Y}_1$.
Hence the result follows from Propositions \ref{discvol1}--\ref{integersareoneapart}.
\end{proof}

\subsection{A variant calculation for $\Gamma \HVS \oT$}
\label{volumeshape}
We now modify the previous arguments of Section~\ref{sec:latticepoints} and Section~\ref{volVSdisc} to calculate a lower bound for the volume of the orbit $\Gamma \HVS \oT$. We have shown that certain small lattice elements $\{\gamma_i\}$ of $\Hv=\on{Stab}_{\bG_1}(v)$ have $\langle v\rangle$ as their common eigenspace for eigenvalue $1$. 
The group $\HV=\overline{g_v\Hv g_v^{-1}}$ on the other hand is the orthogonal group of the quadratic form associated to $A_v=(g_v^Tg_v)_{i,j<d}$ (\cite[Equation (3.3)]{AES}), that is, it is the stabilizer subgroup of $\bG_2$ acting on the space $\on{Sym}_{d-1}$ 
of symmetric matrices by $\gamma\mapsto\Psi_{\gamma}$ where $\Psi_{\gamma}A=\gamma A \gamma^T$ for $A\in\on{Sym}_{d-1}$. We show that again there exist small lattice elements (in terms of the volume) $\{\gamma_i\}$ whose unique fixed point is $A_v$. 
This is implied if the group generated by $\{\gamma_i\}$ is Zariski dense since the special orthogonal group determines the orthogonal form uniquely up to scalar (see e.g.~\cite[Lemma 3.3]{AES}).
This is the formulation of the effective Borel-Wang density theorem as discussed in \cite[Section 17.3]{EMV}.

Let us begin by noting that the proof of Corollary \ref{classificationofH} also gives the same result for $\HVS$, i.e.
the group $\HVS[p]$ is conjugate over $\bZ_p$ to one of the model groups of Section~\ref{classification}.
In particular, we can again use the same geometric structures.

\begin{proposition}
\label{latticepoints_secondfactor}
In all cases $H_{\Lambda_v,p}$ (isomorphic to $\on{SO}(2,1)(\mathbb{Q}_{p})$, $\on{SO}_\eta(3,1)(\mathbb{Q}_{p})$, 
or $\on{SO}(2,2)(\mathbb{Q}_{p})$) there exists elements
$\gamma_1,\ldots,\gamma_4\in \Gamma_4\cap H_{\Lambda_v,S}$
with
\[
 \|\gamma_{i}\|_p\ll p^*\on{vol}(\Gamma \HVS \oT)^* \mbox{ for } i=1,\ldots,4.
\]
such that $\langle\gamma_1,\ldots,\gamma_4\rangle$ is Zariski-dense in $\bH_{\Lambda_v}$. 
\end{proposition}

\begin{proof}
Since the acting group at $\infty$ is $H_\infty$ (and in particular not changing) the volume of $\Gamma \HVS \oT$
is $\asymp N$, where $N$ is the number of disjoint $H_\infty\times H_{\Lambda_v,p}$-orbits needed to cover $\Gamma \HVS \oT$. In other words, the situation again reduces to a purely $p$-adic problem.

	Next we claim that in the cases of $\on{SO}(2,1)(\mathbb{Q}_{p})$ and $\on{SO}_\eta(3,1)(\mathbb{Q}_{p})$ the elements that we found by the argument in the proof of Proposition \ref{latticepoints} (conjugated by $g_v$) already satisfy the Zariski density stated in the proposition. 
	
	For this let $\mathbb{L}$ denote the Zariski-closure of the group generated by $\gamma_1,\gamma_2$, (and $\gamma_3$).
	Since $\mathbb{L}<\bH_v$ is defined over~$\bQ$ and $\bH_v$
	is anisotropic over $\bQ$, it follows that $\mathbb{L}$ cannot have a unipotent radical.
	
	In the case of $\on{SO}(2,1)(\mathbb{Q}_p)$ the proof of Proposition \ref{latticepoints} gives two diagonalizable non-commuting elements $\gamma_1,\gamma_2$ with eigenvalues that do not have $p$-adic absolute value one and belong to $\bQ_p$. However, this implies that the reductive non-abelian subgroup $\mathbb{L}$ has $\bQ_p$-rank one and forces $\mathbb{L}=\bH_v$.
	
	In the case of $\on{SO}_\eta(3,1)(\mathbb{Q}_{p})$
	we found 3 matrices $\gamma_1,\gamma_2,\gamma_3$ that each have at least one eigenvalue in $\bQ_p$ of absolute value not equal to $1$ and that together do not
	have a common eigenvector for eigenvalue 1 other than $v$. 
	We again see that the reductive group $\mathbb{L}$ has $\bQ_p$-rank one. By the structure of reductive subgroups of $\on{SO}_\eta(3,1)(\mathbb{Q}_{p})$
	(which is similar to the more well known statement for $\SL_2(\bC)$ and follows also from the finite-dimensional representation theory of $\mathfrak{sl}_2$) it follows that $\mathbb{L}$ is either a torus, 
	conjugated to $\on{SO}(2,1)(\bQ_p)$, or all of  $\on{SO}_\eta(3,1)(\mathbb{Q}_{p})$. By the eigenvalue statement we see that we must be in the last case. 
	
	As the last case is similar to the cases considered before, we will only sketch it. So suppose  $H_{\Lambda_v,p}$ is isomorphic to $\on{SO}(2,2)(\bQ_p)$. We again consider instead $\on{SL}_2(\bQ_p)\times\on{SL}_2(\bQ_p)$ and argue as in the third part of the proof of Proposition \ref{latticepoints} to find a nontrivial element $\gamma_1\in\Gamma_4$ that is the image of some $(g_1,g_2)\in\on{SL}_2(\bQ_p)\times\on{SL}_2(\bZ_p)$ and with $g_1$ being diagonalizable over $\bQ_p$ with eigenvalues not of $p$-adic absolute value one. As the geometric structure of $\on{SL}_2(\bQ_p)\times\on{SL}_2(\bZ_p)/\SL_2(\bZ_p)\times\SL_2(\bZ_p)$ is very similar to that of a tree (see Proposition \ref{pserretree} for the precise structure) we can argue as in the first part of the proof of Proposition \ref{latticepoints} to find
	a second element $\gamma_2$ that is the image of $(g_1',g_2')\in\on{SL}_2(\bQ_p)\times\on{SL}_2(\bZ_p)$ of the same type so that $g_1$ and $g_1'$ do not commute. Repeating the argument using the subgroup $\on{SL}_2(\bZ_p)\times\on{SL}_2(\bQ_p)$
	we find $\gamma_3,\gamma_4$ in the same way. All of these elements satisfy the same type of estimate on their $p$-adic norm. 
	
	Let $\mathbb{L}<\SL_2\times\SL_2$ be the Zariski-closure
	of the preimage of the group generated by $\gamma_1,\ldots,\gamma_4$. As above $\mathbb{L}$ is defined over $\bQ$ and cannot have a unipotent radical, i.e. $\mathbb{L}$ is a reductive $\mathbb{Q}$-group. 
	By Chevalley's theorem there exists a representation $\rho$ of $\SL_2\times\SL_2$ and a vector $v$ in the associated representation space $V_\rho$ such that $\mathbb{L}$ is the stabilizer of the line spanned by $v$. 
	Let~$(g_1,g_2)\in\on{SL}_2(\bQ_p)\times\on{SL}_2(\bZ_p)$ be as in the construction
	of~$\gamma_1$. 
	As $(g_1,e)$ and $(e,g_2)$ commute, $\rho(g_1,e)$ and $\rho(e,g_2)$ commute and are simultaneously diagonalizable by the Jordan decomposition within algebraic groups. 
	If $\lambda_i,\lambda_i^{-1}$ are the eigenvalues of $g_i$ for $i=1,2$, 
	then the eigenvalues of $\rho(g_1,e)$ are also powers of $\lambda_1$ 
	and similarly for $\rho(e,g_2)$. 
	Since $\lambda_1^{m_1}\lambda_2^{m_2}=\lambda_1^{n_1}\lambda_2^{n_2}$ 
	implies $m_1=n_1$ and $\lambda_2^{m_2}=\lambda_2^{n_2}$ for any $m_1,m_2,n_1,n_2\in\bZ$, 
	it follows that any sum of joint eigenvectors that is an eigenvector for $\rho(g_1,g_2)$ 
	must also be an eigenvector of $\rho(g_1,e)$ and of $\rho(e,g_2)$.  
	For the vector $v$ this shows that $(g_1,e),(e,g_2)\in\mathbb{L}$ and similarly 
	for $\gamma_2,\gamma_3,\gamma_4$. This implies $\mathbb{L}=\SL_2\times\SL_2$ as in the case of $\on{SO}(2,1)(\bQ_p)$ considered above.  
\end{proof}

\subsection{Lower bound for $\vol{\Gamma \HVS^+ \theta_v}$}

\begin{proposition}
\label{discvol2}
There exists $\consta\label{exp:discvol2} > 0$ such that
\[
D^{\ref{exp:discvol2}} \ll \on{vol}(\Gamma \HVS^+ \oT).
\]
\end{proposition}
\begin{proof}
Let $\gamma_1,\ldots,\gamma_4$ be as in Proposition \ref{latticepoints_secondfactor}. 
We consider the system of linear equations on the space of symmetric matrices $\on{Sym}_{d-1}$ 
(where each element $A$ induces a quadratic form) given by $\Psi_{\gamma_i}A=A$ for $i=1,2,3,4$ 
where $\Psi_{\gamma_i}$ denotes the natural action of $\gamma_i\in\SL_{d-1}$ on the space of quadratic forms, i.e. composing the quadratic forms with $\gamma_i$.

If $A$ is a $\bQ$-solution then the group $\HV$ 
stabilizes $A$ since the group generated by the four lattice elements is Zariski dense by Proposition \ref{latticepoints_secondfactor}. 
Since the special orthogonal group determines its quadratic form uniquely up to scalar multiples (\cite[Lemma 3.4]{AES}), 
the matrix $A$ must be a $\bQ$-multiple of $A_v=(g_v^Tg_v)_{i,j<d}$ that defines $\HV=\overline{g_v\Hv g_v^{-1}}$ (\cite[Equation 3.3]{AES}). 
We see from this that the system of linear equations $(\Psi_{\gamma_i}-e)A=0$, $i=1,2,3,4$, has rank $r=\dim(\on{Sym}_{d-1})-1$.

We continue just as before 
(except for a somewhat 
 unexpected complication). We may take $r$ linear independent rows and argue as in the proof of Proposition~\ref{discvol1} 
to obtain a solution $A$ with coefficients in $\bZ[\tfrac1p]$
satisfying the estimate $\|A\|_p\ll p^*V^*$ by the properties of $\gamma_1,\ldots,\gamma_4$.
However, here we have the additional difficulty that we do not yet know a bound for 
\[
 s=\max_{i=1,\ldots,4}\|\gamma_i\|_\infty
\]
(since here $\gamma_i\in\bH_{\Lambda_v}(\bZ[\frac1p])$ 
does not belong to $H_\infty=\on{SO}_{d-1}(\bR)$ and the conjugation takes place within the noncompact group $\SL_{d-1}(\bR)$). Ignoring the value of $s$ for now, we obtain $\|A\|_\infty\leq s^*$ by using again 
the argument that defines $A$ as the vector of certain minors of the above linear equations. 
Multiplying $A$ by its denominator $\|A\|_p$
we obtain an integer solution $A'$ with $\|A'\|_\infty\ll s^*p^*V^*$. We may assume that $A'$ is primitive (as otherwise we simply divide $A'$ by its common divisor and obtain the same estimate for the new matrix).
By \cite[Lemma 3.3]{AES} the quadratic form $\phi_v$ attains the value $1$ on the integers, 
which implies that $A_v$ is primitive as a $(d-1)\times(d-1)$ matrix. 
Therefore, $A'=\pm A_v$ and taking the determinant 
(see Section \ref{geometry}, \cite[Lemma 3.3]{AES}) we obtain
\[
 D=\det A_v\ll s^*p^*V^*,
\]
which becomes utterly useless we have a reasonable estimate for $s$ (e.g. $s\ll D^\varepsilon$ for some small $\varepsilon$ would suffice). 

The relation between $\gamma_i$ and $H_\infty$ is quite simple: the lattice elements $\gamma_i$ are conjugated via $\theta_v$ to elements of $H_\infty$. Also recall that $\theta_v$ is the matrix that describes the shape of the lattice $\Lambda_v$. Hence in an attempt to optimize $s$ we therefore do not choose any basis $w_1,\ldots,w_{d-1}$ of $\Lambda_v$ in the definition of $g_v$ as in Section \ref{orbitsofsubgroups}, but instead use the Minkowski basis of $\Lambda_v$. However, as the covolume of $\Lambda_v$ is $\sqrt{D}$ this may still result after normalizing the covolume e.g.\ in one vector of length $D^{-1/2(d-1)}$ and one vector of length $D^{1/2(d-1)}$ (which would not give a sufficiently strong estimate for $s$).  
To overcome this problem we note that this problematic case corresponds to a lattice rather deep in the cusp.

So assume for the moment that the lattice corresponding to $v$ belongs to $X_\text{cpt}$. This shows that (after normalizing the covolume) we have the upper estimate $\ll p^*$ for the length of the Minkowski basis vectors and so $\max(\Vert\theta_v\Vert_\infty,\Vert\theta_v^{-1}\Vert_\infty)\leq p^*$. Using this estimate we obtain $s\ll p^*$
and hence
\[
 D=\det A_v\ll p^*V^*
\]
Using now that $p\ll D^\varepsilon$ for a sufficiently small $\varepsilon>0$ by our choice of $p$ in Proposition \ref{prime_existence} we can rephrase the last estimate to give us $D^*\ll V$.

Of course, for the given vector $v$ it may not be true that the
lattice corresponding to $v$ belongs to $X_\text{cpt}$. However, by the estimate after Lemma \ref{nondivergence} we know that there are many points in the orbit of $\Gamma H_{\Lambda_v,S}^+(\theta_v,e)$ that belong to $X_\text{cpt}$. 
Suppose therefore that $\Gamma (\theta_v,h_p)\in X_\text{cpt}$ and $h_p=\gamma^{-1}h_p'$ for some $h_p\in H_{\Lambda_v,p}^+$, $\gamma\in\SL_{d-1}(\bZ[\frac1p])$ and $h_p'\in\SL_{d-1}(\bZ_p)$. 
This shows that $\tilde{A}=\gamma A_v\gamma ^T =h_p'A_vh_p'{}^T$ is both rational with only $p$ in the denominator and a $p$-adic integer, i.e.\ $\tilde{A}$ is integral. 
Reversing the equation we also show that if $\tilde{A}$ is divisible by some $k\geq 2$, then the original matrix $A_v$ is also divisible by $k$, hence $\tilde{A}$ is a primitive integral matrix. 
Clearly the determinant of $A_v$ and the determinant of $\tilde{A}$ agree, so it sufficies to estimate $\det\tilde{A}$. It is also clear that the orthogonal group of $\tilde{A}$ equals $\gamma \HV \gamma^{-1}$ and that
\[
 \Gamma H_{\Lambda_v,S}^+(\theta_v,e)=\Gamma H_{\Lambda_v,S}^+(\theta_v,h_p)=\Gamma \gamma\HV(\bR\times\bQ_p)^+\gamma^{-1}(\gamma\theta_v,h_p'). 
\]
The last expression is quite similar to our original description of the orbit,  except that we consider the orthogonal group of $\tilde{A}$ and use the initial point $\Gamma(\gamma\theta_v,h_p')\in X_{\text{cpt}}$. Moreover, we have obtained the same assumptions for the new data, we have the same determinant and the same volume (as we did not change the orbit but only expressed it in a different form). Now we can apply the above arguments and the proposition follows.
\end{proof}


\section{Lie Algebras and Invariant Complements}
\label{liealgebras}
We continue our discussion of invariant complements from Section \ref{liefacts}. Recall that $\goh_1=\lie{\HvS[p]}\subseteq\gog_1$ and $\goh_2=\lie{\HVS[p]}\subseteq\gog_2$
and that we say that $\gor_i$ is an \textit{undistorted invariant complement} of $\goh_i$ 
within $\gog_i$ (for $i=1,2$) 
if $$\gog_i[m]=\goh_i[m]\oplus \gor_i[m]$$ for all $m\geq0$ and $\gor_i$ is invariant under the adjoint action of $\goh_i$. Having fixed the principal $\on{PGL}_{2}$ over $\bQ_p$, 
whose Lie algebra we denote by $\lsl<\goh_i$, we will
define in this section $\gor_i$ and decompose it into  irreducible subspaces for the adjoint represenation
of the principal $\lsl$. We note that this will also give
us the invariant complement  $\gog_1\oplus\gor_2$ of the Lie algebra  $\gol_v=\lie{\mathbb{L}_{v,p}}$.
 
The knowledge which irreducible components appear in the two invariant complements
will be used for the joint equidistribution. It will be sufficient to understand the Lie algebras of the model groups for this purpose. 

 For $r\in\gor_i$ we denote by $r^{\on{hw}}$ and $r^{{\on{lw}}}$ the projection of $r$ onto the heighest respectively lowest weight space of the \emph{largest} dimensional irreducible  representation (there will be exactly one). 
We will finish by finding a finite set of group elements $\mathcal{F}$ in $\HvS[p]$ respectively $\HVS[p]$ of norm one, such that for every $r\in\gor_i$ there exists $g\in\mathcal{F}$ with $\|(\Ad[g]r)^{{\on{lw}}}\|_p=\|\Ad[g]r\|_p=\|r\|_p$. The significance of this is that the highest degree term of the polynomial $\Ad[u_t](\Ad[g]r)$ 
in the variable $t$ (appearing prominently 
in the dynamical argument of the following sections) will then also have one of the largest coefficient.

\subsection{Decomposition of $\goh_i$ into  irreducible subrepresentations}\label{irreduciblesection}
For convenience of the calculation we take the model groups associated to the quadratic forms $-2xz+y^2$, $-2xw+y^2+\eta z^2$ and $2xw-2yz$ which are all equivalent over $\bZ_p$ to those used to define $\on{SO}(2,1)$, $\on{SO}_\eta(3,1)$ and $\on{SO}(2,2)$ in Section \ref{discriminants}. We begin by giving a basis of each of
$\mathfrak{so}(2,1)$, $\mathfrak{so}_\eta(3,1)$, and of $\mathfrak{so}(2,2)$, and by picking
a principal $\lsl$-triple. 
After this we decompose the Lie algebras with respect to the adjoint action of $\lsl$. 

Recall that for each natural number $n$ there exists a unique  representation $V^{(n)}$ 
of $\on{SL}_2$ of dimension $n+1$ and highest weight $n$ under the torus in $\SL_{2}$.

\medskip 
\paragraph{\textbf{Lie algebra $\mathfrak{so}(2,1)$}}
 
The Lie algebra of $\on{SO}(2,1)$ (defined using $-2xz+y^2$) is determined by solving $XJ+JX^T=0$ with $J=\left[\begin{smallmatrix}0&0&-1\\0&1&0\\-1&0&0\end{smallmatrix}\right]$.
Hence
$$\mathfrak{so}(2,1)=\lsl=\on{span}\left\{H=\left[\begin{smallmatrix}2&0&0\\0&0&0\\0&0&-2\end{smallmatrix}\right],\ 
 X=\left[\begin{smallmatrix}0&1&0\\0&0&1\\0&0&0\end{smallmatrix}\right],\  
 Y=\left[\begin{smallmatrix}0&0&0\\2&0&0\\0&2&0\end{smallmatrix}\right]
\right\}.$$

\medskip
\paragraph{\textbf{Lie algebra $\mathfrak{so}_\eta(3,1)$}} 
The following diagram shows the representation $\mathfrak{so}_\eta(3,1)$ with one $V^{(2)}$ at the top and the other on the bottom. The action by $M$ flips the highest and lowest weight vectors respectively. Another $\lsl$-triple $(X_2, H, Y_2)$ has been added.
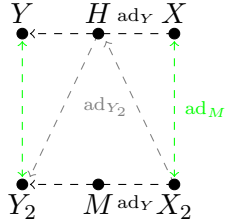
\begin{center}
\begin{tikzpicture}

  \draw[black, dashed, thin, ->] (2,0) -- (0.1,0);
  \draw[] (1.5,2) node[above] {{\tiny\( \ad[Y]\)}};
  \draw[black, dashed, thin, ->] (2,2) -- (0.1,2);
  \draw[] (1.5,0) node[below] {{\tiny\( \ad[Y]\)}};

  \draw[green, dashed, <->] (0,1.9) -- (0,0.1) ;

  \draw[green, dashed, <->] (2,1.9) -- (2,0.1);
  \draw[green] (2,1) node[right] {{\tiny\( \ad[M]\)}};

  \draw[gray, dashed, thin, ->] (1,2) -- (0.1,0.1);
  \draw[gray, dashed, thin, ->] (2,0) -- (1.1,1.90);
  \draw[gray] (1.1,1.3) node[below] {{\tiny\( \ad[Y_2]\)}};

\filldraw (2,2) circle (2pt) node[align=left,   above] {\(X\)};
\filldraw (1,2) circle (2pt) node[align=left,   above] {\(H\)};
\filldraw (0,2) circle (2pt) node[align=left,   above] {\(Y\)};

\filldraw (2,0) circle (2pt) node[align=left,   below] {\(X_2\)};
\filldraw (1,0) circle (2pt) node[align=left,   below] {\(M\)};
\filldraw (0,0) circle (2pt) node[align=left,   below] {\(Y_2\)};

\end{tikzpicture}
\end{center}

The Lie algebra $\mathfrak{so}_\eta(3,1)$ of $\SO[\eta](3,1)$ (defined using 
the quadratic form $-2xw+y^2+\eta z^2$) 
is determined by solving the equation $XJ_\eta+J_\eta X^T=0$ where $J_\eta=\left[\begin{smallmatrix}0&0&0&-1\\0&1&0&0\\0&0&\eta&0\\-1&0&0&0\end{smallmatrix}\right]$.
This leads to the following basis of $\mathfrak{so}_\eta(3,1)$, starting with the principal $\lsl$-triple
$$ H=\left[\begin{smallmatrix}2&0&0&0\\0&0&0&0\\0&0&0&0\\0&0&0&-2\end{smallmatrix}\right],\  X=\left[\begin{smallmatrix}0&1&0&0\\0&0&0&1\\0&0&0&0\\0&0&0&0\end{smallmatrix}\right],\  Y=\left[\begin{smallmatrix}0&0&0&0\\2&0&0&0\\0&0&0&0\\0&2&0&0\end{smallmatrix}\right] $$
and the additional matrices
$$ M=\left[\begin{smallmatrix}0&0&0&0\\0&0&-2&0\\0&2\eta&0&0\\0&0&0&0\end{smallmatrix}\right],\  X_2=\left[\begin{smallmatrix}0&0&1 &0\\0&0&0&0\\0&0&0&\eta\\0&0&0&0\end{smallmatrix}\right],\  Y_2=\left[\begin{smallmatrix}0&0&0&0\\0&0&0&0\\ -2\eta&0&0&0\\0&0&-2 &0\end{smallmatrix}\right] .$$
We note that $M,X_2,Y_2$ are eigenvectors for $\ad[H]$ and eigenvalues $0,2,-2$ respectively.
The relations 
\[
	\ad[Y]X_2=-M,\ \ad[Y]M=2Y_2,\ 
	\ad[X]Y_2=M,\   \ad[X]M=-2X_2 
\]
show that $X_2$ generates another copy of $V^{(2)}$. The relation $\ad[Y_2]X_2=\eta  H$
shows that $H,X_2,Y_2$ would (after rescaling $Y_2$ say) also be a choice of an $\lsl$-triple. Finally we note
the action of $M$ on these elements:
\[
	\ad[M]X=2X_2,\ \ad[M]X_2=-2\eta X, 
	\ad[M]Y=-2Y_2,\  \ad[M]Y_2=2\eta  Y.
\]

\medskip
\paragraph{\textbf{Lie algebra $\mathfrak{so}(2,2)$}} 
This diagram shows the two $V^{(2)}$-representations of the diagonal $\lsl$. We also draw the action of the two $\lsl$-factors on the element $H$ of the principal $\lsl$.
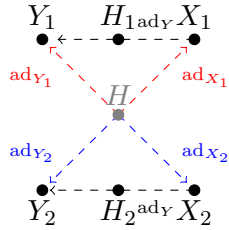
\begin{center}
\begin{tikzpicture}

  \draw[black, dashed, thin, ->] (2,0) -- (0.1,0);
  \draw[] (1.5,2) node[above] {{\tiny\( \ad[Y]\)}};
  \draw[black, dashed, thin, ->] (2,2) -- (0.2,2);
  \draw[] (1.5,0) node[below] {{\tiny\( \ad[Y]\)}};

  \draw[red, dashed, thin, ->] (1,1) -- (0.1,1.90);
  \draw[red, dashed, thin, ->] (1,1) -- (1.9,1.90);

  \draw[blue, dashed, thin, ->] (1,1) -- (0.1,0.10);
  \draw[blue, dashed, thin, ->] (1,1) -- (1.9,0.10);

\draw[] (1.7,0.5) node[align=left,   right,blue,thin] {{\tiny\(\ad[X_2]\)}};
\draw[] (0.3,0.5) node[align=left,   left,blue,thin] {{\tiny\(\ad[Y_2]\)}};

\draw[] (1.7,1.5) node[align=left,   right,red,thin] {{\tiny\(\ad[X_1]\)}};
\draw[] (0.3,1.5) node[align=left,   left,red,thin] {{\tiny\(\ad[Y_1]\)}};

\filldraw (2,2) circle (2pt) node[align=left,   above] {\(X_1\)};
\filldraw (1,2) circle (2pt) node[align=left,   above] {\(H_1\)};
\filldraw (0,2) circle (2pt) node[align=left,   above] {\(Y_1\)};

\filldraw (2,0) circle (2pt) node[align=left,   below] {\(X_2\)};
\filldraw (1,0) circle (2pt) node[align=left,   below] {\(H_2\)};
\filldraw (0,0) circle (2pt) node[align=left,   below] {\(Y_2\)};

\filldraw[gray] (1,1) circle (2pt) node[align=left,gray,   above] {\(H\)};

\end{tikzpicture}
\end{center}
The quadratic form for $\mathfrak{so}(2,2)$ corresponds to (twice) the determinant form used in Proposition \ref{KAKrank2} with the associated matrix $\left[\begin{smallmatrix}0&0&0&1\\0&0&-1&0\\0&-1&0&0\\1&0&0&0\end{smallmatrix}\right]$.  We identify the two $\mathfrak{sl}_2$ in $\mathfrak{so}(2,2)\cong\lsl\times\lsl$ as
$$\mathfrak{sl}_2\otimes\on{Id}=\on{span}\left\{H_1=\left[\begin{smallmatrix}1&0&0&0\\0&-1&0&0\\0&0&1&0\\0&0&0&-1\end{smallmatrix}\right],\  X_1=\left[\begin{smallmatrix}0&1&0&0\\0&0&0&0\\0&0&0&1\\0&0&0&0\end{smallmatrix}\right],\  Y_1=\left[\begin{smallmatrix}0&0&0&0\\1&0&0&0\\0&0&0&0\\0&0&1&0\end{smallmatrix}\right]\right\}$$
and
$$\on{Id}\otimes\mathfrak{sl}_2=\on{span}\left\{H_2=\left[\begin{smallmatrix}1&0&0&0\\0&1&0&0\\0&0&-1&0\\0&0&0&-1\end{smallmatrix}\right],\  X_2=\left[\begin{smallmatrix}0&0&1&0\\0&0&0&1\\0&0&0&0\\0&0&0&0\end{smallmatrix}\right],\  Y_2=\left[\begin{smallmatrix}0&0&0&0\\0&0&0&0\\1&0&0&0\\0&1&0&0\end{smallmatrix}\right]\right\}.
$$
We make the choice of the principal $\mathfrak{sl}_2$ by setting
$$ H=\left[\begin{smallmatrix}2&0&0&0\\0&0&0&0\\0&0&0&0\\0&0&0&-2\end{smallmatrix}\right],\  X=\left[\begin{smallmatrix}0&1&1&0\\0&0&0&1\\0&0&0&1\\0&0&0&0\end{smallmatrix}\right],\  Y=\left[\begin{smallmatrix}0&0&0&0\\1&0&0&0\\1&0&0&0\\0&1&1&0\end{smallmatrix}\right] .$$

\begin{lemma}
\label{weightsofgoh}
With respect to the principal $\lsl$   in each of the above Lie algebras, we have that
$$
\mathfrak{so}(2,1)=V^{(2)},\  \mathfrak{so}_\eta(3,1)=V^{(2)}\oplus V^{(2)},\ 
 \mathfrak{so}(2,2)=V^{(2)}\oplus V^{(2)}.
$$
In each case the first $V^{(2)}$ corresponds to the principal $\lsl$ itself. 
 For $\mathfrak{so}_\eta(3,1)$,  the other is 
the vector space spanned by $\{X_2,M,Y_2\}$. 
 For  $\mathfrak{so}(2,2)$ we may choose the second $V^{(2)}$ to be
one of the direct factors of $\mathfrak{so}(2,2)\cong\lsl\times\lsl$.
\end{lemma}

\subsection{Invariant complement} 
We choose the quadratic form in the ambient space defining 
$\on{SO}(d)$ such that the original group forms the stabilizer group of $e_d$,
i.e.\ we add $\lambda u^2$ to the above quadratic forms for some $\lambda\in\bZ_p^\times$. 
The complement of $\mathfrak{so}(4)$ in $\mathfrak{so}(5)$ can then 
be given by the subspace of matrices of the form
$$\left[\begin{smallmatrix}0&0&0&0&a\\0&0&0&0&b\\0&0&0&0&c\\0&0&0&0&d\\**&*&*&*&0\end{smallmatrix}\right],$$
where the last row is uniquely determined by the vector $(a,b,c,d)$ (and the underlying quadratic form). 
The same holds similarly for $\mathfrak{so}(3)$ within $\mathfrak{so}(4)$. 
In particular, the adjoint action by $\mathfrak{so}(d-1)$ 
on the complement within $\mathfrak{so}(d)$ corresponds to the 
standard representation. We denote the complement of the 
model Lie algebras  in $\mathfrak{so}(d)$ by $\gos_1$ and in $\mathfrak{sl}_{d-1}$ by $\gos_2$. 
The calculations that follow will also define $\gos_2$ concretely
and will give the weight classification of $\gos_i$ for $i=1,2$.

\begin{lemma}
\label{weightsofgos4}
For $d=4$ we have $\gos_1=V^{(2)}$ and $\gos_2=V^{(4)}$. 
\end{lemma}

\smallskip
\paragraph{\textbf{Complement $\gos_1<\mathfrak{so}(4)$}}
We introduce the shorthand
$$(a,b,c)=\left[\begin{smallmatrix}0&0&0&a\\0&0&0&b\\0&0&0&c\\
\lambda c&-\lambda b&\lambda a&0\end{smallmatrix}\right]\in\mathfrak{so}(4),$$
then it is easy to verify that
\begin{center}
\begin{tabular}{ll}
$\ad[X](a,b,c)=(b,c,0),$	&	$\ad[X]^2(a,b,c)=(c,0,0),$   \\
$\ad[Y](a,b,c)=(0,2a,2b),$	&	$\ad[Y]^2(a,b,c)=(0,0,4a).$    \\
\end{tabular}
\end{center}

\medskip\paragraph{\textbf{Complement $\gos_2<\mathfrak{sl}_3$}}
If $z=\left[\begin{smallmatrix}0&0&1\\0&0&0\\0&0&0\end{smallmatrix}\right]\in \mathfrak{sl}_3$ then
\begin{align*}
	 \ad[Y]z&=2\left[\begin{smallmatrix}0&-1&0\\0&0&1\\0&0&0\end{smallmatrix}\right],\   \ad[Y]^2z=4\left[\begin{smallmatrix}1&0&0\\0&-2&0\\0&0&1\end{smallmatrix}\right],\\ \ad[Y]^3z&=24\left[\begin{smallmatrix}0&0&0\\1&0&0\\0&-1&0\end{smallmatrix}\right],\  \ad[Y]^4z=96\left[\begin{smallmatrix}0&0&0\\0&0&0\\1&0&0\end{smallmatrix}\right].
	\end{align*}
Hence these vectors span a $5$-dimensional subspace $\gos_2$. It is easy to check that $z$ is an eigenmatrix 
for the adjoint action of the diagonal element $H\in \lsl$ and in the kernel of
$\ad[X]$, which implies by the structure of finite dimensional representations
of $\lsl$ that the other vectors are also eigenmatrices for $\ad[H]$
and that $\gos_2\cong V^{(4)}$ is indeed 
an invariant complement to the principal $\lsl$.

\begin{lemma}
\label{weightsofgos5}
For $d=5$ we have $\gos_1=V^{(0)}\oplus V^{(2)}$ and $\gos_2\cong V^{(0)}\oplus V^{(2)}\oplus V^{(4)}$
for the adjoint representation of the principal $\lsl$. 
Moreover, $\gos_i$ is irreducible under $\mathfrak{so}_\eta(3,1)$ respectively $\mathfrak{so}(2,2)$
for $i=1,2$.
\end{lemma}

\smallskip\paragraph{\textbf{Complement $\gos_1<\mathfrak{so}(5)$ - Quasi split}}
We again introduce the shorthand
$$(a,b,c,d)=\left[\begin{smallmatrix}0&0&0&0&a\\0&0&0&0&b\\0&0&0&0&c\\0&0&0&0&d\\
\lambda d&-\lambda b&-\lambda c/\eta&\lambda a&0\end{smallmatrix}\right]\in\mathfrak{so}(5),$$
and calculate
\begin{center}
\begin{tabular}{ll}
$\ad[X](a,b,c,d)=(b,d,0,0),$	&	$\ad[X]^2(a,b,c,d)=(d,0,0,0),$   \\
$\ad[Y](a,b,c,d)=(0,2a,0,2b),$	&	$\ad[Y]^2(a,b,c,d)=(0,0,0,4a).$    \\
\end{tabular}
\end{center}
In particular, $\lsl$ acts trivially on the subspace $\{(0,0,c,0)\}\cong V^{(0)}$
and we have $\{(a,b,0,d)\}\cong V^{(2)}$. 
Furthermore,
\begin{center}
\begin{tabular}{ll}
$\ad[X_2](a,b,c,d)=(c,0,\eta d,0),$	&	$\ad[X_2]^2(a,b,c,d)=(\eta d,0,0,0),$   \\
$\ad[Y_2](a,b,c,d)=(0,0,-\eta a,-2 c),$	&	$\ad[Y_2]^2(a,b,c,d)=(0,0,0,4 \eta a),$  
\end{tabular}
\end{center}
which shows invariance of $\gos_1=\{(a,b,c,d)\}$ under $\mathfrak{so}_\eta(3,1)$. 
In particular these relations also show that the two non-isomorphic but irreducible subrepresentations 
of $\gos_1$ with respect to $\lsl$ are not invariant under $\mathfrak{so}_\eta(3,1)$
and so $\gos_1$ must be irreducible with respect to $\mathfrak{so}_\eta(3,1)$. 
Finally, we note that
\begin{equation}\label{gettingminusX}
 [(1,0,0,0),(0,1,0,0)]=-\lambda X,	
\end{equation}
which we will use to generate~$\gog_1$ starting with~$\gos_1$.

\medskip\paragraph{\textbf{Complement $\gos_2<\mathfrak{sl}_4$ - Quasi split}}
The following diagram depicts the representation $V^{(4)}$, $V^{(2)}$ and $V^{(0)}$ and their relations. Note that it includes a $V^{(4)}$-representation of the second $\lsl$-triple from Section \ref{irreduciblesection}.
\begin{center}
\begin{tikzpicture}

  \draw[red, dashed, thin, ->] (4,2) -- (0.1,2);
  \draw[] (1.5,2) node[above,red] {{\tiny\( \ad[Y]\)}};
  \draw[red, dashed, thin, ->] (3,1) -- (1.1,1);
  \draw[] (1.5,1) node[above,red] {{\tiny\( \ad[Y]\)}};

  \draw[green, dashed, <->] (1,1.9) -- (1,1.1) ;
  \draw[green, dashed, ->] (2,1.9) -- (2,1.1) ;
  \draw[green, dashed, <->] (3,1.9) -- (3,1.1) ;
  \draw[green] (2,1.5) node[right] {{\tiny\( \ad[M]\)}};

  \draw[blue, dashed, thin, ->] (4,2) -- (3.1,1.1);
  \draw[blue, dashed, thin, ->] (1,1) -- (0.1,1.90);
  \draw[blue] (0.4,1.5) node[below] {{\tiny\( \ad[Y_2]\)}};
  \draw[blue] (3.7,1.5) node[below] {{\tiny\( \ad[Y_2]\)}};

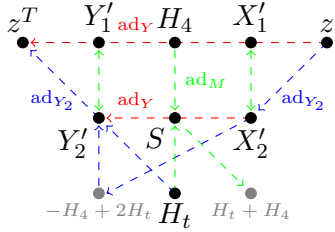
  \draw[blue, dashed, thin, ->] (3,1) -- (1.1,0.0);
  \draw[blue, dashed, thin, ->] (1,0.1) -- (1.0,0.9);

  \draw[blue, dashed, thin, ->] (2,0) -- (1.1,0.9);

  \draw[green, dashed, thin, ->] (2,0) -- (2,0.9);

  \draw[green, dashed, thin, ->] (2,1) -- (2.9,0.1);

\filldraw (0,2) circle (2pt) node[align=left,   above] {\(z^T\)};
\filldraw (1,2) circle (2pt) node[align=left,   above] {\(Y_1'\)};
\filldraw (2,2) circle (2pt) node[align=left,   above] {\(H_4\)};
\filldraw (3,2) circle (2pt) node[align=left,   above] {\(X_1'\)};
\filldraw (4,2) circle (2pt) node[align=left,   above] {\(z\)};

\filldraw (1,1) circle (2pt) node[align=left,   below left] {\(Y_2'\)};
\filldraw (2,1) circle (2pt) node[align=left,   below left] {\(S\)};
\filldraw (3,1) circle (2pt) node[align=left,   below] {\(X_2'\)};

\filldraw (2,0) circle (2pt) node[align=left,   below] {\(H_t\)};
\filldraw[gray] (1,0) circle (2pt) node[align=left,   below] {{\tiny\(-H_4+2H_t\)}};
\filldraw[gray] (3,0) circle (2pt) node[align=left,   below] {{\tiny\(H_t+H_4\)}};

\end{tikzpicture}
\end{center}

Since $\mathfrak{sl}_4$ consists of traceless matrices, we see
that  a linear complement to $\mathfrak{so}_\eta(3,1)$ can be obtained by taking the linear hull of
\begin{align*}
z&=\left[\begin{smallmatrix}0&0&0&1\\0&0&0&0\\0&0&0&0\\0&0&0&0\end{smallmatrix}\right],\ 
 {X}_1'=\left[\begin{smallmatrix}0&1&0&0\\0&0&0&-1\\0&0&0&0\\0&0&0&0\end{smallmatrix}\right], \  
 {H}_4=\left[\begin{smallmatrix}-1&0&0&0\\0&2&0&0\\0&0&0&0\\0&0&0&-1\end{smallmatrix}\right],\  
 {Y}_1'=\left[\begin{smallmatrix}0&0&0&0\\1&0&0&0\\0&0&0&0\\0&-1&0&0\end{smallmatrix}\right],\\
{X}_2'&=\left[\begin{smallmatrix}0&0&1 &0\\0&0&0&0\\0&0&0&-\eta\\0&0&0&0\end{smallmatrix}\right],\ 
S=\left[\begin{smallmatrix}0&0&0&0\\0&0&1&0\\0&\eta&0&0\\0&0&0&0\end{smallmatrix}\right],\ 
{Y}_2'=\left[\begin{smallmatrix}0&0&0&0\\0&0&0&0\\\eta &0&0&0\\0&0&-1 &0\end{smallmatrix}\right],\ 
{H}_{\text{t}}=\left[\begin{smallmatrix}1&0&0&0\\0&1&0&0\\0&0&-3&0\\0&0&0&1\end{smallmatrix}\right],
\end{align*}
and the vector $z^T$. 
As before, $z$ is an eigenvector for $\ad[H]$ and satisfies $\ad[X]z=0$, i.e.
it is a heighest weight vector. In fact $\ad[H]z=4z$ and $z$ together with
$$ 
\ad[{Y}]z=-2{X}_1',\  \ad[{Y}]X_1'=2{H}_4,\  \ad[{Y}] H_4=-6{Y}_1',\   \ad[{Y}]Y_1'=4z^T
$$
span an irreducible subrepresentation for $\lsl$ isomorphic to $V^{(4)}$.
Moreover, from $\ad[H]{X}_2'=2{X}_2'$ and $\ad[X](X_2')=0$ we see that 
$X_2'$, $\ad[Y]X_2'=2S$, and $\ad[Y]S=-2Y_2'$ together 
span an irreducible $\lsl$-subrepresentation isomorphic to $V^{(2)}$. 
Finally
we also find a trivial representation isomorphic to $V^{(0)}$ since 
$\ad[X]({H}_{\text{t}})=\ad[Y]({H}_{\text{t}})=0$. 

For the action of the additional vectors $X_2$, $Y_2$ and $M$ we first note
that
\begin{align*}
	&\ad[M]z=\ad[M]z^T=0,\  \ad[M]X_1'=2X_2',\ \ad[M]X_2'=-2\eta X_1'\\
	&\ad[M]Y_1'=2Y_2',\ \ad[M]Y_2'=-2nY_1', \ \ad[M]H_4=4S\\
	&\ad[M]S=-\frac{4\eta}3(H_4+H_{\text{t}}), \ \ad[M]H_{\text{t}}=8S.
\end{align*} 
Since $\mathfrak{so}_\eta(3,1)$ is generated by $\lsl$ and $M$ these
show that the span of the above vectors gives an invariant complement $\gos_2<\mathfrak{sl}_4$.
Moreover, we again see (e.g. from the last three equations) 
that the three nonisomorphic subrepresentations for
$\lsl$ are not invariant under $\mathfrak{so}_\eta(3,1)$ and neither is the sum
of any two of the three subrepresentations. Hence it follows that the invariant complement
$\gos_2$ is irreducible under $\mathfrak{so}_\eta(3,1)$.
Let us also note the identities
\begin{equation}\label{Y2stattM}
 \begin{gathered}
	\ad[{Y_2}]z=2{X}_2', \ 
	 \ad[{Y_2}]X_2'=\frac{2\eta}3(-{H}_4+2 H_\text{t}),\  \ad[{Y_2}](-H_4+2H_\text{t})=-18 {Y}_2',\\
	 \ad[{Y_2}]Y_2'=-4\eta z^T,\ \ad[{Y_2}]{H}_{\text{t}}=-8{Y}_2',\  \ad[{X_2}](H_{\text{t}})=-4{X}_2'.
 \end{gathered}
\end{equation}

\medskip
\paragraph{\textbf{Complement $\gos_1<\mathfrak{so}(5)$ - Split}}
We introduce the shorthand
$$(a,b,c,d)=\left[\begin{smallmatrix}0&0&0&0&a\\0&0&0&0&b\\0&0&0&0&c\\0&0&0&0&d\\
-\lambda d&\lambda c&\lambda b&-a\lambda &0\end{smallmatrix}\right]
\in\mathfrak{so}(5).$$
With this notation we have
\begin{equation*}
\begin{aligned}
&\ad[X](a,b,c,d)=(b+c,d,d,0),	&	\ad[X]^2(a,b,c,d)=(2d,0,0,0),   \\
&\ad[Y](a,b,c,d)=(0,a,a,b+c),	&	\ad[Y]^2(a,b,c,d)=(0,0,0,2a). 
\end{aligned}
\end{equation*}
In particular, $\lsl$ acts trivially on $\{(0,b,-b,0)\}\cong V^{(0)}$,
and the invariant complement to this subspace is isomorphic to $V^{(2)}$. 
Moreover,
\begin{equation}\label{gettingaround}
\begin{aligned}
&\ad[Y_1](a,b,c,d)=(0,a,0,c),	&	\ad[X_1](a,b,c,d)=(b,0,d,0),   \\
&\ad[Y_2](a,b,c,d)=(0,0,a,b),	&	\ad[X_2](a,b,c,d)=(c,d,0,0), 
\end{aligned}
\end{equation}
together show that $\gos_1$ is invariant under $\mathfrak{so}(2,2)$
and irreducible. Finally we note that
\begin{equation}\label{gettingX}
\begin{aligned}
 &[(1,0,0,0),(0,0,1,0)]=\lambda X_1,\\
 &[(1,0,0,0],(0,1,0,0)]=\lambda X_2,
\end{aligned}
\end{equation}
which we will use to generate~$\mathfrak{g}_1$ out of~$\gos_1$.

\medskip
\paragraph{\textbf{Complement $\gos_2<\mathfrak{sl}_4$ - Split}}
We define $$z=\left[
\begin{smallmatrix}0&0&0&1\\0&0&0&0\\0&0&0&0\\0&0&0&0
	\end{smallmatrix}\right]\in \mathfrak{sl}_4\mbox{ and } 
\gos_2=\left\{\left[
\begin{smallmatrix}w&a&b&x\\e&-w&c&-b\\f&d&-w&-a\\y&-f&-e&w
	\end{smallmatrix}\right]\right\}.$$ 

The following diagram shows the action on $z$ (and $\ad[Y_2]z$, $\ad[Y_2]^2z$) by $Y_1$ moving horizontally
and the corresponding vertical action of $Y_2$.
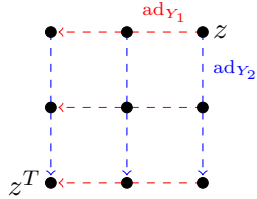
\begin{center}
\begin{tikzpicture}

  \draw[red, dashed, ->] (2,0) -- (0.1,0);
  \draw[red, dashed, ->] (2,1) -- (0.1,1);
  \draw[red, dashed, ->] (2,2) -- (0.1,2);

  \draw[blue, dashed, ->] (0,2) -- (0,0.1);
  \draw[blue, dashed, ->] (1,2) -- (1,0.1);
  \draw[blue, dashed, ->] (2,2) -- (2,0.1);

\filldraw (2,2) circle (2pt) node[align=left,   right] {\(z\)};
\filldraw (2,1) circle (2pt) node[align=left,   right] {};
\filldraw (2,0) circle (2pt) node[align=left,   right] {};

\draw[] (2,1.5) node[align=left,   right,blue,thin] {{\tiny\(\ad[Y_2]\)}};
\draw[] (1.5,2) node[align=left,   above,red,thin] {{\tiny\(\ad[Y_1]\)}};

\filldraw (1,2) circle (2pt) node[align=left,   above] {};
\filldraw (1,1) circle (2pt) node[align=left,   right] {};
\filldraw (1,0) circle (2pt) node[align=left,   right] {};

\filldraw (0,2) circle (2pt) node[align=left,   right] {};
\filldraw (0,1) circle (2pt) node[align=left,   right] {};
\filldraw (0,0) circle (2pt) node[align=left,   left] {\(z^T\)};

\end{tikzpicture}
\end{center}

Similarly to the previous case $z$ generates an irreducible representation for the principal $\lsl$. 
It is easy to check that $\gos_2=\{v\in\mathfrak{sl}_4: v=Jv^TJ^{-1}\}$
is invariant under $\mathfrak{so}(2,2)$ (which is defined by the quadratic form
associated to the symmetric matrix $J$). 
Moreover, we could also repeat the detailed analysis to show that 
there are two more irreducible subrepresentations contained in $\gos_2$ that are isomorphic to $V^{(2)}$
and $V^{(0)}$ respectively. However, this follows in fact by noticing that the case at hand
and the previous case become isomorphic over the algebraic closure of $\bQ_p$
(and since the principal $\lsl$ correspond
to each other as subgroups of the isomorphic subgroups
$\mathfrak{so}_\eta(3,1)$ resp. $\mathfrak{so}(2,2)$ over the algebraic closure). 
As we have seen above the three irreducible subrepresentations resp. also sums of two out of the three
subrepresentations are not invariant under $\mathfrak{so}_\eta(3,1)$, which implies also that $\gos_2$ is 
irreducible for the adjoint representation restricted to $\mathfrak{so}(2,2)$. 

Irreducibility now implies that starting with $z$ one can obtain a basis of $\gos_2$ by taking finitely many commutators
with elements of $\mathfrak{so}(2,2)(\bZ)$. Even without doing this concretely\footnote{We did this concretely
for $\mathfrak{so}(2,1)$ because it was easy and for $\mathfrak{so}_\eta(3,1)$ because
it does not seem completely impossible that the parameter $\eta$ (that depends on $p$) 
might affect the outcome.} we find finitely many integer
matrices (independent of $p$) that give also a $\bZ_p$-basis of $\gos_2[0]$ for all but possibly a finite list
of primes. By Proposition \ref{prime_existence} we may assume that this holds for our chosen prime $p$.

\subsection{Undistorted complement of $\goh_1$ and $\goh_2$}

The study of undistorted complements of the model groups also gives complements for our acting groups.

In fact, the acting group for the first factor $\bG_1$ is precisely the stabilizer subgroup $\bH_v$. 
By our choice of $p$ and Corollary \ref{classificationofH}
we know that it is $\bZ_p$-conjugated to one of our model groups.  Now set $\lambda=D$ and extend
the quadratic form $Q_0$ (in the variables $x,y,z$ and possibly $w$) that defines the model group (and is conjugated
to the sum of $d$ squares restricted to $v^\perp$) by the variable $u$  to obtain $Q_0+D u^2$. It follows that 
the group $\SO[d]$ with the subgroup $\bH_v$ is $\bZ_p$-conjugated to another orthogonal group in $d$ variables
with the model group as its subgroup acting on the first $d-1$ variables. The latter case we studied above
and hence we may conjugate the above invariant complement $\gos_1$ to obtain an invariant complement $\gor_1$
within the standard $\mathfrak{so}(d)$. 

Similarly, the acting group $\bH_{\Lambda_v}$ for the second factor $\bG_2$ is 
the group $\bH_v^{g_v^{-1}}$ projected to $\SL_{d-1}$ (with $\bH_v$ and $g_v$
as defined in Section \ref{ReformulationTowardsHomogeneousDynamics}). 
Alternatively we may also conjugate $\bH_v$ by the matrix $g_v'$ consisting of a $\bZ$-basis of $\Lambda_v$
and the vector $v$ (and determinant $D\in\bZ_p^\times$) to obtain $\bH_{\Lambda_v}$ (without the need of the projection operation). 
By Corollary \ref{classificationofH} 
this group is now $\bZ_p$-conjugated (via some $h_v\in\SL_{d-1}(\bZ_p)$) to one of our model groups. 
Hence we only have to apply the inverse conjugation to the invariant complement $\gos_2$
to obtain the invariant complement $\gor_2$ of $\goh_2$ within $\mathfrak{sl}_{d-1}$.

It is easy to see from the concrete description of $\gos_i$ for $i=1,2$ (resp. in the split rank two case from our choice of $p$) that these complements are undistorted complements
of the model Lie algebras. As the above conjugation is taking place over $\bZ_p$ the same follows for $\gor_i$
and $i=1,2$.

\subsection{Definition of the set $\mathcal{F}$}
For $\gor\in\{\gor_1,\gor_2\}$ let $\gor^{\on{lw}}$ denote the  lowest weight space and let $\gor^\text{hw}$ be the highest weight space where weights are defined using the diagonal subgroup of the principal $\lsl$. 
We also write $(\cdot)^{\text{hw}}:\gor\rightarrow\gor^{\on{hw}}$ and
 $(\cdot)^{\text{lw}}:\gor\rightarrow\gor^{\on{lw}}$ for the projection maps whose kernels consist
of the other weight spaces. 
For any unipotent flow $n_t=\exp{(Zt)}$ of the acting group $H_p$ with $Z\in\goh[0]$  and any $r\in\gor[0]$ we have that
\begin{equation}
\label{eq:expofad}
\Ad[n_t]r=r+t\ad[Z]r+\tfrac{t^2}{2}\ad[Z]^2r+\tfrac{t^3}{3!}\ad[Z]^3r+\tfrac{t^4}{4!}\ad[Z]^4r
\end{equation}
is a polynomial in $t$ with $\bZ_p$ coefficients (because $p\geq5$) of degree at most $d_\gor$ ($d_\gor=2$ if $\gor=\gor_1$ and $d_\gor=4$ if $\gor=\gor_2$).

Given a vector $r$ in the complement, we can apply $\bZ_p$-unipotents to 
maximize the norm of its lowest weight  component corresponding to $(0,\dots,0,1)\in\gos_1$ and $z^T\in\gos_2$ in the above notation.
\begin{lemma} 
\label{movingalongthecomplement}
There exists a finite set $\mathcal{F}(v)$ respectively $\mathcal{F}(\Lambda_v)$ of
uniformly bounded cardinality consisting of unipotent elements in $\HvS[p]\cap K[0]$ respectively $\HVS[p]\cap K[0]$ such that for any $r_1\in\gor_1$ and $r_2\in\gor_2$ there exists $h_1\in \mathcal{F}(v)$ and $h_2\in \mathcal{F}(\Lambda_v)$ such that $\|\Ad[h_1]r_1\|_p=\|(\Ad[h_1]r_1)^{\on{lw}}\|_p$ and $\|\Ad[h_2]r_2\|_p=\|(\Ad[h_2]r_2)^{\on{lw}}\|_p$
\end{lemma}

\begin{proof}
It suffices to work with the invariant complements $\gos_1$ respectively $\gos_2$
for our model groups. In the case of $s\in \gos_1$ (and so either $s=(a,b,c)$ or $s=(a,b,c,d)$), 
assume first that its maximal coefficient does not lie in $V^{(0)}$ which always holds for $d=4$. 
For $d=5$ this assumption says $\max{(|a|_p,|b|_p,|d|_p)}=\|s\|_p$ if $\HvS[p]$ is quasi-split 
respectively $\max{(|a|_p,|b+c|_p,|d|_p)}=\|s\|_p$ if $\HvS[p]$ is split. Then $\Ad[\exp(tY)]=(*,\dots,P(t))$ where $P(t)=c+2bt+2at^2$, $d+2bt+2at^2$ and $d+(b+c)t+at^2$ in the three cases. It follows that there exists $t\in\{0,1,2\}$ such that $\Vert P(t)\Vert_p$ is as large as $\|s\|_p$ 
(so that there is no accidental cancellation in the lowest weight term). 
If on the other hand $d=5$ and the projection of $s$ to $V^{(0)}$ dominates, then we can use $\Ad[\exp{(tY_2)}]$ whose adjoint action on $s$ has lowest weight $d-2 ct+2\eta at^2$ (in the quasi-split case) respectively $d+bt$ (split) and evaluation at $t=1$ suffices.

For $s\in\gos_2$, we may use again $\exp{(tY)}$ if the largest coefficient of $s$ belongs
to a basis vector in $V^{(4)}$. Since $\Ad[\exp{(tY)}]s$ is now of degree four, we can ensure that for  $p\geq16$ 
there exists some $t\in\{0,1,2,3,4\}$ with  $\Vert P(t)\Vert_p=\Vert s\Vert$.
For $d=4$ this finishes the argument.

 If on the other hand $d=5$ and the maximal coefficient is in either $V^{(0)}$ or $V^{(2)}$ and $\HVS[p]$ is quasi-split 
we use the elements $Y$ and $Y_2$ together: If the largest coefficient of $s$ corresponds to $X_2'$ or $Y_2'$ (belonging to 
the representation isomorphic to $V^{(2)}$) or $H_\text{t}$ (spanning $V^{(0)}$)
we use the same argument for the irreducible representation generated by $z$ with respect to the $\lsl$-triple 
$H, X_2, Y_2$, see \eqref{Y2stattM}. If the largest coefficient appears in $S$ (belonging to 
the representation isomorphic to $V^{(2)}$), we first apply $\Ad[\exp(Y)]$
to achieve that afterwards the norm of the coefficient of $Y_2'$ is just as big and 
afterwards apply $\Ad[\exp(tY_2)]$ as before.
 Consequentially, a finite set of products of $\exp{(tY)}$ and $\exp{(tY_2)}$ 
for $t\in\{0,1,2,3,4\}$ might be used to define $\mathcal{F}(\Lambda_v)$.

If $\HVS[p]$ is split then we may apply first $\Ad[\exp(tY_1)]$ for $t\in\{0,1,2\}$
and afterwards $\Ad[\exp(t Y_2)]$ for $t\in\{0,1,2\}$ which leads to the conclusion (e.g.
by studying the weight diagram of $\gos_2$).  
\end{proof}

\subsection{Effective Generation}
\label{boundedgeneration}
Starting with a vector in $\gor_i^{\on{ht}}$ one can obtain
in a sense all other vectors using the lower unipotents of $\goh_i$. Since the Lie algebra $\goh_2$ is maximal in $\gog_2$, we can upgrade this to generate $\gog_2$ from $\gor_2$.
\begin{lemma}
\label{generation}
There exist a finite sets $\mathcal{M}$ and $\mathcal{N}$ of uniformly bounded cardinality satisfying\footnote{We will need the second part only for the proof of the joint equidistribution in Section \ref{proofofequidistribution}.}
\begin{enumerate}
\item \label{generationgor}
 $\mathcal{M}$ is a subset of $\HvS[p]^+\cap K[0]$ respectively $\HVS[p]^+\cap K[0]$
such that for $r\in\gor_i^{\on{ht}}[0]$ with $\|r\|_p=1$ the set 
$\{\Ad[m]r\}_{m\in \mathcal{M}}$ forms a $\bZ_p$-basis of $\gor_i[0]$.
\item \label{generationgog} $\mathcal{N}$ is a subset of $\HVS[p]^+\cap K[0]$ such that 
for $r\in\gor_2^{\on{ht}}[0]$ with $\|r\|_p=1$ and 
	$s\in\gor_2^{\on{lw}}[0]$ with $\|s\|_p=1$
the set 
$\{\Ad[n]\Ad[\exp{s}]r\}_{n\in \mathcal{N}}\cup \{\Ad[m]r\}_{m\in \mathcal{M}}$ forms a $\bZ_p$-basis of $\gog_2[0]$.
\end{enumerate}

In particular, in both bases we only need to take the exponential of nilpotent elements, where the exponential is simply a polynomial with coefficients in $\bZ_p$ (since $p\geq 10$).
\end{lemma}
\begin{proof}
For (1) let $n_t=\exp{(tY)}$ be the unipotent defined by the element $Y$ of the principal $\lsl$ and $t\in\bQ_p$. Then we deduce immediately from equation (\ref{eq:expofad}) and
the Vandermonde determinant that $\{\Ad[n_t]r:t\in\{0,\dots,d_\gor\}\}$ are linearly independent
and therefore span $V^{(2)}$ for $i=1$ respectively $V^{(4)}$ for $i=2$.
In fact, the Vandermonde determinant is independent of $p$ so that the linear independence
also holds for the vectors modulo $p$ over $\bF_p$ (as we may choose $p$ large enough). 
This proves the claim for $d=4$. 

For $d=5$ and $i=1$ we apply the above argument and see that by adding $\Ad[\exp{Y_2}]r$ to the list we can span all of $\gor_1[0]$. 

Assume now $d=5$ and $i=2$ so that $r$ is $\bZ_p^\times$-multiple of $z$. 
If $\HVS[p]$ is quasi-split we apply the above argument and add $\Ad[\exp{Y_2t}]r$ for $t=1,2,3$ and $(\Ad[\exp{Y}]\circ\Ad[\exp{Y_2}])r$. Given our 
concrete formulas for the image of $z$ under $\ad[Y]$ resp. $\ad[Y_2]$ this proves
the lemma in this case.
If $\HVS[p]$ is split, we use $\Ad[\exp{Y_1t_1}\exp{Y_2t_2}]r$ for $t_1,t_2=0,1,2$
together with the same Vandermonde argument.

For the proof of (2) we begin with $d=4$. From (1), we can produce a $\bZ_p$-basis of $\gor_2[0]$. On the other hand, if we take $z^T\in\gor_2^{\on{lw}}[0]$ (or any other $\bZ_p^{\times}$-multiple of $z^T$) then we have $\ad[z^T]z=-\tfrac12H$ which implies
\[
\Ad[\exp{z^T}]z=z-\tfrac12H-z^T.
\]
Acting by an element $h\in \HvS[p]$, we have by invariance that 
\[
\pi_{\goh_2}\Ad[h]\Ad[\exp{z^T}]z=-\tfrac12\Ad[h]H.
\]
Now act by the principle unipotents $\exp{X}$ and $\exp{Y}$, to produce three elements 
whose projection to $\goh_2$ are (up to scalar multiple) $H$, $\Ad[\exp{X}]H=e+H-X$ and 
$\Ad[\exp{Y}]H=e+H+Y$. In particular, complementing these three elements with the $\bZ_p$-basis of 
$\gor_2[0]$ we conclude by setting $\mathcal{N}=\{\exp{X}, \exp{Y}\}$.
If $d=5$ and $\HvS[p]$ is quasi-split then we may apply $\Ad[\exp X_2]$ to $\Ad[\exp{z^T}]z$ to get an element whose projection is a sum of a~$\bZ_p$-multiple of~$H$ and 
a $\bZ_p^\times$-multiple of $X_2$ (see the diagram, going in the reverse direction of $Y_2$). Then applying $\Ad[\exp{Y}]$ again we get the second $V^{(2)}$ in $\goh_2$ modulo $\gor_2[0]$. Thus we set
$\mathcal{N}=\{\exp{X}, \exp{Y}, \exp{Y}\exp{X_2},\exp{Y}\exp{Y}\exp{X_2}\}$.

The analogous study of the diagram for the split case shows that we may take
$\mathcal{N}=\{n^{j}u_i: i=1,2\; j=0,1,2\}$ where $n=\exp{Y}$ and $u_i=\exp{X_i}$.
\end{proof}

We end with the following implicit function theorem.

\begin{lemma}[Implicit function theorem]
\label{hensel}
Assume that $\gog$ has a $\bZ_p$-basis $\{v_i\}_{i\leq k}\subset\gog[0]$ consisting of nilpotent elements. Define $u_i(t)=\exp{tv_i}$ for all $i\leq k$, $t\in\bQ_p$, and define $\underline{u}(t)=u_1(t_1)\dots u_k(t_k)$ for all $\underline{t}\in p\bZ_p^k$. Then $\exp{\gog[m]}=\underline{u}(p^m\bZ_p^k)$ for all $m\geq 1$.
\end{lemma}
This is of course well known, but for the convenience of the reader we outline the proof.

\begin{proof}
In the following we let $\underline{a}\in p\bZ_p^k$ 
and  $\underline{t}\in p^n\bZ_p^k$ for some $n\geq 1$. 
We define $\underline{u}(\underline{a})=u_1(a_1)\cdots u_k(a_k)$ and notice that 
 $u_j(a_j+t_j)\equiv u_j(a_j)+t_jv_j\pmod{p^{n+1}}$
 and $u_j(a_j+t_j)\equiv e\pmod p$
for $j=1,\ldots,k$. Taking the product we obtain from this  
 \[
  \underline{u}(\underline{a}+\underline{t})\equiv \underline{u}(\underline{a})+\sum t_jv_j\pmod{p^{n+1}}.
\]
Also fix some arbitrary $g\in K[m]$ with $m\geq 1$.
Using \(g^{-1}\equiv e \pmod p\) we obtain
$$
 g^{-1}\underline{u}(\underline{a}+\underline{t})\equiv 
 g^{-1}\underline{u}(\underline{a})+ \sum t_jv_j \pmod {p^{n+1}}.
$$

The lemma now follows inductively using Hensel's lifting procedure. 
Recall that $\exp\mathfrak{g}[n]=K[n]$ for all $n\geq 1$ (since $p\geq 2$) and that $\exp(\cdot)$ has the inverse $\log(\cdot)$ which is defined on $K[1]$ and takes values in $\mathfrak{g}[1]$.
For $n=m$ we can solve $g^{-1}\underline{u}(\underline{a})\in K[n]$ by putting $a=0$. Assume therefore that $n\geq m$ and we have already found some $\underline{a}\in\mathfrak{g}[m]$ that solves $g^{-1}\underline{u}(\underline{a})\in K[n]$, or equivalently $\log{g^{-1}\underline{u}(\underline{a})}=w\in\gog[n]$. Using the power series of $\log(\cdot)$ we obtain now for any $\underline{t}\in p^{n}\bZ_p^k$ that
$$\log{g^{-1}\underline{u}(\underline{a}+\underline{t})}\equiv w+\sum t_iv_i \mod p^{n+1}$$
and we may solve for $w+\sum t_iv_i=0$ by using the assumption of the lemma. This concludes the induction step and taking the limit $n\to\infty$ proves the lemma.
\end{proof}


\section{The Dynamical Argument}\label{dynamics}

We let $$(\bG,H, \mathcal{Y}, \mu, \gor, \mathcal{F})$$ to mean one of the data sets $$(\bG_1,\HvS^{+,k_v}, \mathcal{Y}_1^+,(\pi_1)_*\mu_{v,S}, \gor_1, \mathcal{F}(v)),$$ $$(\bG_2,\HVS^{+,\theta_v}, \mathcal{Y}_2, (\pi_2)_*\mu_{v,S}, \gor_2, \mathcal{F}(\Lambda_v)),	$$
or $$(\bG_\text{joint},\LvS^{+,(k_v,e,\theta_v,e)},\mathcal{Y}_\text{joint}^+, \mu_{v,S}, \gog_1\times\gor_2,  \mathcal{F}(\Lambda_v)),$$
where in the last case the set $\mathcal{F}(\Lambda_v)$ is diagonally embedded so that it belongs to the acting group.
In the case of $\bG=\bG_1$ we also define $X_{\on{cpt}}=\mathcal{Y}^+_1$,
and in the case of $\bG=\bG_2$ or $\bG_{\text{joint}}$ we define $X_{\on{cpt}}$ as after Lemma \ref{nondivergence}. Denote by $V$ the volume of $\mathcal{Y}$ as defined in Section \ref{orbitmeasures}. Recall that in Section \ref{liealgebras} we defined heighest and lowest weight spaces of $\gor_i$ for the principal $\lsl$ and discussed the shearing behaviour of the  unipotent 
one-parameter subgroup $\{u(t):t\in\bQ_p\}$
in the principal $\SL_2$. Recall in particular that $\mathcal{F}$ is a finite set that will allow us to arrange elements in the invariant complement to have maximal lowest weight vectors (see Lemma \ref{movingalongthecomplement}).

\subsection{Quantitative Ergodic Theorem}We start by recalling and extending the definition of almost invariance given in Section \ref{single_proof}:
\begin{definition}[Almost invariant measures]
\label{invmeas}
The measure $\mu$ on $\mathcal{Y}$ is called $\varepsilon$-almost invariant w.r.t.~a Sobolev norm $\mathcal{S}_{d'}$ under
\begin{itemize} 
\item $g\in\bG{(\bQ_p)}$ if $|\mu^{g}(f)-\mu(f)|\leq \varepsilon \cS_{d'}(f)$ for all $f\in C^\infty_c(\mathcal{Y})$,
\item  a subgroup $L<K$ if it is $\varepsilon$-almost invariant under all $g\in L$,
\item $v\in\gog[1]$ (or a nilpotent $v\in\gog[0]$) if it is $\varepsilon$-almost invariant under $\exp{(\bZ_p v)}$. 
\end{itemize}
\end{definition} 
There are some easy tools concerning the notion of almost invariance. 
Indeed, if $\mu$ is invariant under $h$ and $\varepsilon$-almost invariant under $g$ w.r.t.~$\mathcal{S}_{d'}$
then
$$\left|\mu^{hg_0h^{-1}}(f)-\mu(f)\right|=\left|\mu^{g_0h^{-1}}(f)-\mu^{h^{-1}}(f)\right|\ll\varepsilon\cS_{d'}(h^{-1}\acts f)\ll\varepsilon\|h\|_S^{4d'}\cS_{d'}(f).$$ Similarly, if $\mu$ is $\varepsilon$-almost invariant under two elements $g_0$ and $g_1$ then
$$\left|\mu^{g_0g_1}(f)-\mu(f)\right|=\left|\mu^{g_0g_1}(f)-\mu^{g_1}(f)+\mu^{g_1}(f)-\mu(f)\right|$$
$$\leq\varepsilon\cS_{d'}(g_1\acts f)+\varepsilon\cS_{d'}(f)\ll\varepsilon\|g_1\|_S^{4d'}\cS_{d'}(f).$$ Finally, by the Lipschitz property (S4) of the Sobolev norm, any measure is $p^{-m}$-invariant under
all elements of $K[m]$.
We collect those facts in the next lemma.
\begin{lemma}
\label{invmeasconj}
Let $d'>0$. If $\mu$ is an $\varepsilon$-almost invariant measure under $g_0$ and $g_1$ w.r.t.~$\mathcal{S}_{d'}$, and invariant under $h$.
Then $\mu$ is 
\begin{itemize}
\item $\ll\varepsilon \|h\|_S^{4d'}$-almost invariant under $hg_0h^{-1}$ w.r.t.~$\mathcal{S}_{d'}$,
\item $\ll\varepsilon \|g_1\|_S^{4d'}$-almost invariant under $g_0g_1$ w.r.t.~$\mathcal{S}_{d'}$,
\item $\ll p^{-m}$-almost invariant under $K[m]$ w.r.t.~$\mathcal{S}_{d'}$
for all~$m\geq 0$ and~$d'\geq d_0$.
\end{itemize}
\end{lemma}
In what follows, we denote the Haar measure on $\bQ_p$ by $\on{d}\!t$ or $|B|$ for
any measurable $B\subset \bQ_p$ and normalize it such that $|\bZ_p|=1$.

\begin{definition}[Discrepancy and generic points]
\label{discrepancy} Fix some integer $M\geq\frac{5}{{\ref{exp:decay2}}}+6$ (where $\ref{exp:decay2}$ is the rate of the decay of matrix coefficients of $u_t$ specified in (S5) after Theorem \ref{propertytau}).
We denote the $p$-adic ball in $\bQ_{p}$ at $ap^{-M\ell}$ with $a\in\bZ_{p}^{\times}$ and radius $p^{(M-1)\ell}$ by 
\[
\cB_{\ell}({a})=\Bigl\{t\in\bQ_p\mid | t-ap^{-M\ell}|_p\leq p^{(M-1)\ell}\Bigr\}
\]
and note that $|\cB_{\ell}({a})|=p^{(M-1)\ell}$. 
Using these balls we define the discrepancy of the average of $u_t$ over $\cB_\ell(a)$ by
$$D_{a,\ell}(f)(x)=\frac1{p^{(M-1)\ell}}\int_{\cB_{\ell}({a})}f(xu_{t})\on{d}\!t-
 \mu(f),$$
where we used the abbreviation~$\mu(f)=\int_\mathcal{Y} f\on{d}\!\mu$.
A point $x\in X$ is called $\ell_{0}$-\textit{generic} w.r.t.~
a Sobolev norm $\mathcal{S}_{d'}$ for some $\ell_0\geq 1$
if for any integer $\ell \geq \ell_{0}$, any $a\in\bZ_{p}^{\times}$ and any smooth $f\in C_c^\infty(\mathcal{Y})$ we have 
\[
 |D_{a,\ell}(f)(x)|\leq p^{-\ell}\cS_{d'}(f).
\]
We say that a point $x \in X$ is $\left[\ell_{0},\ell_{1}\right]$-\textit{generic} w.r.t.~$\mathcal{S}_{d'}$ for some $1\leq \ell_0\leq \ell_{1}$ if the above condition holds for all integers $\ell_{0} \leq \ell \leq \ell_{1}$ (where~$\ell_1=\infty$ corresponds to~$\ell_0$-generic).
A point $x \in X$ is called $(\ell_0,\ell_1,\mathcal{F})$-\textit{generic} if $xg$ is $[\ell_0,\ell_1]$-generic
for all $g\in\mathcal{F}$. 
\end{definition}

We note that it suffices to consider real-valued functions in the above definitions.
The following is an effective version of a pointwise ergodic theorem and is an adaptation of \cite[Sect.\ 9]{EMV} or \cite[Sect.\ 7.5]{EMMV} to our setting.
\begin{proposition}[Quantitative Ergodic Theorem]
\label{qet}
Let $\mathfrak{s}$ be a Lie algebra containing~$\mathfrak{h}$
and suppose that $\mu$ is $p^{-L}$-almost invariant under $\exp(\mathfrak{s}[1])$ w.r.t.~$\mathcal{S}_{d'}$ for some $d' > d_{0}$ and $L > 0$. Then there exists $\beta \in (0,1/2)$ and $d_2 = d_2(d') > d'$, so that the measure of the fraction of points $(x,s) \in X \times \exp(\mathfrak{s}[1])$ (w.r.t.~the product measure of~$\mu$ 
and the Haar measure on $\exp(\mathfrak{s}[1])$) for which $x.s$ is not $\left(\left[\ell_0,\beta L\right],\mathcal{F}\right)$-generic 
with respect to $\mathcal{S}_{d_2}$ is $\ll p^{-\ell_{0}}$.
\end{proposition}

Notice that this Proposition implies in particular that the $\mu$-measure of the set of points that are not $(\ell_{0},\mathcal{F})$-generic for $\mu$ w.r.t.~$\mathcal{S}_{d_2}$ is $\ll p^{-\ell_{0}}$. In the following proof we will use the integers  $d_0\leq d' < d''<d_1<d_2$ with $d'' = d' + d_{0} + 1$ and the orthonormal basis $\{e_k\}$ of the completion of $C_c^\infty(\mathcal{Y})$ with respect to $\cS_{d_2}$ as in property (S2) of Proposition \ref{sobolev} (applied to~$d''$).

\begin{proof}
We defined $S[1] = \exp(\mathfrak{s}[1])$.
The proposition will follow from Chebychev's inequality after estimating
\begin{align}
A&=\frac{1}{m_S(S[1])}\int_{X \times S[1]}D_{a,\ell}(f)(xs)^{2}\on{d}\!\mu(x) \on{d}\!m_S(s),\nonumber\\
&= \frac{1}{m_S(S[1])}\int_{S[1]}\int_{X}D_{a,\ell}(f)^{2}\on{d}\!\mu^{s}(x) \on{d}\!m_S(s)\nonumber\\
& = \frac{1}{m_S(S[1])}\int_{S[1]}\int_{X}F \on{d}\!\mu^{s}(x) \on{d}\!m_S(s) +  \mu(f)^{2} \label{startingtofinish},
\end{align}
where $\mu^{s}$ denotes the push-forward measure obtained from $\mu$
with respect to the map $x\mapsto xs$  and
\begin{align*}
F(x) & = D_{a,\ell}(f)(x)^{2} - \mu(f)^{2}\\
& = \Big(\frac{1}{p^{(M-1)\ell})}\int_{\cB_{\ell}({a})}f(xu_{t})\on{d}\!t\Big)^{2} - \frac{2\mu(f)}{p^{(M-1)\ell}}\int_{\cB_{\ell}({a})}f(xu_{t})\on{d}\!t
\end{align*}
is a compactly supported smooth function satisfying
\[
 \int F\on{d}\!\mu=\int D_{a,\ell}(f)^{2}\on{d}\!\mu- \mu(f)^{2} .
\]

 By the assumed almost invariance of~$\mu$ (defined via smooth functions of compact support)
  the first summand in \eqref{startingtofinish}
 equals~$\int F\on{d}\!\mu+O(p^{-L}\mathcal{S}_{d'}(F))$, and hence we now obtain  
\begin{equation}\label{eq:qet}
A \ll  \int_{X}\left|D_{a,\ell}(f)\right|^{2}\on{d}\!\mu+p^{-L} \mathcal{S}_{d'}(F).
\end{equation} 
We start bounding the $L^{2}$-norm of $D_{a,\ell}(f)$. Using invariance of $\mu$ under $u_t$ and Fubini's theorem we see that
$$
\|D_{a,\ell}(f)\|^{2}_{L^{2}(\mu)}=\sqint\left(\langle u_{s-t}\acts f,f\rangle-\mu(f)^{2}\right) \on{d}(s,t),
$$
where $\sqint$ denotes the normalized integral over the box $\cB_{\ell}({a})\times \cB_{\ell}({a})$
with respect to the product measure for the Haar measure on $\bQ_p$. We want to apply Theorem~\ref{propertytau} to those $(s,t)$ for which $|s-t|_p\geq p^{\alpha \ell}$ with $\alpha>0$ as below. The set of points $(s,t)$ in the box for which $|s-t|_p< p^{\alpha \ell}$  has measure at most $p^{(M-1)\ell}p^{\alpha \ell}$. Splitting the 
above normalized integral accordingly we see therefore that 
$$\|D_{a,\ell}(f)\|^{2}_{L^{2}(\mu)}\ll p^{-{\ref{exp:decay2}} \ell\alpha}\cS_{d_0}(f)^{2}+\frac{p^{\ell\alpha}}{p^{(M-1)\ell}}\cS_{d_0}(f)^{2}$$
by Proposition~\ref{sobolev}~(S1) and (S4) after Theorem~\ref{propertytau}. If we choose $\alpha=\frac5{\ref{exp:decay2}}$ and use our choice of $M\geq\alpha+6$ we arrive at the bound $\ll p^{-5\ell}\cS_{d_0}(f)^{2}$.
 
We now estimate the second expression on the right hand side of (\ref{eq:qet}). Using Proposition~\ref{sobolev}~(S1), (S3), and (S5), there exists a constant $\kappa > 0$ such that
\begin{align*}
\mathcal{S}_{d'}(F) & \ll \mathcal{S}_{d''}\Big(\frac{1}{p^{(M-1)\ell}}\int_{\mathcal{B}_{\ell}(a)}f(xu_{t})\on{d}\!t\Big)^{2} + \mu(f)\mathcal{S}_{d''}\Big(\frac{1}{p^{(M-1)\ell}}\int_{\mathcal{B}_{\ell}(a)}f(xu_{t})\on{d}\!t\Big) \\
& \ll \Big(\frac{1}{p^{(M-1)\ell}} \int_{\mathcal{B}_{\ell}(a)}\mathcal{S}_{d''}(u_{t}\acts f)\on{d}\!t\Big)^{2} + \left\|f\right\|_{\infty} \Big(\frac{1}{p^{(M-1)\ell}} \int_{\mathcal{B}_{\ell}(a)}\mathcal{S}_{d''}(u_{t}\acts f)\on{d}\!t\Big)\\
& \ll p^{\ell\kappa d''}\mathcal{S}_{d''}(f)^{2}.
\end{align*}
This implies that the second expression on the right hand side of (\ref{eq:qet}) is
\[
\ll p^{-L} p^{\ell \kappa d''}\mathcal{S}_{d''}(f)^{2}.
\]

Now choose $\beta \in (0,1/2)$ so that $p^{-L}p^{\ell \kappa d''} \leq p^{-5\ell}$ whenever $p^{\ell} \leq p^{\beta L}$. Therefore, with $p^{\ell} \leq p^{\beta L}$,
\[
A=\tfrac{1}{m_S(S[1])}\int_{S[1]}\int_{X}\left|D_{a,\ell}(f)\right|^{2}\on{d}\!\mu^{s}(x) \on{d}\!m_S(s) \ll p^{-5 \ell}\mathcal{S}_{d''}(f)^{2}.
\]
Chebychev's inequality now gives
$$\tfrac{1}{m_S(S[1])}\mu\times m_S(\set{(x,s):|D_{a,\ell}(f)(xs)|\geq \lambda})\ll \lambda^{-2}p^{-5\ell}\cS_{d''}(f)^{2}$$
for any~$\lambda>0$.
Note that given $\ell$, there are $(p-1)p^{\ell-1}$ mutually disjoint balls 
of the form $\cB_{\ell}({a})$. 
Let $A_{\ell}$ be a set of representatives of these $<p^{\ell}$ many different midpoints $a\in\bZ_p$. 
We apply the above inequality to the set
$$
B=\bigcup_{\ell,a\in A_{\ell}, k\geq1}\set{(x,s):|D_{a,\ell}(e_{k})(xs)|\geq c\cS_{d_1}(e_{k})p^{-\ell}}, 
$$
where the union runs over all~$\ell$ with~$\ell_0\leq \ell\leq \beta L$ and we define
the absolute constant $c$ below. This gives
$$\tfrac{1}{m_S(S[1])}\mu\times m_S(B)\ll c^{-2}\sum_{\ell=\ell_{0}}^{\lfloor\beta L\rfloor}\sum_{k}p^{\ell}p^{2\ell} p^{-5\ell}\frac{\cS_{d''}({e_{k}})^2}{\cS_{d_1}(e_{k})^{2}}.$$
By Proposition~\ref{sobolev}~(S2) the sum over $k$ is finite and thus $\tfrac{1}{m_S(S[1])}\mu\times m_S(B)\ll c^{-2}p^{-\ell_{0}}$. This implies the claim of the proposition as follows: Recall that $e_{k}$ is an orthonormal basis with respect to $\cS_{d_2}$, let $(x,s)\not\in B$ and $f=\sum f_{k} e_{k}\in C_c^{\infty}(\mathcal{Y})$, and apply Cauchy-Schwarz to obtain
$$|D_{a,\ell}(f)(xs)|=\Bigl|\sum f_{k}D_{a,\ell}(e_{k})(xs)\Bigr|\leq cp^{-\ell}\left(\sum f_{k}^{2}\right)^{\frac12}\left(\sum \cS_{d_1}(e_{k})^{2}\right)^{\frac12}.
$$
Putting $c =\left(\sum \cS_{d_1}(e''_{k})^{2}\right)^{-\frac12}$ implies therefore that for all $(x,s)\not\in B$
$$|D_{a,\ell}(f)(xs)|\leq p^{-\ell}\cS_{d_2}(f).$$

It is now easy to obtain the conclusion of the proposition using the measure preserving
action of the elements~$g\in\mathcal{F}$ on~$X\times S[1]$ defined by~$g.(x,s)\mapsto(xg^{-1},gsg^{-1})$. 
It follows that $B'=\bigcup_{g\in\mathcal{F}}g.B$
satisfies essentially
the same estimate as~$B$ and that $(x,s)\in\mathcal{Y}\times S[1]\setminus B'$
implies that~$xs$ is $\left(\left[\ell_0,\beta L\right],\mathcal{F}\right)$-generic.
\end{proof}

\subsection{Tuples of generic points in a single factor}
\label{singlefactorgenericpoints}
Recall that $V$ denotes the volume of $\Gamma\HvS^+ k_v$ respectively $\Gamma \HVS^+ \oT$ which we defined as $V=m_{H}(\Theta)^{-1}$. Combining the adjustment claim in 
Lemma \ref{adjustmentlemma} and the existence of generic points in 
Proposition \ref{qet} gives rise to nearby generic points (see also \cite[Lemma~7.7]{EMMV}).

\begin{proposition}[Nearby generic points]
\label{genericpoints}
There exists $\ell_0>0$ and $d_2>0$ such that for any $m>0$ with $2m_{G_S}(\Omega[m])^{-1}<V$ there exist $z_1, z_2\in X_{\on{cpt}}$ and $g\in \Omega[m]$ satisfying  
\begin{itemize}
	\item $z_2=z_1g$,
\item $z_1$, $z_2$ are both $\ell_0$-generic for $\mu$ w.r.t.~$\cS_{d_2}$,
\item $g_p=\exp{r}$ where $r\in\gor$ satisfies $\|r\|_p=\|r^{\on{lw}}\|_p>0$.
\end{itemize}
\end{proposition}
\begin{proof}
Let $E\subset\mathcal{Y}$ be the set of $(\ell_0,\mathcal{F})$-generic points from Proposition \ref{qet} (applied with~$\mathfrak{s}=\mathfrak{h}$)
so that $\mu(E^c)\ll p^{-\ell_{0}}.$ 
Hence we may choose $\ell_0$ such that $\mu(E)$ exceeds measure $0.99$ (independent of $p$). Then 
the $\mu$-measure of the set
$$
E'=\Bigl\{x\in X_{\on{cpt}}: m_{H_p}\bigl(\bigl\{h_p\in K[1]\cap H_p:xh_p\in E\bigr\}\bigr)>\tfrac34m_{H_p}(K[1])\Bigr\}
$$
exceeds $\frac34$ by applying Chebychev's inequality and Fubini's theorem to the function $\mathbbm{1}_{X_{\on{cpt}}\setminus E}(xh_p)$ in $(x,h_p)\in X\times K[1]\cap H_p$.

We are now in the position to use the pigeon hole principle in Lemma \ref{transversalpoints} to $E'$ and $\Omega[m]$ and deduce that there are
$y_1,y_2\in E$ such that $y_2=y_1g_0$ where $g_0\in \Omega[m]$ and $g_0\not\in H$. By definition of $E'$, there are sets $A_i\subset K[1]$ for $y_i$ such that $y_iA_i\subset E$ and are of relative measure $>\frac34$. By the adjustment statement in Lemma~\ref{adjustmentlemma} we deduce that there  are $x_i=y_i\alpha_i\in E$ where $\alpha_i\in A_i$ is such that the new displacement $g'=\alpha_1^{-1}g_0\alpha_2$ between $x_1$ and $x_2$ satisfies $g_p'\in\exp\gor[m]$. Since $\alpha_1,\alpha_2\in H_p$ but $g_0\not\in H$,
we also have $g'\not\in H$ and can exclude the possibility that $g_p'=e$ by Lemma \ref{stabilizer}. 
We now use the additional property for genericity concerning $\mathcal{F}$. By definition
of $\mathcal{F}$ in Lemma \ref{movingalongthecomplement} and the set $E$ there
exists some  $n\in \mathcal{F}$ such that
 $z_i=x_in$ are $\ell_0$-generic and have a displacement $g=n^{-1}g'n$ satisfying $\|r\|_p=\|r^{\on{lw}}\|_p$ where $g_p=\exp{r}$. 
\end{proof}

In the following we will always work with the~$\ell_0$ as in Proposition~\ref{genericpoints}.

\subsection{Additional Invariance}

The next lemma shows the existence of an admissible polynomial in the sense of \cite[Section 6.8]{EMMV}. 
Here we have the additional assumption that the projection to the 
lowest weight space  is large, 
which ensures optimal behaviour with respect to the `time lapse'
appearing in the next lemma.

\begin{lemma}
\label{adpol}
Let $r\in\gor[0]$ and assume that $\|r\|_p=\|r^{\on{lw}}\|_p$. Then there exists a constant 
$\consta\label{exp:error}>0$, some $T\in \bQ_{p}$ with  $\|r\|_p^{-1/d_\gor}\geq|T|_p\geq p^{-1}\|r\|_p^{-1/d_\gor}$,
 and an $\gor^{{\on{hw}}}$-valued monomial $q$ of homogeneous degree $d_\gor\leq 4$ satisfying $\max_{t\in \bZ_{p}}\|q(t)\|_p\in [p^{-{d_\gor}},1]$ and
\begin{equation*}
\on{Ad}_{u_{t}}(r)=q(t/T)+O(\|r\|_p^{1/d_\gor}) \text{ for all } t\in T\bZ_{p}.
\end{equation*}
\end{lemma}

\begin{proof}
Write $P(t)=\on{Ad}_{u_{t}}(r)$ then $P(t)=\sum \frac{t^{\ell}}{\ell!}\on{ad}^{\ell}_{v}r$ where $v\in\mathfrak{g}[0]\setminus\mathfrak{g}[1]$ 
is chosen such that $\exp{tv}=u_{t}$. The coefficient for the highest degree term of $P$ is 
\[
 c_{d_\gor}=\frac{1}{{d_\gor}!}\on{ad}^{d_\gor}_{v}(r)=\frac{1}{{d_\gor}!}\on{ad}^{d_\gor}_v(r^{\on{lw}}).
\]
 Since $d_\gor\leq 4$ and we may assume $p\geq 5$,
$|{d_\gor}!|_p=1$ and 
\[
 \|c_{d_\gor}\|_p=\|\on{ad}^{d_\gor}_v(r^{\on{lw}})\|_p=\|r^{\on{lw}}\|_p=\|r\|_p.
\] 
We note that $P(0)=r$ and that for $t$ with $|t|_p>1$ we have $\|P(t)\|_p=\|c_{d_\gor}t^{d_\gor}\|_p=\|r\|_p|t|_p^{d_\gor}$. 
Moreover, $\Vert P(t)-c_{d_\gor}t^{d_\gor}\Vert_p=O(\Vert r t^{d_\gor-1}\Vert_p)$.

We define $T$ by taking $j$ with $0\leq j<d_\gor$ such that $T=\left(\frac{p^{j}}{\|c_{d_\gor}\|_p}\right)^{1/d_\gor}$
exists in $\bQ_p$. For $t\in\bZ_p$ this gives that
$$P(tT)=p^jt^{d_\gor}w^{{\on{hw}}}+O\left(\|r\|_p^{1/d_\gor}\right)$$
where $w^{{\on{hw}}}$ is a vector in $\gor^\text{hw}$ of norm one 
and we set $q(t)=p^jt^{d_\gor}w^{{\on{hw}}}$.
\end{proof}

The following step may be viewed as an effective version
of the shearing properties appearing in Ratner's measure classification theorem (see \cite{ratner,margulis}).

\begin{proposition}
\label{addinv}
There exist absolute constants 
$\consta\label{exp:addinvinp},\consta\label{exp:addinv}>0$
with the following property. Let $d_2>0$ and assume that  $x_{1}, 
x_{2}=x_{1}g\in \mathcal{Y}$ with $g\in \Omega[1]$ are $[\ell_0,\ell_1]$-generic for $\mu$ w.r.t.\ the Sobolev norm $\mathcal{S}_{d_2}$,  
that $g_p=\exp{r}$ with $r\in\gor[0]$ and $\|r^{{\on{lw}}}\|_p=\|r\|_p$ and $\ell_1\geq\frac{1}{d_\gor M}\log_p(\|r\|_p^{-1})$.
Then there exists $w\in\gor^{\on{hw}}$ of norm $\|w\|_p=1$ under which $\mu$ is $\ll p^{\ref{exp:addinvinp}}\|r\|_p^{\ref{exp:addinv}}$-almost invariant, i.e.\
$$\left|\mu(f)-\exp{(tw)}_{*}\mu(f)\right|\ll p^{\ref{exp:addinvinp}}\|r\|_p^{\ref{exp:addinv}}\cS_{d_2}(f)$$
for all $t\in\bZ_p$ and $f\in C^\infty_c(\mathcal{Y})$.
\end{proposition}
\begin{proof}
We divide the proof into several steps following \cite[Section 6.9 and 7.8]{EMMV}.

\textit{Step 1, Applying Lemma~\ref{adpol}.}
If $g=g_\infty\exp{r}$ then we may write 
\(
	x_2u_t=x_1u_tg_\infty\exp{\Ad[u_t](r)}.
\)
If $t_0=ap^{-M\ell}$ is a midpoint of $\cB_{\ell}({a})$ (introduced in Definition~\ref{discrepancy}) then for any $t\in \cB_{\ell}({a})$ we have
\[
\left|t^{d_\gor}-t_0^{d_\gor}\right|_p=\left|\left(t^{d_\gor}-t_0^{d_\gor}+t_0^{d_\gor}\right)-t_0^{d_\gor}\right|_p\ll|t-t_0|_p|t_0|_p^{d_\gor-1}\leq |t_0|_p^{d_\gor-1/M}.
\]
Let now $T$ and $q$ be as in Lemma~\ref{adpol} applied to $r$. Thus for any $t_0$ with $|t_0|_p\leq |T|_p$ we obtain
\begin{multline*}
\Ad[u_t](r)=q(t/T)+O(\|r\|_p^{1/d_\gor})=q(t_0/T)+O(|T|_p^{-1/M}+\|r\|_p^{1/d_\gor})\\
=q(t_0/T)+O(p\|r\|_p^{1/Md_\gor})
\end{multline*}
since $|T|_p\geq p^{-1}\|r\|_p^{-1/d_\gor}$. This gives $\exp{\Ad[u_t](r)}=\exp{q(t_0/T)}\widetilde{g}$ with $\widetilde{g}\in K[c-\log_p(p\|r\|_p^{1/Md_\gor})]$
for some constant $c>0$ coming from the $O$-notation.
We will see below that we can set
$\ref{exp:addinv}=(2Md_{\mathfrak{r}})^{-1}$ 
and note that if $\|r\|_p^{\ref{exp:addinv}}\geq c'$
for some absolute constant $c'$, then the conclusion of the proposition holds trivially by the Sobolev embedding claim in Proposition~\ref{sobolev}(S1) and adjusting the implicit constant in the conclusion.
 Hence we may assume that $\|r\|_p$ is sufficiently small such that $\widetilde{g}\in K[1]$, and so we can apply the Lipschitz property of Proposition~\ref{sobolev} (S4) to see
 that
\begin{equation}\label{almost-parallel}
	f(x_2u_t)=f(x_1u_tg_\infty \exp{q(t_0/T)})+O(p\|r\|_p^{1/Md_\gor}\cS_{d_2}(f)).
\end{equation}

We distinguish between the two cases $|t_0|_p\geq |T|_p^{1/2}$ and $|t_0|_p\leq|T|_p^{1/2}$.
If $t_0$ is small in the sense that $|t_0|_p\leq|T|_p^{1/2}$ then also $\|q(t_0/T)\|_p\leq |T|_p^{-d_\gor/2}\leq\|r\|_p^{1/2}$, which by the third property of Lemma~\ref{invmeasconj} implies that $\mu$ is $\|r\|_p^{1/2}$-almost invariant under $\exp{q(t_0/T)}$. 
 
In the former (and more interesting) case, we have 
\[
 p^{M\ell}=|t_0|_p\geq |T|_p^{1/2}\geq p^{-1}\|r\|_p^{-1/2d_\gor}.
\]
As before, we may assume that $\|r\|_p\leq p^{-2d_\gor M(\ell_0+1)}$, for otherwise we may increase $\ref{exp:addinvinp}$ to ensure that $p^{\ref{exp:addinvinp}}\|r\|_p^{\ref{exp:addinv}}\geq 1$ and apply the Sobolev embedding again. Therefore, we are reduced to the case $\ell\geq\ell_0$. 
Recall now that $D_{a,\ell}(f)(x)$ from Definition~\ref{discrepancy} is the discrepancy between $\int f\on{d}\!\mu$ and the normalized integral over $xu(\cB_{\ell}({a}))$ and by assumption on $[\ell_0,\ell_1]$-genericity, $|D_{a,\ell}(f)(x_i)|\leq p^{-\ell}\cS_{d_2}(f)$ for $i=1,2$ and $\ell_0\leq \ell \leq \ell_1$.
Using this for both points together with \eqref{almost-parallel}
we obtain
\[
	\mu(f)=\mu(f^{g_\infty \exp{q(t_0/T)}})+O(p\|r\|_p^{1/2 Md_\gor}\cS_{d_2}(f)+p\|r\|_p^{1/Md_\gor}\cS_{d_2}(f)).
\]
as long as $|t_0|_p\leq p^{M\ell_1}$ and thus for all $|t_0|_p\leq |T|_p$ if
$p^{M\ell_1}\geq\|r\|_p^{-1/d_\gor}\geq |T|_p$.

\textit{Step 2, Removing the real displacement for $|t_0|_p\geq|T|_p^{1/2}$.}

The above shows that $\mu$ is $\ll\varepsilon=p\|r\|_p^{\ref{exp:addinv}}$-almost invariant for $\ref{exp:addinv}=(2Md_{\mathfrak{r}})^{-1}$ under the element $g_\infty\exp(q(t_0/T))$ whenever $|t_0|_p\geq|T|_p^{1/2}$. 
Applying this also to $2t_0$ we obtain
that $\mu$ is $\ll\varepsilon$-almost invariant
under $g_\infty\exp(q(2t_0/T))=g_\infty\exp(2^{d_\mathfrak{r}}q(t_0/T))$. Applying Proposition \ref{sobolev}(S3) and Lemma \ref{invmeasconj} we can take the quotient and obtain that $\mu$ is $\ll\varepsilon$-almost invariant under $\exp((2^{d_\mathfrak{r}}-1)q(t_0/T))$. 
In the trivial first case $|t_0|_p\leq|T|_p^{1/2}$
we can repeat the argument for $(2^{d_\mathfrak{r}}-1)q(t_0/T)$ instead of $q(t_0/T)$.
As we only will need the almost invariance and not how we
came to the polynomial we will simply write again $q$ for the polynomial $(2^{d_\mathfrak{r}}-1)q$. 

\textit{Step 3, Rescaling $q$.}

The coefficient of the monomial $q$ from Lemma~\ref{adpol} might have $p$-adic norm
that is only  as large as $p^{-d_\gor}$. 
However, since $q$ is $r^{\on{hw}}$-valued, we may conjugate by the diagonal element $a$ of the principal $\on{SL}_2(\bQ_p)$ satisfying $\|a\|_p=p$ to obtain elements of norm bigger than $1$. 
Let $m\leq 1$ be such that $\on{Ad}_a^mq(1)\not\in\gor[1]$ 
and let $n\in\{0,\ldots,d_\gor\}$ such that $p^{n}=\|\on{Ad}_a^mq(1)\|_p$, 
or equivalently $p^{n}\on{Ad}_a^mq(1)\in \gor[0]\setminus\gor[1]$. 
Apply the first point of Lemma~\ref{invmeasconj} to $h=a$ 
and $g_0=\exp{q(t)}$ to see $\ll p^{4d_\gor d_2}\varepsilon$-almost invariance under $\exp{\Ad[a]^mq(t)}$
and the second point (applied $p^n$-times to $g_0=g_1=\exp{\Ad[a]^mq(t)}$) to obtain $\ll p^np^{4d_2}p^{4d_\gor d_2}\varepsilon$-almost invariance under $\left(\exp{\on{Ad}_a^mq(t)}\right)^{p^{n}}=\exp{{p^{n}}\on{Ad}_a^mq(t)}$ for all $t\in\bZ_p$.
We may replace therefore $q$ with ${p^{n}}\on{Ad}_a^mq(t)$ which is a scalar multiple of $q$.

\textit{Step 4, From $q$ to linear displacement.}

Write $q(t)=t^{d_\gor} q(1)$ and using the Hilbert-Waring theorem (\cite{hilbert}) that says that any integer can be written as sum of $d_\gor$-powers with at most $g(d_\gor)<\infty$ terms, we see that 
\[
\bZ q(1)\subset\left\{\sum_{i=1}^{g(d_\gor)}q(t_i):t_i\in\bZ_p \mbox{ for } i=1,\dots,d_\gor\right\}
\] 
has dense image in $\bZ_pq(1)$. By the second property of Lemma~\ref{invmeasconj} we get $\ll g(d_\gor) p^np^{4d_2}p^{4d_\gor d_2}\varepsilon$-almost-invariance under $\exp{\bZ q(1)}$. Using density and the last property of Lemma~\ref{invmeasconj} we may fill the gaps to deduce $\ll p^np^{4d_2}p^{4d_\gor d_2}\varepsilon$-almost invariance under $\mu$. Collecting the $p$ terms we deduce the promised $p^{\ref{exp:addinvinp}}\|r\|_p^{\ref{exp:addinv}}$-almost invariance of $\mu$ under $w=q(1)$. 
\end{proof}

With that we are now ready to finish the proof of the
equidistribution on the single factors (i.e.\ the
first statement of Theorem \ref{maindynamicalresultsingle})
in Section \ref{single_proof}.

\begin{proof}[Proof of Proposition \ref{almostinvariance}]
We want to maximize $m$ in Proposition \ref{genericpoints}, which is supposed to satisfy 
$2m_{G_{i,S}}(\Omega[m])^{-1}<V$. 
Recall from Section \ref{smallneighborhood} that $\Omega[0]=\Omega_S=\Omega_\infty\times K[0]$ was chosen to be injective for the orbit map for all points in $X_{\text{cpt}}$, and we remarked that for $\Omega_\infty$, a ball of radius $\gg p^{-\ref{exp:uniforminjrad}}$ suffices which implies that
$m_{G_{i,S}}(\Omega[0])\gg p^{-\ref{exp:uniforminjrad}\dim(\bG_i)}$.
In view of Section \ref{sec:discvol}, we also have $D^{\ref{exp:discvol1}}\ll V$ (resp.\ $D^{\ref{exp:discvol2}}\ll V$ for $i=2$). Combining these two, we want the following inequality to hold:
\[
m_{G_{i,S}}(\Omega[m])^{-1}=p^{m\dim{\bG_i}} m_{G_{i,S}}(\Omega[0])^{-1}\leq c_0 p^{(m+\ref{exp:uniforminjrad})\dim(\bG_i)}\overset{!}{\leq} c_1D^{\ref{exp:discvol1}}\leq V
\]
resp.\ with $\ref{exp:discvol1}$ replaced by $\ref{exp:discvol2}$. But we can find $\kappa>0$ such any $m$ for which $p^m\leq D^\kappa$ will satisfy the above (assuming, as we may, that $D$ is suffiently big).

We apply Proposition~\ref{genericpoints} for the maximal $m$ such that $p^m\leq D^\kappa$ from which we get a tuple of generic points with $p$-adic displacement $r$. 
Applying Proposition~\ref{addinv} produces $z\in \gor_i^{\on{hw}}[0]$ for which $\mu$ is $\ll p^{\ref{exp:addinvinp}+1}D^{-\ref{exp:addinv}\kappa}$-almost invariant, and we use the bound on $p$ one more time to deduce that $\mu$ is $\ll D^{-\ref{exp:addinvD}}$-almost invariant. 
By Lemma~\ref{generation}, we can conjugate $z$ by elements of $\HvS[p]\cap K[0]$ (resp.\ $\HVS[p]\cap K[0]$) to form a $\bZ_p$-basis $\{z_j\}$ of $\gor_i[0]$. Note that by the first bullet point of Lemma~\ref{invmeasconj}, $\mu$ is $D^{-\ref{exp:addinvD}}$-almost invariant under each $z_j$. As $\mu$ is already invariant under $\HvS[p]$ (resp.\ $\HVS[p]$), the second bullet point of Lemma~\ref{invmeasconj} combined with the implicit function theorem  in Lemma~\ref{hensel}(1) we conclude that $\mu$ is $D^{-\ref{exp:addinvD}}$-almost invariant under $\bG_i{(\bQ_p)}^+\cap K[1]$ w.r.t.~$\mathcal{S}_{d_2}$.
\end{proof}


\section{Proof of Joint Equidistribution}
\label{proofofequidistribution}
\subsection{Almost Invariance for Joint equidistribution}

The aim of this subsection is to 
prove almost invariance of $\Gamma\LvS^+(k_v,e,\oT,e)\subset\mathcal{Y}_{\on{joint}}^+$
under $G_{p}^{+}$. To achieve this, we use equidistribution of $(\pi_1)_*\mu_{v,S}$  and  $(\pi_2)_*\mu_{v,S}$ to produce closeby generic points that differ ``significantly'' along the invariant complement $\goh_1 \oplus \gor_{1} \oplus \gor_{2}$ which arguing as before leads to  almost invariance under a Lie algebra~$\mathfrak{f}<\gog_{1} \oplus \gog_{2}$.  We then iterate the argument relying on the almost invariance under $\mathfrak{f}$ in order to create a new direction under which we are almost invariant. 
For this we have to work with the more general notion of generic points (using a lower bound~$\ell_0$ as before but also
an upper bounds~$\ell_1$ as in Definition~\ref{discrepancy}), which allows us to
work with points outside the orbit (see Proposition~\ref{qet}). 
This is the reason why we have to produce a ``significant''
displacement as we are not any longer allowed to use the shearing argument arbitrarily 
far from the original points.
Eventually we obtain almost invariance under the full Lie algebra~$\gog_{1} \oplus \gog_{2}$. 
This will allow us  to repeat the convolution step from Section~\ref{convolution}.

To allow for the above mentioned iteration of the argument 
we suppose that the measure $\mu_{v,S}$ is 
$\ll p^{\ref{exp:involdinp}}D^{-\ref{exp:involdinD}}$
almost invariant under a Lie algebra~$\mathfrak{f}<\gog_{1} \oplus \gog_{2}$
for some $\consta\label{exp:involdinp}>0$, $\consta\label{exp:involdinD}>0$
and some choice of Sobolev norm~$\cS_{d'}$. 
We may assume that $\mathfrak{f}$ contains the $p$-adic 
Lie algebra  $\Delta_\goh$ of the acting group $L_{v,p}$
and let~$\mathfrak{f}'<\gog_{1} \oplus \gog_{2}$ be an undistorted complement
such that~$\gog_{1} \oplus \gog_{2}=\mathfrak{f}\oplus\mathfrak{f}'$. 
We will start the iteration with the complement
\[
 V=\goh_{1} \oplus \gor_{1} \oplus \gor_{2}
\] 
to the Lie algebra $\Delta_\goh$ and will see in the inductive step that whenever
we can increase~$\mathfrak{f}$ the new algebra still has an undistorted complement.

\begin{proposition}
\label{genericpointsjoint}
Suppose~$\mathfrak{f}\neq\gog_1\oplus\gog_2$. 
There exists $\consta\label{exp:addinvinpJointFinal}>0$, $\consta\label{exp:addinvJointFinal}>0$ (which depend on $\ref{exp:involdinp},\ref{exp:involdinD}$)  
and $w\in(\mathfrak{f}')^{\on{hw}}$ of norm $\|w\|_p=1$ under which $\mu_{v,S}$ is $\ll p^{\ref{exp:addinvinpJointFinal}}D^{-\ref{exp:addinvJointFinal}}$-almost invariant, i.e.\
$$\left|\mu_{v,S}(f)-\exp{(tw)}_{*}\mu_{v,S}(f)\right|\ll p^{\ref{exp:addinvinpJointFinal}}D^{-\ref{exp:addinvJointFinal}}\cS_{d_2}(f)$$
for all $t\in\bZ_p$ and $f\in C_c^\infty(\mathcal{Y}_{\text{joint}}^+)$ and some fixed $d_2=d_2(d')>d'$.
\end{proposition}

\begin{proof}
	Recall that~$\mathfrak{r}_2$ is irreducible w.r.t.~the adjoint action of~$\mathfrak{h}$
	and that~$\mathfrak{g}$ contains only one subspace isomorphic to~$\mathfrak{r}_2$. Hence we
	have~$\mathfrak{r}_2<\mathfrak{f}'$ or~$\mathfrak{r}_2<\mathfrak{f}$. 
	
We first assume that~$\mathfrak{f}'=\mathfrak{f}'_1\oplus\mathfrak{r}_2$ for some~$\mathfrak{f}_1'<\mathfrak{g}_1$.  
Below we will choose a small~$\delta>0$ (only depending 
on the parameters in the assumed almost invariance, 
the effective equidistribution on~$\mathcal{Y}_2$, and derived parameters).
Using this number $\delta$, the volume $V$ 
of $\Gamma\LvS^+(k_v,e,\oT,e)\subset\mathcal{Y}_{\on{joint}}^+$
and the sets $\Omega[\cdot]$ as in Section \ref{liefacts},
we define the maximal integer $m$ such that  
$$
m_{G_{\textrm{joint},S}^+}(\Omega[m])\geq V^{-\on{dim}(\bG_{\textrm{joint}})\delta}.
$$
We note that $m_{G_{\textrm{joint}S}^+}(\Omega[m])\asymp p^{-m\dim\bG_{\textrm{joint}}}$. 
By maximality of $m$ this gives
\begin{equation}\label{delta1eq}
m_{G_{\textrm{joint},S}^+}(\Omega[m])\ll p^{\dim\bG_{\textrm{joint}}} V^{-\on{dim}(\bG_{\textrm{joint}})\delta} \mbox{ and }
p^{-1}V^{\delta}\ll p^{m}\ll V^{\delta}. 
\end{equation}

We start by showing that there exists $\ell_0>0$, 
$z_1, z_2=z_1g\in X_{\on{cpt}}$, $g\in\Omega[m]$
such that 
\begin{itemize}
\item $z_1$, $z_2$ are both $[\ell_0,\beta(\on{log}_{p}(D)\ref{exp:involdinp}-\ref{exp:involdinD})]$-generic for $\mu_{v,S}$ with respect to $\cS_{d_2}$, where $\beta$ is chosen as in Proposition \ref{qet}.
\item $g_p=\exp{r}$ where $r=(r_1,r_2)\in\mathfrak{f}_{1}'\oplus\gor_2$ satisfies $\|r_2^{\on{lw}}\|_p > p^{-10m}.$
\item we further have
$D^{\beta \ref{exp:involdinp}}p^{-\beta\ref{exp:involdinD}}\geq p^{10m/6\dim{\mathfrak{f}}}$ (which will allow us to combine
the first two bullets).
\end{itemize}
We closely follow Proposition \ref{genericpoints} to deduce the existence of closeby generic points. 

We first note that the third bullet is always satisfied (for $D$ sufficiently large) as long as $\delta$ is chosen small enough (which will force $m$ to be small) as the other parameters are fixed throughout, and using Proposition~\ref{discvoljoint} to relate volume and discriminant, $D^{\ref{exp:discvol1}} \ll V \ll D^{\ref{exp:integersareoneapart}}$.

The reader is invited to go over Proposition \ref{genericpoints} once more to recall that we defined $E$, the set of $(\ell_0,\mathcal{F})$-generic points where $\ell_0$ is chosen such that $\mu_{v,S}(E)>0.99$ and a set $E'$ for which also \emph{most} translates along $H_p\cap K[1]$ are in $E$. Here we use
the subgroup~$F[1]=\exp(\mathfrak{f}[1])$
to define the set $E$ consisting of all~$(x,f)\in X_{\on{cpt}}\times F[1]$ such that~$xf$
is $(\ell_0,\beta(\on{log}_{p}(D)\ref{exp:involdinp}-\ref{exp:involdinD}),\mathcal{F
)}$-generic w.r.t.~the Sobolev norm~$\cS_{d_2}$. 
Next we 
define the set 
\[
	 E'=\Bigl\{x\in X_{\on{cpt}}: m_{F[1]}\bigl(\bigl\{f\in F[1]
	   :(x,f)\in E\bigr\}\bigr)
	 >\tfrac34m_{F[1]}(F[1])\Bigr\}.
\]
We note that the Fubini argument 
concerning the set $E'$
and the adjustment claim in Lemma~\ref{adjustmentlemma}
work equally well
after replacing $K[1]\cap H_p$ with the subgroup $F[1]$. 
Therefore, $\mu_{v,S} (E')>\frac34 $ 
and by choice of $m$ (assuming $\delta<\frac1{\dim\bG_{\textrm{joint}}}$) we are able to apply the pigeonhole principle Lemma~\ref{transversalpoints} as before. 
More specifically, recall from the proof of Lemma~\ref{transversalpoints}, that we cover $X_{\on{cpt}}$ 
with $I\leq m_{G^+_{{\on{joint}},S}}(\mathcal{N})^{-1}$ many translates $P_i=z_i\mathcal{N}_2$ 
for $i=1\dots I$ to find $y_1$ and $y_2$ in a common set $P_i\cap E'$
satisfying $y_2\in y_1(\Omega[m])_4$.  

By Lemma~\ref{adjustmentlemma} for $y_1,y_2\in E'$ with $y_2\in y_1(\Omega[m])_4$ there exists $\alpha_1,\alpha_2\in F[1]$ such that the translates $x_1=y_1\alpha_1$ and $x_2=y_2\alpha_2$ have their displacement 
of the form
\begin{equation}\label{displacement}
	 (g_\infty,(\exp{r_1},\exp{r_2}))\in 
(\Omega_\infty)_4\times\exp(\mathfrak{f}_{1}'[m]\oplus\gor_2[m]).
\end{equation}
If we find a pair of points such that the displacement satisfies 
$$\|r_2\|_p> p^{-10m},$$
then we argue just as in the proof of Proposition \ref{genericpoints} using the set $\mathcal{F}\subset L_{v,p}$
to find new points for which the new replacement satisfies 
$$\|r_2^{\on{lw}} \|_p> p^{-10m},
$$ 
which gives the claim from the beginning of the proof.

In the following we will assume indirectly that the displacement in \eqref{displacement}
of the points $x_1,x_2$ never satisfies the desired inequality, i.e.
 that we have $\|r_2\|_p\leq p^{-10m}$ or equivalently
\[
 x_2\in x_1\bigl((\Omega_\infty)_4\times\exp(\mathfrak{f}_{1}'[m]\oplus\gor_{2}[10m])\bigr),
 \]
which we will use to derive a  contradiction. 
We now describe what this means 
 for the original points $y_1$ and $y_2$. 
For this recall first
that $\Ad[\alpha_2]$ is an isometry on $\gog_1\oplus\gog_2$
and that
\[
 \alpha_2,\alpha_1\alpha_2^{-1}\in F[1]\subset H_{p}[1]\times g_vH_{p}g_v^{-1}[1].
\]
Together with the indirect assumption this gives
\begin{eqnarray}
y_2&=&y_1\alpha_1(g_\infty,(\exp{r_1},\exp{r_2}))\alpha_2^{-1}\nonumber\\
&=&y_1(g_\infty,\alpha_1\alpha_2^{-1}( \exp{\Ad[\alpha_2]r_1},\exp{\Ad[\alpha_2]r_2})) \in y_1 \widetilde{\Omega}[m],\label{Omega-claim}
\end{eqnarray}
where we use the shorthand
\[
\widetilde{\Omega}[m]= (\Omega_\infty)_4\times\exp(\gog_1[1]\oplus\goh_2[1]\oplus\gor_{2}[10m]).
\]
Also note 
that $\widetilde{K}[m]=\exp(\gog_1[1]\oplus\goh_2[1]\oplus\gor_{2}[10m])$ 
is actually a subgroup of~$G_{\textrm{joint},p}^+$.  
Hence our indirect assumption gives that for 
every $P_i$ either $P_i\cap E'$ is empty or there exists some $y_{1,i}$
with 
\[
 P_i\cap E'\subset Q_i=y_{1,i}\widetilde{\Omega}[m].
\]

For each such $Q_i$ we choose a smooth ``upper bound'' $f_i\in C^\infty_c(\mathcal{Y}_2)$ of the characteristic
function $\mathbbm{1}_{\pi_2(Q_i)}$ in the following way. In fact 
let us shrink $\Omega$ slightly so that we may assume that $g\in\Omega_8\mapsto xg$ 
is injective for all $x\in X_{\on{cpt}}$, which allows us to construct one function on $G_2$
that will be used to define $f_i$ for all $i$. 
 We fix some $f_\infty\in C^\infty_c(G_{2,\infty})$
with $\mathbbm{1}_{(\Omega_{2,\infty})_4}\leq f_\infty\leq \mathbbm{1}_{(\Omega_{2,\infty})_8}$ 
such that the derivatives $\mathcal{D}f_\infty$ are bounded by $\ll p^{\star}$ for all 
monomials $\mathcal{D}$ of order $\leq d_2$ as in the definition of the Sobolev norm. 
We also set $f_p=\mathbbm{1}_{\widetilde{K}[m]}$
so that $f((g_\infty,g_p))=f_\infty(g_\infty)f_p(g_p)$ satisfies $f\geq \mathbbm{1}_{\widetilde{\Omega}[m]}$.
If we now set $f_i(\pi_2(y_{1,i})g)=f(g)$ for all $g\in G_2\cap \Omega_8$  
and define $f_i$ to be zero outside of $\pi_2(y_{1,i}) (G_2\cap \Omega_8)$ we have
\[
	\cS_{d_2}^{\bG_2}(f_i)\ll p^{\star}p^{10d_2m}.
\]
By construction we also have 
\begin{align*}
	m_{\mathcal{Y}_2}(\pi_2(Q_i))\leq \int f_i \on{d}\!m_{\mathcal{Y}_2} &\ll
	m_{G_{2,S}}(\widetilde{\Omega}[m]\cap G_2) \ll p^{-10m\dim(\gor_2)}.
\end{align*}
By Theorem~\ref{maindynamicalresultsingle} on the second factor (proven in Section \ref{single_proof}--\ref{dynamics}) 
we obtain from this
\begin{align*}
	\mu_{v,S}(P_i\cap E')&\leq 
	{\pi_2}_*\mu_{v,S}(\pi_2(Q_i))\\
	& \ll m_{\mathcal{Y}_2}(\pi_2(Q_i))+\cS_{d_2}^{\bG_2}(f_i)D^{-\ref{exp:maindynamicalresultsingle}}  \\
	&\ll p^\star V^{-10\on{dim}(\gor_2)\delta}+
	p^\star V^{10d_2\delta}D^{-\ref{exp:maindynamicalresultsingle}},
\end{align*}
where we also used \eqref{delta1eq}.
By Proposition~\ref{discvoljoint} we also have $D^{\ref{exp:discvol1}} \ll V \ll D^{\ref{exp:integersareoneapart}}$
so that $D^{-\ref{exp:maindynamicalresultsingle}}\ll V^{-\ref{exp:maindynamicalresultsingle}/\ref{exp:integersareoneapart}}$.
Choosing $\delta$ very small  makes the sets $Q_i$ ``almost macroscopic'' 
in the sense that the error term in the above estimate becomes less than the first term. For that reason we may and will drop the error term in the further discussion.

This implies, with $I\ll V^{\on{dim}(\bG_{\textrm{joint}})\delta}$, that
\begin{align*}
3/4<\mu_{v,S}(E')&= \sum_{i\leq I} \mu_{v,S}\left( P_i\cap E' \right)\\
& \ll I p^\star V^{-10\on{dim}(\gor_2)\delta}
 \ll p^\star V^{\delta\left(\on{dim}(\bG_{\textrm{joint}})-10\on{dim}(\gor_2)\right)}.
\end{align*}
We note that the exponent of $V$ is now equal to $-14\delta$ for $d=4$ resp.\ $-65\delta$ for $d=5$.
As the implicit constants are absolute and $p^\star=O_\varepsilon(D^\varepsilon)$ for all $\varepsilon>0$
we may now choose $\delta$ small enough to fulfil the above requirements,
choose $\varepsilon$ even smaller,
and obtain a contradiction for sufficiently large $D$.

Having found the $[\ell_0,\ell_1]$-generic points as claimed in the beginning of the proof, we verify that
$\ell_1=\beta(\on{log}_{p}(D)\ref{exp:involdinp}-\ref{exp:involdinD})\geq 10 m/6\dim{\mathfrak{f}}\geq \log_p(\|r\|_p^{-1/6\dim{\mathfrak{f}}})$ so that
we can use Proposition~\ref{addinv} to deduce $p^{\ref{exp:addinvinpJointFinal}}D^{-\ref{exp:addinvJointFinal}}$-almost invariance under a vector $w \in (\mathfrak{f}')^{\on{hw}}$ with $\left\|w\right\|_{p} = 1$.

In the case~$\mathfrak{r}_2<\mathfrak{f}$ we have~$\gog_2<\mathfrak{f}$ as~$\mathfrak{r}_2$ generates~$\gog_2$.
As also~$\gor_1$ generates~$\gog_1$ and we have $\mathfrak{f}\neq\gog_1\oplus\gog_2$, the invariant
complement~$\gor_1$ does not belong to~$\mathfrak{f}$. As~$\gor_1$ is irreducible and its isomorphism
type appears only once in $\gog_1\oplus\gog_2$ we see that~$\gor_1<\mathfrak{f}'$. Switching the roles
of the first and the second factor makes no difference in the above argument (except for the precise
exponents in the final estimates, which now are~$-16\delta$ for $d=4$ and~$-15\delta$ for~$d=5$). 
\end{proof}

In order to apply Proposition \ref{genericpointsjoint} iteratively, we need to show that a vector under which $\mu_{v,S}$ is almost invariant generates a Lie algebra in an effective way that still leaves the 
measure almost invariant. In particular, we need the following notion of almost invariance 
under an element of a Lie algebra.

\begin{definition}
\label{almost_invariance_lie_algebra}
Let $w \in \mathfrak{g}[0]$ and\footnote{The case~$\ell=0$ is only allowed if~$w\in\mathfrak{g}[1]$ or if
$w$ is nilpotent, 
as in these case~$\exp(tw)$ exists for all~$t\in\bZ_p$).} $\ell \geq 0$. The measure 
$\mu_{v,S}$ is called $\varepsilon$-almost invariant of level $\ell$ under $w$ if $\mu_{v,S}$ is $\varepsilon$-almost invariant under $\exp(tw)$ for all $t \in p^{\ell}\mathbb{Z}_{p}$. 
Moreover, $\mu_{v,S}$ is called $\varepsilon$-almost invariant under $w$ if it is $\varepsilon$-almost invariant of level $1$ under $w$.
\end{definition}

We start by collecting some useful facts concerning almost invariant Lie algebras.

\begin{lemma}
\label{iteration}
Let $w \in \mathfrak{g}[0]$ and assume that $\exp(tw)$ exists for all $t\in \mathbb{Z}_{p}$. 
If $\mu_{v,S}$ is $\varepsilon$-almost invariant under $\exp(w)$ w.r.t.\ some Sobolev 
norm $\mathcal{S}_{d'}$, then $\mu_{v,S}$ is $\sqrt{\varepsilon}$-almost invariant 
under $\exp(nw)$ w.r.t.\ $\mathcal{S}_{d'}$ for all integers 
$n \leq \varepsilon^{-\frac12}$. In particular, $\mu_{v,S}$ is 
$\ll p\sqrt{\varepsilon}$-almost invariant of level $0$ under $w$ in the sense of
 Definition~\ref{almost_invariance_lie_algebra}.
\end{lemma}

\begin{proof}
Using the almost invariance under $\exp(w)$, we get that for any $f \in C_{c}(\mathcal{Y})$,
\begin{align*}
&\left|\mu_{v,S}(\exp(nw)f) - \mu_{v,S}(\exp((n-1)w)f)\right|\\
&\quad= \left|\mu_{v,S}(\exp(w)\exp((n-1)w)f) - \mu_{v,S}(\exp((n-1)w)f)\right|\\
&\quad\leq \varepsilon \mathcal{S}_{d'}(\exp((n-1)w)f).
\end{align*}
Moreover, by property (S3) of the Sobolev norm in Proposition~\ref{sobolev}
we have that the latter Sobolev norm equals~$\mathcal{S}_{d'}(f)$. 
Therefore, we may use the triangle inequality $n$-times and obtain
\[
\left|\mu_{v,S}(\exp(nw)f) - \mu_{v,S}(f)\right| \leq n\varepsilon \mathcal{S}_{d'}(f) 
\]
as desired. Since~$\{0,1,\ldots,\lfloor\varepsilon^{-\frac12}\rfloor\}$
is~$p\varepsilon^{\frac12}$ dense in~$\bZ_p$ the proposition follows from 
the last two properties in Lemma~\ref{invmeasconj}. 
\end{proof}

\begin{lemma}[Removing small portions]
\label{almost_invariance_under_sums} 
Let $w_{1}, w_{2} \in \mathfrak{g}[0]$ and assume that $\mu_{v,S}$ is $\varepsilon$-almost invariant 
(of level 1) under $w_{1} + p^{\ell}w_{2}$, where $p^{\ell} \geq \varepsilon^{-\kappa}$ for some $\kappa \in (0,1]$. Then, $\mu_{v,S}$ is $\ll\varepsilon^{\kappa}$-almost invariant under $w_{1}$.
\end{lemma}

\begin{proof}
  By Lemma \ref{invmeasconj} the measure~$\mu$ is also $\ll\varepsilon^\kappa$-almost invariant under any
  element of~$\mathfrak{g}[\ell]$.
  Now note that the Campbell-Baker-Hausdorff formula shows for every~$t\in p\bZ_p$ that
  \[
   \exp(tw_1) \exp(-(tw_{1} + p^{\ell}w_{2}))=\exp(w')
  \] 
  for some~$w'\in\mathfrak{g}[\ell]$, and apply Lemma \ref{invmeasconj} to obtain
  that~$\mu$ is~$\ll\varepsilon^\kappa$-almost invariant under~$\exp(w')\exp(t(w_1+p^\ell w_2))=\exp(tw_1)$.
\end{proof}

\begin{lemma}[Weight Splitting]
\label{weight_splitting}
Suppse that $\mu_{v,S}$ is $\varepsilon$-almost invariant under $w \in \mathfrak{g}$
and let $w = \sum_{j=1}^Jw_{j}$ so that $w_{1},\ldots,w_J$ are weight vectors of different weights with respect to
some~$\lsl$ contained in the Lie algebra~$\mathfrak{h}$ of the acting group. Assume furthermore
that~$[w_{j_1},w_{j_2}]=0$ for all~$j_1,j_2=1,\ldots, J$. Then 
there exists an absolute constant~$q$ such that $\mu_{v,S}$
is $\ll p^\star\varepsilon$-almost invariant of level $q$ 
under $w_j$ for~$j=1,\ldots,J$.
If the weight of~$w_j$ is zero, we also have almost invariance of level $0$.
Moreover, there exists some~$\kappa>0$ such that
if the weight of~$w_j$ is nonzero and~$\|w_j\|\leq p^{-r}$
for some~$r\geq 0$, then 
$\mu_{v,S}$
is $\ll p^{\kappa r}p^\star\varepsilon$-almost invariant of level $0$ 
under $p^{-r}w_j$. 
\end{lemma}

\begin{proof}
  	Using conjugation by the element~$a$ of the acting subgroup
	corresponding to the diagonal matrix with eigenvalues~$p,p^{-1}$
	and the first part of Lemma~\ref{invmeasconj} we see that~$\mu_{v,S}$ is 
	$\ll p^\star\varepsilon$-almost invariant under~$\exp(t\sum_jp^{km_j}w_j)$ for all~$t\in p\bZ_p$
	and~$k=0,\ldots,J-1$. 
	Here~$m_j$ are the 
	weights of~$w_j$. Let
	\[
	L=-\min\{0,m_j:j=1,\ldots,J\}\geq 0.
	\] 
	Using the second part of Lemma~\ref{invmeasconj} we now see that~$\mu_{v,S}$
	is~$\ll p^\star\varepsilon$-almost invariant under
	\begin{multline*}
	\exp\bigl(t_0\sum_j w_j\bigr)
	\exp\bigl(t_1p^L\sum_j p^{m_j}w_j\bigr)\cdots
	\exp\bigl(t_{J-1}p^{(J-1)L}\sum_jp^{(J-1)m_j}w_j\bigr)=\\
	  \exp\biggl(\sum_{j=1}^J\sum_{k=0}^{J-1}(t_k p^{(L+m_j)k})w_j\biggr)
	\end{multline*}
	for all~$t_1,\ldots,t_J\in p\bZ_p$. 
	By assumption the weights~$m_j$ are different which makes the Vandermonde matrix 
	implicitly appearing
	in the above exponential invertible. 
	Hence the sum~$\sum_{j=1}^J\sum_{k=0}^{J-1}(t_k p^{(L+m_j)k})w_j$ can be made to agree
	with any linear combination of the vectors~$w_1,\ldots,w_J$
	if we were to use coefficients~$t_1,\ldots,t_J\in\bQ_p$. 
	In particular this applies to any multiple of~$w_{j_0}$
		for some fixed index~$j_0\in\{1,\ldots,J\}$.
	The restriction to coefficients~$t_1,\ldots,t_J\in p\bZ_p$
	amounts to the restriction that all vectors in~$p^q \bZ_p w_{j_0}$
	can be obtained, where~$q$ only depends on the Vandermonde determinant.
	
	Suppose now~$w_1$ has nonzero weight, then we can apply
	the first part of Lemma~\ref{invmeasconj} with~$g_0=\exp(tw_1)$ for~$t\in p^{q}\bZ_p$ and $h=a^r$
	for some~$r\in\bZ$. This gives us that
    $\mu_{v,S}$ is $\ll p^{\star}\|a\|^{4d'r}\varepsilon$-almost invariant under~$\exp(tp^{rm}w_1)$.
	I.e.\ we can use this to lower the level to $0$ (if $rm<0$) and divide~$w_1$ by~$p^r$
	at the cost of increasing the error term
	in the almost invariance by a fixed power of~$p$ resp.~of~$p^r$.
\end{proof}

\subsection{The case $d=4$}
We are going to prove that the almost invariance of $\mu_{v,S}$ under a highest weight vector $w$ implies that $\mu_{v,S}$ is also almost invariant under a Lie algebra containing $w$. 
This then allows us to iteratively apply Proposition~\ref{genericpointsjoint} until $\mu_{v,S}$ is finally almost invariant under all of $\mathfrak{g}_{1} \oplus \mathfrak{g}_{2}$. 
In what follows, $\varepsilon^{\star}$ always denotes a positive power of $\varepsilon$, 
where the exponent only depends on the dimension and on the Sobolev norm. We will keep 
writing $\varepsilon^{\star}$ even though the exponent will change in the course of the 
proof. As discussed in Section \ref{liealgebras}, we 
have $\mathfrak{g}_{1} \cong \mathfrak{sl}_{2} \times \mathfrak{sl}_{2}$ over~$\bQ_p$.
By Proposition \ref{genericpointsjoint}, there exists a 
vector $w = (w_{1},w_{2},0,w_{4}) \in V \subseteq \mathfrak{sl}_{2} \times \mathfrak{sl}_{2} \times \{0\} \times \mathfrak{r}_{2}$ 
of highest weight and with $\left\|w\right\| = 1$, under which 
$\mu_{v,S}$ is $\ll\varepsilon$-almost invariant, where $\varepsilon = p^{\consta\label{exp:effgen1}}D^{-\consta\label{exp:effgen2}}$
 for some positive constants $\ref{exp:effgen1},\ref{exp:effgen2}>0$. If $w_{4}$ 
 satisfies $\|w_4\|\geq\varepsilon^{1/(2\kappa)}$, we may apply Lemma~\ref{weight_splitting} to get 
 $\ll p^\star \varepsilon^{1/2}$-almost invariance 
under the element~$p^{-r}w_4$ of norm one. Using both parts of
 Lemma~\ref{boundedgeneration} and Lemma \ref{hensel} it follows that~$\mu_{v,S}$
is $\ll p^\star\varepsilon^\star$-almost invariant under $\mathfrak{g}_{2}$. 
We then choose $\mathfrak{f}_{\on{new}} = \mathfrak{f} \oplus \mathfrak{g}_{2}$ and if necessary 
apply Proposition \ref{genericpointsjoint} again with $\mathfrak{f}_{\on{new}}$ and~$\mathfrak{f}_{\on{new}}'=\mathfrak{f}'\cap(\mathfrak{r}_1\oplus\{0\})$.

We now assume that $w_{4}$ is of size less than $\ll\varepsilon^{1/(2\kappa)}$. Using Lemma \ref{almost_invariance_under_sums}, this implies that $\mu_{v,S}$ is also $\ll\varepsilon^{\star}$-almost invariant under the vector $w_{1}X_{1} + w_{2}X_{2} \in \mathfrak{h}_{1} \times \mathfrak{r}_{1}$. Also recall that the principal $\mathfrak{sl}_{2}$ was defined in Section \ref{liealgebras} such that $X_{1} + X_{2} \in \mathfrak{h}_{1}$. Consider the element (which will bring~$[Y,w]$ into the argument)
\[
w' = \log\left(\exp\left(p^{\ell}Y\right)\exp\left(p^{\ell}w\right)\exp\left(-p^{\ell}Y\right)\exp\left(-p^{\ell}w\right)\right),
\]
where the integer $\ell \geq 1$ will be chosen later. Since
\[
\left\|\exp(p^{\ell}Y)\right\|_{p} = \left\|\exp(p^{\ell}w)\right\|_{p} = 1,
\]
 we get by Lemma \ref{invmeasconj} that $\mu_{v,S}$ is $\ll  \varepsilon^{\star}$-almost invariant under $w'$. Notice however that by the Campbell-Baker-Hausdorff formula (applied twice), we also have
\begin{align*}
w' &\in \log\left(\exp\left(p^{\ell}Y + p^{\ell}w + \tfrac{1}{2}p^{2\ell}\left[Y,w\right] + \mathfrak{g}[3\ell]\right)\right. \\
& \hspace{1.3cm} \left. \exp\left(-p^{\ell}Y - p^{\ell}w + \tfrac{1}{2}p^{2\ell}\left[Y,w\right] + \mathfrak{g}[3\ell]\right)\right) \\
&\subseteq p^{2\ell}\left[Y,w\right] + \mathfrak{g}[3\ell].
\end{align*}
Therefore, there exists an element $v' \in \mathfrak{g}[3\ell]$ so that $\mu_{v,S}$ is $\ll \varepsilon^{\star}$-almost invariant under $\exp(w')$ with $w'=p^{2\ell}\left[Y,w\right] + v'$.

Now, consider the element
\[
w'' = \log\left(\exp\left(p^{2\ell}w\right)\exp\left(w'\right)\exp\left(-p^{2\ell}w\right)\exp\left(-w'\right)\right).
\]
Applying Lemma \ref{invmeasconj} as before, $\mu_{v,S}$ is $\ll \varepsilon^{\star}$-almost invariant under $\exp(w'')$. Notice however that by the Campbell-Baker-Hausdorff formula,
\[
w'' = p^{4\ell}\left[w,\left[Y,w\right]\right] + v''
\]
for some $v'' \in \mathfrak{g}[5\ell]$. Consider the element $a$ in the principal~$\SL_2$
corresponding to the diagonal matrix with eigenvalues~$p,p^{-1}$, which has norm $\left\|a\right\|_{p} = p^4$ (since~$4$ is the largest weight appearing in~$\mathfrak{g}=\mathfrak{g}_1\times\mathfrak{g}_2$). 
Lemma \ref{invmeasconj} implies that $\mu_{v,S}$ is $\ll p^{16d'\ell}\varepsilon^{\star}$-almost 
invariant under the conjugated element
\[
a^{-2\ell}\exp(w'')a^{2\ell} = \exp(\left[w\left[Y,w\right]\right] + v''') = \exp(-2(w_{1}^{2}X_{1} + w_{2}^{2}X_{2}) + v''')
\]
for some $v''' \in \mathfrak{g}[\ell]$. If we now choose $\ell \geq 1$  maximal so that 
\[
p^{16d'\ell}\varepsilon^{\star} \leq p^{-\ell},
\]
Lemma \ref{almost_invariance_under_sums} and Lemma~\ref{iteration} imply
 that $\mu_{v,S}$ is $\ll p^\star\varepsilon^{\star}$-almost invariant under $\left[w,\left[X,w\right]\right] = -2(w_{1}^{2}X_{1} + w_{2}^{2}X_{2})$. In particular, there exists an absolute constant $\consta\label{exp:effgen3}> 0$, so that $\mu_{v,S}$ is $p^\star\varepsilon^{\ref{exp:effgen3}}$-almost invariant under $w_{1}X_{1} + w_{2}X_{2}$ \emph{and} under $w_{1}^{2}X_{1} + w_{2}^{2}X_{2}$.

We now consider the matrix $A= \left(\begin{smallmatrix}w_{1} & w_{1}^{2}\\w_{2} & w_{2}^{2}\end{smallmatrix}\right)$ and its determinant 
\[
T = w_{1}w_{2}(w_{2}-w_{1}).
\] 
We distinguish the following four cases, using $\kappa>0$ as in Lemma~\ref{weight_splitting}:

\textit{Case 1: $|w_{1}| \leq \varepsilon^{\ref{exp:effgen3}/6\kappa}$.} This means that $|w_{2}| = 1$ and Lemma \ref{almost_invariance_under_sums} applied to the vector $w_{1}X_{1} + w_{2}X_{2}$ implies that $\mu_{v,S}$ is $p^\star\varepsilon^{\ref{exp:effgen3}/6\kappa}$-almost invariant under $X_{2}$. 
We now apply conjugation by~$\exp(Y)$, where~$Y$ is the opposite nilpotent element in the principal~$\lsl$.
This shows that~$\mu_{v,S}$ is also~$\ll p^\star \varepsilon^{\ref{exp:effgen3}/6\kappa}$-almost invariant
under~$\Ad[\exp(Y)] X_2$ and~$\Ad[\exp(Y)]^2 X_2$. However, these three give a basis of~$\{0\}\times\mathfrak{sl}_2$
and we may apply Lemma \ref{hensel} to see that~$\mu_{v,S}$ is almost invariant under~$\{0\}\times\mathfrak{sl}_2$.
We define $\mathfrak{f}_{\on{new}}=\mathfrak{f}\oplus(0 \times \mathfrak{sl}_{2})$ and~$\mathfrak{f}'_{\on{new}}=\mathfrak{f}'\cap(\mathfrak{sl}_2\times\{0\}\oplus\mathfrak{r}_2)$ 
and go back to Proposition \ref{genericpointsjoint} if necessary.

\textit{Case 2: $|w_{2}| \leq \varepsilon^{\ref{exp:effgen3}/6\kappa}$.} As in the previous case, we see that $\mu_{v,S}$ is $\ll p^\star\varepsilon^{\star}$-almost invariant under the Lie algebra $\mathfrak{sl}_{2} \times 0$ and we set $\mathfrak{f}_{\on{new}} = \mathfrak{f} \oplus (\mathfrak{sl}_{2} \times 0)$ in Proposition~\ref{genericpointsjoint}.

\textit{Case 3: $|w_{2} - w_{1}| \leq \varepsilon^{\ref{exp:effgen3}/6\kappa}$.} This means that there exists
$w'\in \mathfrak{g}$ with~$\|w\|\leq\varepsilon^{\ref{exp:effgen3}/6\kappa}$  
so that $\mu_{v,S}$ is $\varepsilon^{\ref{exp:effgen3}/6\kappa}$-almost invariant under $w_{1}(X_{1} + X_{2}) + w'$. Moreover, $\|w_{1}\| = 1$ since~$\|w\|=1$. Lemma \ref{almost_invariance_under_sums} then implies that $\mu_{v,S}$ is $p^{\star}\varepsilon^{\ref{exp:effgen3}/6\kappa}$-almost invariant under $(X_{1} + X_{2})$ and thus, arguing as in the first case, it is also $\varepsilon^{\star}$-almost invariant under the Lie algebra $\mathfrak{h}_{1}$. We may therefore set $\mathfrak{f}_{\on{new}} = \mathfrak{f} \oplus \mathfrak{h}_{1}$,
$\mathfrak{f}_{\on{new}}' = \mathfrak{f}' \cap ((\mathfrak{sl}_2\times\{0\})\oplus\mathfrak{r}_2)$,
and apply Proposition \ref{genericpointsjoint} again with $\mathfrak{f}_{\on{new}}$ if necessary.

\textit{Case 4: $|T| > \varepsilon^{\ref{exp:effgen3}/2\kappa}$.} 
Since~$X_1$ and~$X_2$ commute, we see that~$\mu_{v,S}$ is $p^\star\varepsilon^{\ref{exp:effgen3}}$-almost invariant under
\begin{multline*}
\exp\bigl(t_1(w_{1}X_{1} + w_{2}X_{2}) +t_2(w_{1}^{2}X_{1} + w_{2}^{2}X_{2})\bigr)\\
=\exp\bigl((t_1w_1+t_2w_1^2)X_1+(t_1w_2+t_2w_2^2)X_2)
\end{multline*}
for all~$t_1,t_2\in p\bZ_p$. As~$T$ is the determinant of~$A$ we see that using~$t_1,t_2\in p\bZ_p$
we obtain almost invariance under all elements in~$pT\bZ_p X_1+pT\bZ_p X_2$ -- which amounts
to a level restriction depending on~$T$.
Using the last claim in Lemma~\ref{weight_splitting} 
 we can lift that restriction at the cost
of increasing the error term. We have set up the cases in a way so that this now gives that~$\mu_{v,S}$
is~$\ll p^\star\varepsilon^{\ref{exp:effgen3}/2}$-almost invariant of level 0
under all vectors in~$\bZ_p X_1+\bZ_p X_2$.
Arguing as in the previous cases, we see that $\mu_{v,S}$ is $\varepsilon^{\star}$-almost invariant under all of $\mathfrak{sl}_{2} \times \mathfrak{sl}_{2} = \mathfrak{g}_{1}$ and we set $\mathfrak{f}_{\on{new}} = \mathfrak{f} \oplus \mathfrak{g}_{1}$.

As before, we go back to Proposition \ref{genericpointsjoint} if necessary, i.e.\ if~$\mathfrak{f}_{\on{new}}\neq\mathfrak{g}_1\times\mathfrak{g}_2$.

\subsection{The case $d=5$, split}
As in the case for $d=4$, we may assume that $\mu_{v,S}$ is $\varepsilon$-almost invariant under a highest weight vector $w = (w_{1},w_{2},0,0)\in V=\mathfrak{h}_{1}\times \mathfrak{r}_{1} \times 0 \times \mathfrak{r}_{2}$ with $\left\|w\right\|_{p} = 1$. Indeed, if $w$ had a significant component in the highest weight direction of $w_{4}$, we can apply Lemma \ref{weight_splitting} to get $\varepsilon^{\star}$-almost invariance under $\mathfrak{g}_{2}$ as argued before.

Recall from Section \ref{liealgebras} that $\mathfrak{h}_{1} = V^{(2)} \oplus V^{(2)}$ and 
write $w_{1} = (w_{1,1}, w_{1,2})$ for the corresponding decomposition. In this notation~$w=((w_{1,1}, w_{1,2}), w_2,0,0)$.
As we are in the split case, we may also consider one of the direct factor~$\lsl$ of the Lie algebra of the acting group.
Using this~$\lsl$ instead of the principal~$\lsl$ we see that the three remaining components of~$w$ all 
have different weights. Hence we can apply Lemma \ref{weight_splitting} again to see that~$\mu_{v,S}$ is
almost invariant under all~$w_{1,1},w_{2,2},w_2$ seperately and of level $0$.
One of the three vectors has norm $1$. If~$\|w_{1,1}\|=1$ or~$\|w_{1,2}\|=1$, then we obtain that~$\mu_{v,S}$
is almost invariant under one of the direct factors of~$\mathfrak{h}_1$ and define~$\mathfrak{f}_{\on{new}}$
accordingly.

So suppose~$w_2$ has norm~$1$. In the notation of Section~\ref{liealgebras} this means that~$w_2$
is a~$\bZ_p^\times$-multiple of~$(1,0,0,0)\in\mathfrak{r}_1$. Using the nilpotent element~$Y_1$
of the acting group and the first relation in \eqref{gettingaround}
we see as before that~$\mu_{v,S}$ is also almost invariant under
\[
 \on{Ad}_{\exp Y_1}(1,0,0,0)=\exp(\ad[Y_1])(1,0,0,0)=(1,0,0,0)+(0,1,0,0).
\]
However, these are of different weights for the $\lsl$-factor of the acting group
that corresponds to~$Y_1$, which shows by Lemma~\ref{weight_splitting}
that we also have almost invariance of level $0$ under $(0,1,0,0)$. Using the third relation in \eqref{gettingaround}
in the same way, we obtain almost invariance under~$(0,0,1,0)$ and $(0,0,0,1)$. 
We can now use the first part of \eqref{gettingX} to see almost invariance
under
\[
\on{Ad}_{\exp(1,0,0,0)}(0,0,1,0)=\exp(\ad[(1,0,0,0)])(0,0,1,0)=(0,0,1,0)+\lambda X_1
\]
Once more the two vectors on the right have different weights for the~$\lsl$-factor of the acting group
corresponding to~$X_1$
and we obtain almost invariance under~$X_1$. Using the second part of \eqref{gettingX} we also
obtain almost invariance under~$X_2$. Applying~$\on{Ad}_Y$ to these two twice we obtain a basis
of~$\gog_1$ and can apply Lemma~\ref{hensel} to obtain that~$\mu_{v,S}$ is~$\ll p^\star\epsilon^\star$-almost invariant
under~$\gog_1$. We define~$\mathfrak{f}_{\on{new}}=\mathfrak{f}+\gog_1$ and apply Proposition~\ref{genericpointsjoint}
if necessary.

\subsection{The case $d=5$, quasi-split}
As in the previous cases, we may assume that $\mu_{v,S}$ is $\varepsilon$-almost invariant under a highest weight vector $w = (w_{1},w_{2},0,0)\in V=\mathfrak{h}_{1}\times \mathfrak{r}_{1} \times 0 \times \mathfrak{r}_{2}$ with $\left\|w\right\|_{p} = 1$.  

Applying conjugation by the element~$\exp(pM)$ of the acting group we have also almost invariance under
\[
 \on{Ad}_{\exp(pM)}(w_1,w_2,0,0)=\exp(\ad[pM])(w_1,0,0,0)+(0,w_2,0,0).
\]
Using that this vector together with the original vector~$(w_1,w_2,0,0)$ span an abelian Lie algebra,
we can take the difference and obtain almost invariance under the vector
\[
 p\ad[M](w_1,0,0,0)+p^2\frac12\ad[M]^2(w_1,0,0,0)+\cdots
\]
which belongs to the linear span of~$X$ and~$X_2$ and has norm equal to~$\|pw_1\|$. 
If~$\|w_1\|\geq \varepsilon^{\ref{exp:effgen3}/2\kappa}$ we use the last claim in Lemma \ref{weight_splitting}
to obtain~$\ll p^\star\varepsilon^{\ref{exp:effgen3}/2}$-almost invariance of level $0$ under some element of norm $1$
in the linear hull of $X$ and~$X_2$. We can use this, the element~$M$ as above, and also
the element~$Y$ to generate~$\goh_1$ effectively,
 define~$\mathfrak{f}_{\on{new}}=\mathfrak{f}+\goh_1$, and go back to Proposition~\ref{genericpointsjoint}
if necessary.                          

So suppose~$\|w_1\|< \varepsilon^{\ref{exp:effgen3}/2\kappa}$. Applying Lemma~\ref{almost_invariance_under_sums} 
we obtain almost invariance under the element~$(0,w_2,0,0)$ of norm one, or using the 
abbreviation used in Section~\ref{liealgebras} equivalently under~$(1,0,0,0)$.
As in the split case we can use~$Y$ and $Y_2$ to obtain almost invariance under~$\gor_1$ from this.
Moreover, using \eqref{gettingminusX} (instead of \eqref{gettingX}) as in the split case we obtain almost invariance
under~$X$.   From this we can again generate~$\goh_1$
and hence even~$\gog_1$. We define $\mathfrak{f}_{\on{new}}=\mathfrak{f}+\gog_1$, and go back to Proposition~\ref{genericpointsjoint} if necessary.

\subsection{Summary}
Iteratively applying this procedure finitely many times and arguing as in the proof of Proposition \ref{almostinvariance}, we have proved the following. 
\begin{proposition}
\label{producingH}
There exist $\consta\label{exp:addinvjoint1}>0$ and $\consta\label{exp:addinvjoint2}>0$ such that $\mu_{v,S}$ is $p^{\ref{exp:addinvjoint1}}D^{-\ref{exp:addinvjoint2}}$-almost invariant under $\mathbb{G}_{\on{joint}}(\mathbb{Q}_{p}) \cap K[1]$
\end{proposition}

\subsection{Proof of Theorem~\ref{maindynamicalresult}}
\label{proofmaintheorem}
In order to upgrade almost invariance of $\mu_{v,S}$ under $\mathbb{G}_{\on{joint}}(\mathbb{Q}_{p}) \cap K[1]$ to saying that $\mu_{v,S}$ is close to the Haar measure on $\mathcal{Y}^{+}_{\on{joint}}$, we use the same convolution step as in Section~\ref{convolution}. Note that the property (S6) of the Sobolev norm also holds in this setting, i.e.
\[
\left|\mathbb{T}_{t}(f)(x) - m_{\mathcal{Y}^{+}_{\on{joint}}}(f)\right| \ll \on{ht}(x)^{d_{2}}\left|t\right|^{\ref{exp:heckeA11}}p^{3d{2}L}\mathcal{S}_{d_{2}}(f),
\]
where $\bT_t=\on{Av}_L\star\;\delta_{u(t)}\star \on{Av}_L$ is the Hecke operator on $L^2_{m_{\mathcal{Y}^+_{\on{joint}}}}$ and $\on{Av}_{L}$ denotes convolution with the characteristic function on $\mathbb{G}^{+}(\mathbb{Q}_{p}) \cap K[L]$. We may now follow the proof in Section \ref{convolution} line by line to obtain an upper bound $\left| \int f\on{d}\!\mu - \int f\on{d}\!m_{\mathcal{Y}_{\on{joint}}^{+}}\right| \ll D^{-\consta\label{exp:final}}\mathcal{S}_{d_{2}}(f)$ for some $\ref{exp:final} > 0$ as required.\qed

\appendix

\section{Good Subgroups}\label{Good_Subgroups}

In order to apply Theorem \ref{propertytau}, we have to verify that for these groups and the ambient groups $\bG_1$ and $\bG_2$, their $\bZ_p$-points indeed define good maximal compact subgroups. 
\begin{lemma}
The groups $\on{SL}_2(\bZ_p)$, $\on{SL}_2(\bZ_p(\sqrt{\eta}))$ and $\on{SL}_2(\bZ_p)\times \on{SL}_2(\bZ_p)$ are good maximal subgroups in $\on{SL}_2(\bQ_p)$, $\on{SL}_2(\bQ_p(\sqrt{\eta}))$ and $\on{SL}_2(\bQ_p)\times \on{SL}_2(\bQ_p)$ respectively.
\end{lemma}
\begin{proof}
The following will introduce the notation of \cite[Chapter 2.1]{Oh} immediately applied to the simple case $\SL_2$.
We let $k$ denote either $\bQ_p$ or $\bQ_p(\sqrt{\eta})$, $\mathcal{O}$ the ring of integers $\bZ_p$ resp.\ $\bZ_p(\sqrt{\eta})$ and $|\cdot|_p$ the (extended) $p$-adic absolute value. Let $A$ be the diagonal group of $G=\SL_2(k)$ (a maximal $k$-split torus of $G$), $B$ the group of upper diagonal matrices (a minimal parabolic subgroup of $G$ containing $A$) and $K=\SL_2(\mathcal{O})$.
The upper nilpotent matrix $z=\smallmat{0&1\\0&0}$ gives rise to a character $\chi$ on $A$ by $\Ad[a]z=\chi(a)z$ where $\chi(\diag{\lambda,\lambda^{-1}})=\lambda^2$. We call $\chi$ a positive root, and together with $\chi^{-1}$ defines the simple root system $\Psi=\Psi(\on{SL}_2,A)$. Denote by $\Psi^+$ the one-point set containing $\chi$ (the set of \emph{positive} roots with respect to the choice of $B$).

Let $X(A)$ denote the set of characters on $A$
and $X(A)^+$ the set of positive characters with respect to the ordering above.
Let $k^0=\{p^n:n\in\bZ\}$ and $\widehat{k}=\{p^{-n}:n\in\bN\}$. Set 
\[
A^0=\{a\in A: \alpha(a)\in k^0\mbox{ for each } \alpha\in X(A)\}=\{\diag{p^n,p^{-n}}: n\in \bZ\}
\]
and the positive Weyl chamber (with respect to $B$)
\[
A^+=\{a\in A: \alpha(a)\in \widehat{k}\mbox{ for each } \alpha\in \Psi^+\}=\{\diag{p^{-n},p^{n}}: n\in \bN\}.
\]
The centralizer $Z=Z_G(A)$ of $A$ in $\SL_2(k)$ is $A$ and we define the subsets
\[
Z_0=\{a\in Z: |\alpha(a)|_p=1\mbox{ for each } \alpha\in X(A)^+\}=A\cap K
\]
and
\[
Z_+=\{a\in Z: |\alpha(a)|_p\geq1\mbox{ for each } \alpha\in X(A)^+\}=\{\diag{\lambda,\lambda^{-1}}: |\lambda|_p\geq1\}.
\]
Let $N_G(A)$ denote the normalizer of $A$. The group $K$ is a good maximal compact subgroup of $G$ if the following conditions hold.
\begin{enumerate}
\item $N_G(A)\subset KA$.
\item $G=K(Z_+/Z_0)K$ and$G=K(Z/Z_0)R_u(B)$ where $R_u(B)$ is the maximal unipotent radical of $B$ (the set of upper unipotent matrices).
\item For any subset $\Delta$ of $\Psi$, let $M=Z_G(\{a\in A: \alpha(a)=1\mbox{ for all }\alpha\in\Delta\})$ then $(M,K\cap M, A\cap M)$ satisfies (1) and (2) (just as $(G,K,A)$).
\end{enumerate}
Since $N_G(A)$ is generated by $\omega=\smallmat{0&-1\\1&0}\in K$ and $AZ_0$, we have (1). 
The second property asks for the Iwasawa and Cartan decomposition, which for $\SL_2$ is easily obtained. We first note that $(Z_+/Z_0)\cong A^+$ and $(Z/Z_0)\cong A^+(A^+)^{-1}$. Let $g=\smallmat{a&b\\ c&d}\in G$ be arbitrary. 
For the Cartan decomposition, we may permute rows and columns (with some sign changes) by multiplying $g$ with $\omega$ from the right and from the left as necessary to ensure $|a|_p=\|g\|_p$. Then we may multiply by $\smallmat{1&0\\ -c/a&1}\in K$ from the left to reduce $g$ to $\smallmat{a&b\\ 0&a^{-1}}$ and multiplying by $\smallmat{1& -b/a\\0&1}\in K$ from the right, we obtain $g\in K\diag{a,a^{-1}}K$. Since $\diag{a|a|_p,(a|a|_p)^{-1}}\in K$ we get $g\in KA^+K$. 
For the second decomposition, we may exchange rows by multiplying by $\omega$ from the left to have $-c/a\in \mathcal{O}$. Then multiplying again by the above lower unipotent we are reduced to $\smallmat{a&b\\ 0&a^{-1}}=\smallmat{a|a|_p&0\\0&(a|a|_p)^{-1}}\smallmat{|a|_p^{-1}&\\ 0&|a|_p}\smallmat{1&ba^{-1}\\ 0&1}$ and the Iwasawa decomposition follows.

For the final condition, we have the possibilities $\Delta=\{\chi\}, \Delta=\{\chi^{-1}\}$ and $\Delta=\{\}$. In the first two cases we have $M=G$ (which we already considered above). In the last case we obtain $M=Z_G(A)=Z$, where (1) and (2) both reduce to $Z=Z_0A$.

The first two properties for $\on{SL}_2(\bZ_p)\times \on{SL}_2(\bZ_p)$ 
from from the statement for a single factor. 
In the third part more subsets $\Delta$ are possible, but here again all
cases follow from the case $\on{SL}_2$ that we already considered. 
\end{proof}


\section{Regular Trees} \label{Buildings}
\subsection{Proof of Proposition \ref{diagonalQF}}\label{proofoffirstinB}
Denote by $e_{i}$ the $i$th standard basis vector in $\bQ_{p}^{n}$. We have
\[
2a_{ij} = Q\left(e_{i} + e_{j}\right) - Q\left(e_{i}\right) - Q\left(e_{j}\right)
\]
and since $p \neq 2$ also
\[
\max_{i,j}\left|a_{ij}\right|_{p} = \max_{i,j}\left|Q\left(e_{i} + e_{j}\right) - Q\left(e_{i}\right) - Q\left(e_{j}\right)\right|_{p}.
\]

Notice that by the strong triangle inequality for the $p$-adic norm, there exists a $v \in \mathbb{Z}_{p}^{n}$ of the form $e_{i}$ or $e_{i} + e_{j}$ for some $i,j \in \left\{1,2,\ldots, n\right\}$, such that $\max_{i,j}\left|a_{ij}\right|_{p} = |Q(v)|_{p}$. Since $v \in \bZ_{p}^{n}$ is primitive, we can extend it to a $\mathbb{Z}_{p}$-basis $v_{1} = v,v_{2},v_{3},\ldots,v_{n}$ of $\bZ_{p}^{n}$. Using this, we see that $Q$ is $\bZ_{p}$-equivalent to a quadratic form
\[
\tilde{Q}\left(x\right) = \sum_{i,j}b_{ij}x_{i}x_{j}
\]
with $\left|b_{11}\right|_{p} \geq \left|b_{1j}\right|_{p}$ for all $j \in \left\{1,2,\ldots n\right\}$.  If $b_{11}=0$ then $Q=0$ and the conclusion of the proposition already holds, so we assume that $b_{11}\neq 0$.

Hence we can write
\[
\tilde{Q}\left(x\right) = b_{11}\left(x_{1} + \frac{b_{12}}{b_{11}}x_{2} + \ldots + \frac{b_{1n}}{b_{11}}x_{n}\right)^{2} + F\left(x_{2},\ldots,x_{n}\right),
\]
where $F$ is a quadratic form in $n-1$ variables with coefficients in $\bQ_{p}$. Moreover, $\frac{b_{1j}}{b_{11}} \in \bZ_{p}$, since $\left|b_{11}\right|_{p} \geq \left|b_{1j}\right|_{p}$. Therefore, $Q$ is $\bZ_{p}$-equivalent to
\[
b_{11}x_{1}^{2} + F\left(x_{2},\ldots,x_{n}\right).
\]
Now the first claim follows by induction. For the second notice
$\operatorname{diag}(c_{1},\dots ,c_{n}) = gAg^{T}$
for some $g\in \on{GL}_n(\bZ_p)$ and the bi-invariance of the norm under $\on{GL}_{n}(\mathbb{Z}_{p})$.
\qed

\subsection{Cartan Decomposition of the Model Groups}\label{secB2-cartan}
\begin{proposition}
\label{pRankOneKAK}
Let $g \in H = \operatorname{SO}_{\eta}(3,1)(\bQ_{p})$  with $e_{1}g = (w_{1},w_{2},w_{3},w_{4})$. Then we have that
\[
\max(\left|w_{1}\right|_{p},\left|w_{2}\right|_{p}) \geq \max(\left|w_{3}\right|_{p},\left|w_{4}\right|_{p}).
\]
Moreover, if we set $K = \operatorname{SO}_{\eta}(3,1)(\bZ_{p})$ and
\[
A_{+} = \left\{\operatorname{diag}(p^{-m},p^{m},1,1) : m \in \bZ_{\geq 0}\right\},
\]
then every element $g \in H$ can be written as $g = k_{1}ak_{2}$ with $k_{1},k_{2} \in K$ and some uniquely determined $a \in A_+$. 
\end{proposition}

For the proof of Proposition \ref{pRankOneKAK} it will be useful at times to work over the finite field $\bF_{p}$. We denote the reduction map modulo $p$ from $\bZ_{p}$ to $\bF_{p}$ by $\eta\mapsto \overline{\eta}$ and extend its definition to integral matrices.

\begin{proof}[Proof of Proposition \ref{pRankOneKAK}]
Let $g \in H$ be arbitrary with $e_{1}g = (w_{1},w_{2},w_{3},w_{4})$ and assume that $\max{(|w_{1}|_{p},|w_{2}|_{p})}<\max{(|w_{3}|_{p},|w_{4}|_{p})}$. Then we can multiply $e_{1}g$ with $p^k$, such that afterwards $p^k(w_{1},w_{2},w_{3},w_{4})$ is integral and both $p^kw_{1}$ and $p^kw_{2}$ vanish after reduction modulo $p$. Since $e_{1}$ is isotropic, we obtain that $(\overline{p^kw_{3}},\overline{p^kw_{4}})$ defines a non-zero isotropic vector for the anisotropic quadratic form $z^{2}+\bar{\eta}w^{2}$ over $\bF_p$, which is a contradiction. For this recall that
for $\eta\in\bZ_p^\times$ the conditions $\eta\in\bZ_p^2$ and $\bar{\eta}\in\bF_p^2$ are equivalent
by Hensel's Lemma.

For the Cartan decomposition, we define the following unipotent elements
\begin{align*}
u_{3}\left(t\right) &= \begin{bmatrix}1&&&\\-t^{2}/2&1&t&\\-t&&1&\\&&&1\end{bmatrix} ~,~~v_{3}\left(t\right) = \begin{bmatrix}1&-t^{2}/2&t&\\&1&&\\&-t&1&\\&&&1\end{bmatrix}\\
u_{4}\left(t\right) &= \begin{bmatrix}1&&&\\-t^{2}\eta/2&1&&\eta t\\&&1&\\-t&&&1\end{bmatrix} ~,~~v_{4}\left(t\right) = \begin{bmatrix}1&-t^{2}\eta/2&&\eta t\\&1&&\\&&1&\\&-t&&1\end{bmatrix}.
\end{align*}
A direct calculation shows that these elements belong to $H$ for all $t\in\mathbb{Q}_{p}$ and that they belong to $K$ for all $t \in \mathbb{Z}_{p}$. Let $g \in H$ and $w_1,w_2,w_3,w_4$
be as above and let $e_{2}g = (x_{1},x_{2},\ast,\ast)$. Multiplying $g$ on the left and right with $\omega = \left[\begin{smallmatrix}&-1&&\\1&&&\\&&1&\\&&&1\end{smallmatrix}\right]\in K$ if necessary, we may assume that $\left|w_{1}\right|_{p} \geq \max(\left|w_{2}\right|_{p},\left|x_{1}\right|_{p},\left|x_{2}\right|_{p})$. 
Below we will keep multiplying~$g$ on the left and the right by elements of~$K$ with the goal to obtain an element of~$A_+$.
To simplify the notation we will keep writing $g$ also for the matrix after multiplication.

We now multiply with $v_{3}(t)$ from the right to obtain
\[
e_{1}gv_{3}(t) = \left(w_{1},-\tfrac{t^{2}}{2}w_{1} + w_{2} - tw_{3} , tw_{1} + w_{3} , w_{4}\right)
\]
and choose $t \in \mathbb{Q}_{p}$ such that $tw_{1} + w_{3} = 0$. Note that $\left|t\right|_{p} \leq 1$, since $\left|w_{3}\right|_{p} \leq \left|w_{1}\right|_{p}$ by the first part of the proposition.
 Multiplying from the right with $v_{4}(t)$ for some 
$t \in \mathbb{Z}_{p}$, we may assume that $e_{1}g = (w_{1},w_{2},0,0)$. Note that the entry $w_{2}$ may have changed, but since $0 = Q(e_{1}) = Q(ge_{1}) = 2w_{1}w_{2}$, we now have that $w_{2} = 0$. 
 
 We would like to use a similar argument to simplify the first column of $g$. To do so, let $A_{Q}$ be the symmetric matrix corresponding to $Q$ and note  (by taking the inverse
 of the equation below) that
\begin{equation}\label{dualityeq}
	 g\in \on{SO}_{Q}\mbox{ if and only if }g A_{Q}g^T=A_{Q}\mbox{ if and only if } g^{T}\in\SO[\tilde{Q}],
\end{equation} 
where $\tilde{Q}$ corresponds to the quadratic form defined by $A_{Q}^{-1}$. Therefore, multiplying $g$ with $u_{3}(t), u_{4}(t)$ on the left corresponds to multiplying $g^{T}$ with $v_{3}(t), v_{4}(t)$ on the right and applying the above argument, we may assume that $g$
is of block form
\[
g = \left[\begin{matrix}w_{1} &0&0&0\\ 0&\\ 0&&\ast\\0\end{matrix}\right].
\]

Note that $e_{1}g = w_{1}e_{1}$ also implies that $e_{1}^{\perp}g = e_{1}^{\perp}$. Here, $e_{1}^{\perp}$ denotes the orthogonal complement of $e_{1}$ with respect to the inner product defined by $Q$, which is three dimensional and spanned by $e_{1}$, $e_{3}$ and $e_{4}$. But this implies that $e_{3}g$ and $e_{4}g$ are in $e_{1}^{\perp}$ as well and therefore, $ge_{2}^{T} = (0,x_{2},0,0)^{T}$. Using $g^{T}$ again, we see that also $e_{2}g = (0,x_{2},0,0)$, so we may assume that $g$ is of the form
\[
g = \left[\begin{matrix}w_{1} &0&0&0\\ 0&x_{2}&0&0\\0&0&y_{3}&y_{4}\\0&0&z_{3}&z_{4}\end{matrix}\right].
\]
We claim that $g_{3,4} = \left[\begin{smallmatrix}y_{3}&y_{4}\\z_{3}&z_{4}\end{smallmatrix}\right] \in \on{SO}_{z^{2}+\eta w^{2}}(\mathbb{Q}_{p}) = \on{SO}_{z^{2} + \eta w^{2}}(\mathbb{Z}_{p})$.
For otherwise we would multiply $e_3g$ or $e_4g$ with a positive power of $p$ and taking
the so obtained integral vector modulo $p$ we would 
again find an isotropic vector for $z^{2}+\bar{\eta} w^{2}$ over $\bF_p$. 
Hence we can multiply with one more element of $K$ from the left to obtain $\on{diag}(p^{-m},p^{m},1,1)\in A_+$ for some 
integer $m \geq 0$. 
\end{proof}

We note that since $\on{SO}(2,1)<\on{SO}_\eta(3,1)$ the first statement also applies to $g\in\on{SO}(2,1)(\bQ_p)$ (and $w_4=0$). Moreover, an analoguous statement of the Cartan decomposition of $\on{SO}(2,1)$ is well known to be true and follows in the same way.

We also state the Cartan decomposition of our model group in the split case.

\begin{proposition}\label{KAKrank2}
Let $H = \operatorname{SO}(2,2)(\bQ_{p})$, $K = \operatorname{SO}(2,2)(\bZ_{p})$ and
\[
A_+ = \left\{\operatorname{diag}(p^{-m},p^{m},p^{-n},p^{n}) : m \geq n \in \bZ_{\geq 0}\right\}.
\]
Then, every element $g \in H$ can be written 
as $g = k_{1}ak_{2}$ with $k_{1},k_{2} \in K$ 
and some uniquely determined $a \in A_+$.
\end{proposition}

\begin{proof}
For any $t \in \bQ_{p}$ we define the following unipotent elements in $H$
\begin{align*}
u_{3}\left(t\right) &= \begin{bmatrix}1&&&\\&1&&-t\\t&&1&\\&&&1\end{bmatrix} ~,~~v_{3}\left(t\right) = \begin{bmatrix}1&&-t&\\&1&&\\&&1&\\&t&&1\end{bmatrix}\\
u_{4}\left(t\right) &= \begin{bmatrix}1&&&\\&1&-t&\\&&1&\\t&&&1\end{bmatrix} ~,~~v_{4}\left(t\right) = \begin{bmatrix}1&&&-t\\&1&&\\&t&1&\\&&&1\end{bmatrix}.
\end{align*}
Also fix some element $g \in H$ and let $e_{1}g = (w_{1},w_{2},w_{3},w_{4})$.
Once more we will multiply $g$ from the left and from the right with elements from $K$ until we end up with an element in $A_{+}$.  First, notice that $\omega_{1}=\left[\begin{smallmatrix}0 & 1 & 0 & 0 \\1 & 0 & 0 & 0 \\0 & 0 & 0 & 1 \\0 & 0 & 1 & 0\end{smallmatrix}\right]$ and $\omega_{2} = \left[\begin{smallmatrix}0&0&1&0\\0&0&0&1\\1&0&0&0\\0&1&0&0\end{smallmatrix}\right]$ are in $K$. We may therefore assume that $\left|w_{1}\right|_{p}$ is maximal under the norms of all entries of $g$ by multiplying $g$ with $\omega_{1}$ and $\omega_{2}$ from the left and right as necessary.

Now we may argue exactly as in the proof of 
Proposition \ref{pRankOneKAK} using the above unipotent elements with $t\in\bZ_p$
to reduce $g$ to the form
\[
g = \begin{bmatrix}w_{1} & 0 & 0 & 0\\0 & y_{1} & 0 & 0\\0 & 0 & z_{1} & z_{2}\\0 & 0 & z_{3} & z_{4}\end{bmatrix},
\]
where $g_{3,4} = \left[\begin{smallmatrix}z_{1}&z_{2}\\z_{3}&z_{4}\end{smallmatrix}\right] \in \operatorname{O}_{2wz}(\bQ_{p})$. However, 
this shows that there are only two possibilities for $g_{3,4} $, namely
\[
g_{3,4} = \begin{bmatrix}\delta^{-1} & 0\\0 & \delta\end{bmatrix} ~~ \text{or} ~~ g_{3,4} = \begin{bmatrix} 0 & \delta \\ \delta^{-1} & 0\end{bmatrix}
\]
for  some $ \delta \in \bQ_{p}$.
Applying another element of $K$ (from the left or from the right) gives $g = \operatorname{diag}(p^{-m},p^{m},p^{-n},p^{n})$ for some $m,n \in \mathbb{Z}_{\geq 0}$.
Recall that $\|w_1\|=\|g\|=\|p^{-m}\|$ which also implies $n\leq m$.

For the uniqueness note that $\|g\|_p=\|a\|_p=p^m$ determines $m$ uniquely.
Similarly $\|\bigwedge^2g\|_p=\|\bigwedge^2 a\|_p=p^{{m+n}}$ uniquely
determines $m+n$ and so also $n$.
\end{proof}

\subsection{Coset Calculations}

We now prove the facts about counting left cosets of $K$ in the level sets $Ka_{p}^{\ell}K$ that were used to determine $m(B_{\ell})$ in Section \ref{quasisplitgeometry}.

\begin{lemma}
\label{KaKindex}
For any $a\in H$
the number of left cosets of $K$ in $KaK$ equals the number of left cosets of $aKa^{-1}\cap K$ in $K$.
\end{lemma}
\begin{proof}
Assume the coset decomposition $K=\bigsqcup_{j=1}^{n} h_{j}(aKa^{-1}\cap K)$. If now $g=h_{i}k\in K$ for some $i \in \left\{1,\dots,n\right\}$ and with $k\in aKa^{-1}\cap K$, then $$gaK=h_{i}kaK=h_{i}a(a^{-1}ka)K=h_{i}aK$$ so that $\left\{h_{i}aK : i = 1,\dots, n\right\}$ contains all possible $K$-cosets within $KaK$.
On the other hand, if $h_{j}a = h_{i}ak$ for some $k \in K$, then $h_{j} = h_{i}aka^{-1}$ which implies $i=j$ and the lemma.
\end{proof}

\begin{lemma}
\label{cosetdecomposition}
For $\on{SO}(2,1)(\bQ_{p})$ the number of left cosets of $(a_{p}Ka_{p}^{-1}\cap K)$ in $K$ is $p+1$ and for $\on{SO}_{\eta}(3,1)(\bQ_{p})$, the number of left cosets of $(a_{p}Ka_{p}^{-1}\cap K)$ in $K$ is $p^{2}+1$.
\end{lemma}

\begin{proof}
Once more we refrain from doing calculations for both groups as the ternary case follows from the same ideas. Hence we set $Q = 2xy + z^{2} + \eta w^{2}$ for a non-square $\eta \in \mathbb{Z}_{p}^{\times}$. We will prove the lemma by explicitely finding $p^{2}+1$ disjoint left cosets of $(a_{p}Ka_{p}^{-1} \cap K)$ in $K$ and showing that their union is all of $K$. Let $k\in K$ with first column vector equal to $(w_{1},w_{2},w_{3},w_{4})^{T}$. 
We define the following unipotent elements:
\[
u_{i,j}= \left[\begin{matrix}1 & 0 & 0 & 0 \\-\frac{i^2+ j^2\eta}{2} & 1 & -i & -j\eta \\i & 0 & 1 & 0 \\j & 0 & 0 & 1\end{matrix}\right] \in K.
\]
Moreover, we note that the elements of $(a_{p}Ka_{p}^{-1}\cap K)$ are precisely
the elements of the shape
\[
\left[\begin{matrix}* & * & * & * \\p^2* & * & p* & p* \\p* & * & * & * \\p* & * & * & *\end{matrix}\right]
\]
where we denote an entry by $p*$ (or $p^{2}*$) if it belongs to $p\bZ_{p}$ (or $p^{2}\bZ_{p}$).
We claim that if $w_{1}$ is invertible in $\mathbb{Z}_{p}$, then $u_{i,j}g$ is of that shape and thus in $(a_{p}Ka_{p}^{-1}\cap K)$ for some $i,j\in\set{0,\dots,p-1}$. 

Multiplying $u_{i,j}$ from the left corresponds to row operations and we pick $i$ and $j$ such that
\[
w_{1}i+w_{3}\in p\mathbb{Z}_{p} \mbox{ and } w_{1}j+w_{4}\in p\mathbb{Z}_{p}.
\]
By this choice the last two rows of $u_{i,j}k$ are of the right form and it remains to check that the second row vector of $u_{i,j}k$, say $v=(v_{1},v_{2},v_{3},v_{4})$, is of the form $(p^{2}*,*,p*,p*)$. Considering its first entry $v_{1}$ we use the following trick: 
The first standard vector $e_{1}$ is isotropic for the quadratic form
$\tilde{Q}(x,y,z,w)=2xy+z^2+\eta^{-1}w^2$ as in \eqref{dualityeq}, i.e. $\tilde{Q}(e_{1})=0$ and thus
\[
0=Q(e_1 (u_{i,j}ke_{1})^T)=2w_{1}v_{1}+(w_{1}i+w_{3})^{2}+\eta^{-1}(w_{1}j+w_{4})^{2}.
\]
But the last two summands are in $p^{2}\bZ_{p}$ and $w_{1}$ is invertible in $\mathbb{Z}_{p}$, so $v_{1}$ must be in $p^{2}\bZ_{p}$ as well.

To obtain that $v_{3},v_{4} \in p\mathbb{Z}_{p}$, we recall
that $ Q(e_{2})=0$ and so also
\[
 {Q}(e_{2}(u_{i,j}g ))=0=2v_{1}v_{2}+v_{3}^{2}+\eta v_{4}^{2}.
\]
The first summand has norm at most $p^{-2}$, and thus this is also a bound for the norm of $v_{3}^{2}+\eta^{-1}v_{4}^{2}$ . This implies, after reducing to $\bZ_{p}/p\bZ_{p}=\bF_{p}$ that the square $\overline{v}_{3}^{2}=v_{3}^{2}\;(p)$ equals the non-square $\overline{\eta}\overline{v}_{4}^{2}$, unless $\overline{v}_{3}=\overline{v}_{4}=0$ which proves our claim.

Now if $w_{1}$ is not invertible, we apply Proposition \ref{pRankOneKAK} to conclude that $w_{2}$ must be in $\bZ_{p}^{\times}$. Apply $\omega=\left[\begin{smallmatrix}0 & 1 & 0 & 0 \\1 & 0 & 0 & 0 \\0 & 0 & -1 & 0 \\0 & 0 & 0 & 1\end{smallmatrix}\right]\in K$ to $k$ from the left to essentially interchange the first and second rows. We again see that
\[
\tilde{Q}(e_{1}(\omega k)^T)=0=2w_{2}w_{2}+w_{3}^{2}+\eta^{-1} w_{4}^{2}
\]
with $w_{1}w_{2}\in p\bZ_{p}$ forces $w_{3},w_{4}\in p\bZ_{p}$ by reducing modulo $p$ and using that $\eta$ is a non-square. Using the same argument as above we now obtain
$\omega g\in (a_p Ka_p^{-1})\cap K$ as claimed. 

Finally, the $p^{2}+1$ elements $u_{i,j}$ and $\omega$ are all inequivalent to each other with respect to $a_{p}Ka_{p}^{-1}\cap K$, so they give a representative system for the coset decomposition of $(a_{p}Ka_{p}^{-1}\cap K)$ in $K$. In other words,
\[
K = \omega (a_{p}Ka_{p}^{-1}\cap K) \sqcup \bigsqcup_{i,j}u_{i,j}(a_{p}Ka_{p}^{-1}\cap K)
\]
and the lemma follows.
\end{proof}

\begin{lemma}
\label{higherlevelcosets}
For $\on{SO}(2,1)(\bQ_{p})$ it holds that
\[
[a_{p}^{\ell}Ka_{p}^{-\ell}\cap K:a_{p}^{\ell+1}Ka_{p}^{-(\ell+1)}\cap K]=p
\]
and for $\on{SO}_{\eta}(3,1)(\bQ_{p})$ one has
\[
[a_{p}^{\ell}Ka_{p}^{-\ell}\cap K:a_{p}^{\ell+1}Ka_{p}^{-(\ell+1)}\cap K]=p^{2}.
\]
\end{lemma}
Observe that the first row vector of some $g\in a_{p}^{\ell}Ka_{p}^{-\ell}\cap K$ is of the shape $(w_{1},p^{2\ell}*,p^{\ell}*,p^{\ell}*)$ which implies that $w_{1}$ must always be invertible and one always is in a case similar to the first case treated in the previous lemma -  whose proof we can follow essentially line by line.

\subsection{Regular Trees}\label{app-tree}
We define an incidence relation by setting $gK\sim hK$ to be neighbours if $d(gK,hK)=1$. This gives $H/K$ the structure of a $p^{2}+1$-regular tree on which $H$ acts transitively and neighbour preserving. Explicitely, the neighbours of $gK$ are $gu_{i,j}K$ and $g\omega K$, and if $g\notin K$ then $g\omega K$ is the unique neighbour of distance less than $d(eK,gK)$ to $eK$. 
We focus on $\on{SO}_\eta(3,1)(\bQ_{p})$ but this discussion easily implies the structure for $\on{SO}(2,1)(\bQ_p)$ as well.
Note that by combining the proofs of Lemma \ref{KaKindex} and Lemma \ref{cosetdecomposition}, the coset decomposition of $Ka_pK$ is given by $\omega K\bigsqcup_{i,j}u_{i,j}K$ with
\[
\omega=\mat{&p&&\\p^{-1}&&&\\&&1&\\&&&-1}\quad\text{ and }\quad u_{i,j}=\mat{p^{-1}&&&\\-\frac{i^{2} + \eta j^{2}}{2}p^{-1}&p&-i&-\eta j\\ip^{-1}&&1&\\jp^{-1}&&&1}
\]
for $i,j=0,\ldots, p-1$. Set $S=\left\{u_{i,j} : i,j \in \left\{0,\dots , p-1\right\}\right\}$ and denote by $S^{\ell}$ the set of words of length $\ell$, by which we simply mean that the elements $\underline{m}\in S^{\ell}$ are matrix products of the form $m_{1}\hdots m_{\ell}$ for $m_{i}\in S$.
\begin{lemma}
The words of length $\ell$ in $S^{\ell}$ together with words of the form $\omega\underline{m}$ for $\underline{m}\in S^{\ell-1}$ exhaust a representative system of the right-cosets of $K$ in $Ka^{\ell}K$.
\end{lemma}

\begin{proof}
We know by Lemma \ref{cosetdecomposition} that the number of $K$-cosets in $Ka^{\ell}K$ is given by $p^{2(\ell-1)}(p^{2}+1)$ which is also the number of words we have at our disposal. It suffices therefore to show that the corresponding cosets are all inequivalent.
\begin{claim}
Let $\underline{m},\underline{n} \in S^{\ell}$. Then $\left\|\underline{m}\right\|_p = p^{\ell}$ and $\underline{m}K = \underline{n}K$ implies that $\underline{m} = \underline{n}$.
\end{claim}
Let $\underline{m} = u_{i_{1},j_{1}}\dots u_{i_{\ell},j_{\ell}}$ and $\underline{n} = u_{a_{1},b_{1}}\dots u_{a_{\ell},b_{\ell}}$ and notice that $p^{\ell}\underline{m}$ and $p^{\ell}\underline{m}$ are integral. Then, $\underline{m}K = \underline{n}K$ implies that
\[
(pu_{i_{1},j_{1}})\dots (pu_{i_{\ell},j_{\ell}}) \mathbb{Z}_{p}^{4} = (pu_{a_{1},b_{1}})\dots (pu_{a_{\ell},b_{\ell}})\mathbb{Z}_{p}^{4} \mod p
\]
or equivalently, $\mathbb{F}_{p}(1,-\frac{i_{1}^{2} + \eta j_{1}^{2}}{2},i_{1},j_{1})^{T} = \mathbb{F}_{p}(1,-\frac{a_{1}^{2} + \eta b_{1}^{2}}{2},a_{1},b_{1})^{T}$ and therefore, $i_{1} = a_{1}$ and $j_{1} = b_{1}$. Inductively, we conclude that $\underline{m} = \underline{n}$. In particular, this argument also shows that the integral matrix $p^{\ell}\underline{m}$ is not divisible by $p$ and so $\left\|\underline{m}\right\| = p^{\ell}$.

Using the claim for $\ell-1$, we also see that $\omega \underline{m}K = \omega \underline{n}K$ implies $\underline{m} = \underline{n}$, where $\underline{m},\underline{n} \in S^{\ell -1}$.

\begin{claim}
Let $\underline{m} \in S^{\ell}$ and $\underline{n} \in S^{\ell-1}$. Then the cosets $\underline{m}K$ and $\omega \underline{n}K$ are disjoint.
\end{claim}
As before, let $\underline{m} = u_{i_{1},j_{1}}\dots u_{i_{\ell},j_{\ell}}$ and $\underline{n} = u_{a_{1},b_{1}} \dots u_{a_{\ell-1},b_{\ell-1}}$. Once again, $p^{\ell}\underline{m}$ and $p^{\ell}\omega \underline{n}$ are integral and therefore, $\underline{m}K = \omega \underline{n}K$ would imply
\[
(pu_{i_{1},j_{1}}) \dots (pu_{i_{\ell},j_{\ell}})\mathbb{Z}_{p}^{4} = (p\omega)(pu_{a_{1},b_{1}})\dots (pu_{a_{\ell-1},b_{\ell-1}})\mathbb{Z}_{p}^{4} \mod p
\]
or equivalently, $\mathbb{F}_{p}(1,-\frac{i_{1}^{2} + \eta j_{1}^{2}}{2},i_{1},j_{1})^{T} = \mathbb{F}_{p}(0,1,0,0)^{T}$, which is a contradiction.
\end{proof}

Note that the metric on $H/K$ defined in Section \ref{rankonevolume} satisfies $d(gK,K) = \ell$ if $g = \underline{m} \in S^{\ell}$ is a word of length $\ell$ as in the previous lemma. Moreover, we have the following
\begin{corollary}
Let $gK = \underline{m}K$ and $hK = \underline{n}K$, where $\underline{m} = m_{1}m_{2} \dots$ and $\underline{n} = n_{1}n_{2} \dots$ are words and denote by $|\underline{m}| = \log_{p}\left\|g\right\|_p$ and $|\underline{n}| = \log_{p} \left\|h\right\|$ the word lengths of $\underline{m}$ and $\underline{n}$. Set $j_{0} = \max \left\{j : n_{i} = m_{i} \mbox{ for all } i<j\right\}$. Then,
\[
d(gK,hK) = \left| n_{j_{0}}\dots n_{|\underline{n}|}\right| + \left|m_{j_{0}} \dots m_{|\underline{m}|} \right|.
\]
\end{corollary}

\bibliographystyle{alpha}
\bibliography{distributionofshapes}

\end{document}